\documentclass[letterpaper]{amsart}

\usepackage[letterpaper]{geometry}
\usepackage{amsmath, amsthm, amsfonts, amsbsy, thmtools, amssymb, color}
\usepackage{graphicx}

\usepackage{thm-restate}
\usepackage{hyperref}
\usepackage[]{ytableau}
\usepackage[mathscr]{euscript}
\usepackage{stmaryrd}
\usepackage[all, arc, curve, frame]{xy}
\usepackage[T1]{fontenc}
\usepackage{tikz}
\usetikzlibrary{arrows}

\usepackage{chngcntr}

\makeatletter
\def\namedlabel#1#2{\begingroup
	#2%
	\def\@currentlabel{#2}%
	\phantomsection\label{#1}\endgroup
}
\makeatother

\numberwithin{equation}{section}

\SetSymbolFont{stmry}{bold}{U}{stmry}{m}{n}

\DeclareMathOperator{\Id}{Id}

\DeclareMathOperator{\b|}{\boldsymbol{|}}

\DeclareMathOperator{\wt}{wt}

\usepackage{color}

\newcommand{\ubound}{B}

\newcommand{\C}{{\mathbb C}}

\newcommand{\cC}{{\mathcal C}}
\newcommand{\cH}{{\mathcal H}}

\newtheorem{thm}{Theorem}[section]
\newtheorem{prop}[thm]{Proposition}
\newtheorem{lem}[thm]{Lemma}

\newtheorem{Corollary}{Corollary}

\newtheorem{Theorem}{Theorem}
\newtheorem*{NoNumberTheorem}{Theorem}

\theoremstyle{remark}
\newtheorem{rem}[thm]{Remark}
\newtheorem{Remark}{Remark}

\theoremstyle{definition}
\newtheorem{definition}[thm]{Definition}

\title{Large Genus Asymptotics for Volumes of Strata of Abelian Differentials}

\author{Amol Aggarwal}

\begin{document}

\begin{abstract}
	
	In this paper we consider the large genus asymptotics for Masur-Veech volumes of arbitrary strata of Abelian differentials. Through a combinatorial analysis of an algorithm proposed in 2002 by Eskin-Okounkov to exactly evaluate these quantities, we show that the volume $\nu_1 \big( \mathcal{H}_1 (m) \big)$ of a stratum indexed by a partition $m = (m_1, m_2, \ldots , m_n)$ is $\big( 4 + o(1) \big) \prod_{i = 1}^n (m_i + 1)^{-1}$, as $2g - 2 = \sum_{i = 1}^n m_i$ tends to $\infty$. This confirms a prediction of Eskin-Zorich and generalizes some of the recent results of Chen-M\"{o}ller-Zagier and Sauvaget, who established these limiting statements in the special cases $m = 1^{2g - 2}$ and $m = (2g - 2)$, respectively.

	We also include an Appendix by Anton Zorich that uses our main result to deduce the large genus asymptotics for Siegel-Veech constants that count certain types of saddle connections.
\end{abstract}

\maketitle

\tableofcontents

\section{Introduction}

\label{Introduction}

\subsection{The Moduli Space of Abelian Differentials}

Fix a positive integer $g > 1$, and let $\mathcal{H} = \mathcal{H}_g$ denote the moduli space of pairs $(X, \omega)$, where $X$ is a Riemann surface of genus $g$ and $\omega$ is a holomorphic one-form on $X$. Equivalently, $\mathcal{H}$ is the total space of the Hodge bundle over the moduli space $\mathcal{M}_g$ of complex curves of genus $g$; $\mathcal{H}$ is typically referred to as the \emph{moduli space of Abelian differentials}.

For any $(X, \omega) \in \mathcal{H}$, the one-form $\omega$ has $2g - 2$ zeros (counted with multiplicity) on $X$. Thus, the moduli space of Abelian differentials can be decomposed as a disjoint union $\mathcal{H} = \bigcup_{m \in \mathbb{Y}_{2g - 2}} \mathcal{H} (m)$, where $m$ is ranged over all partitions\footnote{See Section \ref{Partitions} for our conventions and notation on partitions. In particular, $\mathbb{Y}_{2g - 2}$ denotes the set of partitions of size $2g - 2$.} of $2g - 2$, and $\mathcal{H} (m) \subset \mathcal{H}$ denotes the moduli space of pairs $(X, \omega)$ where $X$ is again a Riemann surface of genus $g$ and $\omega$ is a holomorphic differential on $X$ with $n$ distinct zeros of multiplicities $m_1, m_2, \ldots , m_n$. These spaces $\mathcal{H} (m)$ are (possibly disconnected \cite{CCMS}) orbifolds called \emph{strata}.

There is an action of the general linear group $\text{GL}_2 (\mathbb{R})$ on the moduli space $\mathcal{H}$ that preserves its strata $\mathcal{H} (m)$. This action is closely related to billiard flow on rational polygons \cite{RBF,TSOC,FS}; dynamics on translation surfaces \cite{RBF,TSOC,FS}; the theory of interval exchange maps \cite{SSCT,ETMF,RBF,MTSIEM,FS}; enumeration of square-tiled surfaces \cite{SSCT,FS}; and Teichm\"{u}ller geodesic flow \cite{ERG,FS}. We will not explain these topics further here and instead refer to the surveys \cite{RBF,TSOC,FS} for more information.

In any case, there exists an measure on $\mathcal{H}$ (or equivalently, on each stratum $\mathcal{H} (m)$) that is invariant with respect to the action of $\text{SL}_2 (\mathbb{R}) \subset \text{GL}_2 (\mathbb{R})$; this measure can be defined as follows. Let $m = (m_1, m_2, \ldots , m_n) \in \mathbb{Y}_{2g - 2}$, let $(X, \omega) \in \mathcal{H} (m)$ be a pair in the stratum corresponding to $m$, and define $k = 2g + n - 1$. Denote the zeros of $\omega$ by $p_1, p_2, \ldots , p_n \in X$, and let $\gamma_1, \gamma_2, \ldots , \gamma_k$ denote a basis of the relative homology group $H_1 \big( X, \{ p_1, p_2, \ldots , p_n \}, \mathbb{Z} \big)$. Consider the \emph{period map} $\Phi: \mathcal{H} (m) \rightarrow \mathbb{C}^k$ obtained by setting $\Phi (X, \omega) = \big( \int_{\gamma_1} \omega, \int_{\gamma_2} \omega, \ldots , \int_{\gamma_k} \omega \big)$. It can be shown that the period map $\Phi$ defines a local coordinate chart (called \emph{period coordinates}) for the stratum $\mathcal{H} (m)$. Pulling back the Lebesgue measure on $\mathbb{C}^k$ yields a measure $\nu$ on $\mathcal{H} (m)$, which is quickly verified to be independent of the basis $\{ \gamma_i \}$ and invariant under the action of $\text{SL}_2 (\mathbb{R})$ on $\mathcal{H} (m)$.

As stated, the volume $\nu \big( \mathcal{H} (m) \big)$ will be infinite since $(X, c \omega) \in \mathcal{H} (m)$ for any $(X, \omega) \in \mathcal{H} (m)$ and constant $c \in \mathbb{C}$. To remedy this issue, let $\mathcal{H}_1 (m) \subset \mathcal{H} (m)$ denote the moduli space of pairs $(X, \omega) \in \mathcal{H} (m)$ such that $\frac{\mathrm{i}}{2} \int_X \omega \wedge \overline{\omega} = 1$; this is the hypersurface of the stratum $\mathcal{H} (m)$ consisting of $(X, \omega)$ where $\omega$ has area one.

Let $\nu_1$ denote the measure induced by $\nu$ on $\mathcal{H}_1 (m)$. It was established independently by Masur \cite{ETMF} and Veech \cite{MTSIEM} that $\nu_1$ is ergodic on each connected component of $\mathcal{H}_1 (m)$ under the action of $\text{SL}_2 (\mathbb{R})$ and that the volume $\nu_1 \big( \mathcal{H}_1 (m) \big)$ is finite for each $m$. This volume $\nu_1 \big( \mathcal{H}_1 (m) \big)$ is called the \emph{Masur-Veech volume} of the stratum indexed by $m$.

\subsection{Explicit and Asymptotic Masur-Veech Volumes}

\label{Volumes}

Although the finiteness of the Masur-Veech volumes was established in 1982 \cite{ETMF,MTSIEM}, it was nearly two decades until mathematicians produced general ways of finding these volumes explicitly. One of the earlier exact evaluations of these volumes appeared in the paper \cite{SVMS} of Zorich (although he mentions that the idea had been independently suggested by Eskin-Masur and Kontsevich-Zorich two years earlier), in which he evaluates $\nu_1 \big( \mathcal{H}_1 (m) \big)$ for some partitions $m$ corresponding to small values of the genus $g$.

Through a different method, based on the representation theory of the symmetric group and asymptotic Hurwitz theory, Eskin-Okounkov \cite{ANBCTV} proposed a general algorithm that, given $g \in \mathbb{Z}_{> 1}$ and $m = (m_1, m_2, \ldots , m_n) \in \mathbb{Y}_{2g - 2}$, determines the volume of the stratum $\nu_1 \big( \mathcal{H}_1 (m) \big)$. Although this intricate algorithm did not lead to closed form identities, Eskin-Okounkov were able to use it to establish several striking properties of these volumes. For instance, they showed that $\nu_1 \big(\mathcal{H}_1 (m) \big) \in \pi^{2g} \mathbb{Q}$ for any $m \in \mathbb{Y}_{2g - 2}$, a fact that had earlier been predicted by Kontsevich-Zorich.

Once it is known that these volumes are finite and can in principle be determined, a question of interest is to understand how they behave as the genus $g$ tends to $\infty$. In the similar context of Weil-Petersson volumes, such questions were investigated at length by Mirzakhani-Zograf in \cite{GVRHSLG,LGAIMSC,LGAV}.

To that end, the algorithm of Eskin-Okounkov enabled Eskin to write a computer program to evaluate the volumes $\nu_1 \big( \mathcal{H}_1 (m) \big)$ for $m \in \mathbb{Y}_{2g - 2}$ such that $g \le 10$. Based on the numerical data provided by this program, Eskin and Zorich predicted in 2003 (although the conjecture was not published until over a decade later; see Conjecture 1 and equations (1) and (2) of \cite{VSDCLG}) that
\begin{flalign}
\label{estimatenu1}
\nu_1 \big( \mathcal{H}_1 (m) \big) = \displaystyle\frac{4}{\prod_{i = 1}^n (m_i + 1)}  \Bigg( 1 + \mathcal{O} \bigg( \displaystyle\frac{1}{g^{1 / 2}} \bigg) \Bigg),
\end{flalign}

\noindent uniformly in $g > 1$ and $m \in \mathbb{Y}_{2g - 2}$.

\begin{rem}
	
\label{errorpartitions}

Eskin and Zorich mention at the end of Section 2 of \cite{VSDCLG} that their data suggest that the error on the right side of \eqref{estimatenu1} should be smallest (over all $m \in \mathbb{Y}_{2g - 2}$) when $m = 1^{2g - 2}$ consists only of ones and largest when $m = (2g - 2)$.
	
\end{rem}

\begin{rem}
	
\label{rational}

It was observed as a curiosity in Remark 1 of \cite{VSDCLG} that the right side of \eqref{estimatenu1} is asymptotically a rational number, while for each $m$ the left side is a rational multiple of $\pi^{2g}$ (as mentioned above). Our method will see this as a consequence of the fact that the Riemann zeta function $\zeta (2g)$ is a rational multiple of $\pi^{2g}$ but converges to $1$ as $g$ tends to $\infty$.

\end{rem}

\begin{rem}
	
	 Theorem 2.12 of the recent work of Delecroix-Goujard-Zograf-Zorich \cite{SSCT} shows that \eqref{estimatenu1} implies (and is essentially implied by) asymptotics for the relative contribution of one-cylinder separatrix diagrams to the Masur-Veech volume of a stratum $\mathcal{H}_1 (m)$. This provides an alternative interpretation of \eqref{estimatenu1}.

\end{rem}

Before this work, the asymptotic \eqref{estimatenu1} had been verified in two cases. First, the work of Chen-M\"{o}ller-Zagier \cite{QLGL} established \eqref{estimatenu1} if $\mathcal{H} (m)$ is the \emph{principal stratum}, that is, when $m = 1^{2g - 2}$; this corresponds to the stratum in which all zeros of the holomorphic differential $\omega$ are distinct. By analyzing a generating function for the sequence $\big\{ \nu_1 \big( \mathcal{H}_1 (1^{2g - 2}) \big) \big\}_{g \ge 1}$, they show as Theorem 19.3 of \cite{QLGL} that
\begin{flalign}
\label{12g2estimate}
\nu_1 \big( \mathcal{H}_1 (1^{2g - 2}) \big) = 2^{4 - 2g} \Bigg( 1 - \displaystyle\frac{\pi^2}{24 g} + \mathcal{O} \bigg( \displaystyle\frac{1}{g^2} \bigg) \Bigg).
\end{flalign}

Second, the work of Sauvaget \cite{VSI} established \eqref{estimatenu1} in the case of the \emph{minimal stratum} $m = (2g - 2)$, when $\omega$ has one zero with multiplicity $2g - 2$. Through an analysis of Hodge integrals on the moduli space of curves (based on his earlier work \cite{CCSD}), he shows as Theorem 1.9 of \cite{VSI} that
\begin{flalign}
\label{2g2estimate}
\nu_1 \big( \mathcal{H}_1 (2g - 2) \big) = \displaystyle\frac{4}{2 g - 1} \Bigg( 1 + \mathcal{O} \bigg( \displaystyle\frac{1}{g} \bigg) \Bigg).
\end{flalign}

\subsection{Results}

\label{AsymptoticEstimate}

In this paper we establish the asymptotic \eqref{estimatenu1} for all strata, as indicated by the following theorem.

\begin{thm}
	
	\label{volumeestimateasymptotic}
	
	Let $g > 1$ be a positive integer, and let $m = (m_1, m_2, \ldots , m_n)$ denote a partition of size $2g - 2$. Then,
	\begin{flalign}
	\label{volumeestimate1h1}
	 \displaystyle\frac{4}{\prod_{i = 1}^n (m_i + 1)}  \left( 1 - \displaystyle\frac{2^{2^{200}}}{g}  \right) \le \nu_1 \big( \mathcal{H}_1 (m) \big) \le \displaystyle\frac{4}{\prod_{i = 1}^n (m_i + 1)}  \left( 1 +  \displaystyle\frac{2^{2^{200}}}{g}  \right).
	\end{flalign}
	
	\end{thm}

\begin{rem}
	\label{error121}
	
	Observe that the error in \eqref{volumeestimate1h1} (which is of order $\frac{1}{g}$) is in fact smaller than what was predicted by \eqref{estimatenu1}. However, this is consistent with Remark \ref{errorpartitions} and \eqref{2g2estimate}. Indeed, the former states that the error should be largest when $m = (2g - 2)$, and the latter states that if $m = (2g - 2)$ then the error is of order $\frac{1}{g}$. Thus, one should expect the error to be of order $\frac{1}{g}$ for all $m$, as in \eqref{volumeestimate1h1}.
\end{rem}

The proof of Theorem \ref{volumeestimateasymptotic} (or in fact the equivalent Theorem \ref{volume} below) will appear in Section \ref{EstimateM} and Section \ref{VolumeProof}; we will very briefly discuss this proof (see Section \ref{VolumeEstimateInitial} for a slightly more detailed description) and describe the organization for the remainder of this paper in Section \ref{Outline}. However, before doing so, let us make a few additional comments about the conjectures in \cite{VSDCLG}.

Eskin-Zorich made a number of asymptotic predictions in addition to \eqref{estimatenu1}. In particular, they also have conjectures on the large genus asymptotics for the \emph{area Siegel-Veech constants} of the strata $\mathcal{H} (m)$. Although we will not carefully define it here, the area Siegel-Veech constant is a different numerical invariant of a stratum $\mathcal{H} (m)$ of Abelian differentials, and it can be directly equated with several quantities of geometric interest, such as asymptotics for the number of closed geodesics on a translation surface \cite{AFS} and the sum of the positive Lyapunov exponents of the Hodge bundle under the Teichm\"{u}ller geodesic flow \cite{ERG}. The previously mentioned results of Chen-M\"{o}ller-Zagier \cite{QLGL} and Sauvaget \cite{VSI} confirm the predictions of \cite{VSDCLG} on the asymptotics for these constants (in addition to \eqref{estimatenu1}) for the principal stratum and the minimal stratum, respectively.

We have not attempted to see whether our methods can be applied to establish these predictions on the area Siegel-Veech constants in full generality, but let us recall that the work of Eskin-Masur-Zorich \cite{PBC} provides identities that express Siegel-Veech constants of a given stratum in terms of the Masur-Veech volumes of (often different) strata. By combining these results with Theorem \ref{volumeestimateasymptotic},  the \hyperref[Constants]{Appendix} by Anton Zorich evaluates the large genus asymptotics for Siegel-Veech constants counting various types of saddle connections. It might be possible to also use Theorem \ref{volumeestimateasymptotic} to determine the large genus asymptotics for area Siegel-Veech constants of some families of strata, but we will not pursue this here.

\begin{rem}
	
Subsequent to the appearance of this paper, we in \cite{LGAC} established the Eskin-Zorich prediction on area Siegel-Veech constants for connected strata of Abelian differentials. After this, Chen-M\"{o}ller-Sauvaget-Zagier \cite{VSCC} proposed an independent and very different algebro-geometric proof of both the volume asymptotic \eqref{volumeestimate1h1} and area Siegel-Veech constant asymptotic predicted in \cite{VSDCLG}. Later, using both combinatorial ideas from the present work and algebro-geometric methods from \cite{VSCC}, Sauvaget in \cite{LGAEV} proved an all-order genus expansion for the Masur-Veech volume of an arbitrary stratum. In \cite{TRV,VFGIN}, several predictions were proposed for asymptotics for Masur-Veech polynomials and volumes associated with strata of quadratic differentials under various limiting regimes. 
\end{rem}

\subsection{Outline}

\label{Outline}

The proof of Theorem \ref{volumeestimateasymptotic} is based on a combinatorial analysis of the original algorithm proposed by Eskin and Okounkov for evaluating $\nu_1 \big(\mathcal{H}_1 (m) \big)$ in \cite{ANBCTV}. However, as mentioned previously, this algorithm is intricate; it expresses the Masur-Veech volume through the composition of three identities, each of which involves a sum whose number of terms increases exponentially in the genus $g$. What we will show is that each of these sums is dominated by a single term, and the remaining (non-dominant) terms in the sum decay rapidly and can be viewed as negligible. However, instead of explaining this method in full generality immediately, it might be useful to see it implemented in a special case.

Therefore, after recalling some notation and combinatorial estimates in Section \ref{Estimates}, we will in Section \ref{StratumPrincipal} consider the case of the principal stratum, $m = 1^{2g - 2}$. In this setting, Eskin-Okounkov provide an explicit identity (see Lemma \ref{identityprincipal} below) for the volume $\nu_1 \big( \mathcal{H}_1 (m) \big)$. This identity will retain the complication of involving a large sum, but it will allow us to avoid having to implement the three-fold composition mentioned above. Thus, we will use Lemma \ref{identityprincipal} to obtain a quick proof of \eqref{12g2estimate} and, in so doing, hopefully provide some indication as to how one can estimate the types of large sums that will appear later in this paper.

Next, we will recall the Eskin-Okounkov algorithm in Section \ref{VolumeEvaluate} and explain how it can be used to provide a heuristic for Theorem \ref{volumeestimateasymptotic} in Section \ref{VolumeEstimateInitial}. The remaining Section \ref{EstimateM} and Section \ref{VolumeProof} will then be directed towards establishing the estimates required for the proof of Theorem \ref{volumeestimateasymptotic} (or rather its equivalent version Theorem \ref{volume}).

The \hyperref[Constants]{Appendix} by Anton Zorich then applies Theorem \ref{volumeestimateasymptotic} to evaluate the large genus asymptotics for certain classes of Siegel-Veech constants.

\subsection*{Acknowledgments}

The author heartily thanks Anton Zorich for many stimulating conversations, valuable encouragements, and enlightening explanations, and also for kindly offering to attach the \hyperref[Constants]{Appendix} to this paper. The author also would like to express his profound gratitude to Alexei Borodin, Dawei Chen, Eduard Duryev, Adrien Sauvaget, and Peter Smillie for helpful comments and discussions. Additionally, the author thanks the anonymous referee for detailed suggestions on the initial draft of this paper. This work was partially supported by the NSF Graduate Research Fellowship under grant numbers DGE1144152 and DMS-1664619.

\section{Miscellaneous Preliminaries}

\label{Estimates}

In this section we recall some notation and estimates that will be used throughout this paper. In particular, Section \ref{Partitions} will set some notation on partitions and set partitions, and Section \ref{Estimates1} will collect several estimates to be applied later.

\subsection{Notation}

\label{Partitions}

A \emph{partition} $\lambda = (\lambda_1, \lambda_2, \ldots , \lambda_k)$ is a finite, nondecreasing sequence of positive integers. The numbers $\lambda_1, \lambda_2, \ldots , \lambda_k$ are called the \emph{parts} of $\lambda$; the number of parts $\ell (\lambda) = k$ is called the \emph{length} of $\lambda$; and the sum of the parts $|\lambda| = \sum_{i = 1}^k \lambda_i$ is called the \emph{size} of $\lambda$. We will also require the (slightly nonstandard) notion of the weight of the partition, which is defined as follows.

\begin{definition}[{\cite[Definition 4.26]{ANBCTV}}]
	
	\label{partitionweight}
	
The \emph{weight} of $\lambda$ is defined to be $\wt (\lambda) = |\lambda| + \ell (\lambda)$.

\end{definition}

For each integer $n \ge 0$, let $\mathbb{Y}_n$ denote the set of partitions of size $n$, and let $\mathbb{Y}_n (k)$ denote the number of partitions of size $n$ and length $k$. Further let $\mathbb{Y} = \bigcup_{n \ge 0} \mathbb{Y}_n$ denote the set of all partitions. For each $i \ge 1$ and any $\lambda \in \mathbb{Y}$, let $M_i (\lambda)$ denote the multiplicity of $i$ in $\lambda$; stated alternatively, $M_i (\lambda)$ denotes the number of indices $j \in \big[1, \ell (\lambda) \big]$ such that $\lambda_j = i$.

Observe in particular that $\sum_{i = 1}^{\infty} M_i (\lambda) = \ell (\lambda)$. Furthermore, for any positive integers $n$ and $k$, we have that
\begin{flalign}
\label{mlambda}
\displaystyle\sum_{\lambda \in \mathbb{Y}_n (k)} \displaystyle\frac{k!}{\prod_{j = 1}^{\infty} M_j (\lambda)!} = \binom{n - 1}{k - 1},
\end{flalign}

\noindent since both sides of \eqref{mlambda} count the number of \emph{compositions} of $n$ of length $k$, that is, the number of (ordered) $k$-tuples $(j_1, j_2, \ldots , j_k)$ of positive integers that sum to $n$. We denote the set of compositions $j = (j_1, j_2, \ldots , j_k)$ of $k$-tuples of positive integers summing to $n$ by $\mathcal{C}_n (k)$. Also denote the set of \emph{nonnegative compositions} of some integer $n \ge 0$, that is, the set of (ordered) $k$-tuples $(j_1, j_2, \ldots , j_k)$ of nonnegative integers that sum to $n$, by $\mathcal{G}_n (k)$. Observe that
\begin{flalign}
\label{ynkestimate}
\big| \mathbb{Y}_n (k) \big| \le \big| \mathcal{C}_n (k) \big| = \binom{n - 1}{k - 1}; \qquad \big| \mathcal{G}_n (k) \big| = \binom{n + k - 1}{k - 1}.
\end{flalign}

In addition to discussing partitions, we will also consider set partitions. For any finite set $S$, a \emph{set partition} $\alpha = \big( \alpha^{(1)}, \alpha^{(2)}, \ldots , \alpha^{(k)} \big)$ of $S$ is a sequence of disjoint subsets $\alpha^{(i)} \subseteq S$ such that $\bigcup_{i = 1}^k \alpha^{(i)} = S$; these subsets $\alpha^{(i)}$ are called the \emph{components} of $\alpha$. The \emph{length} $\ell (\alpha) = k$ of $\alpha$ denotes the number of components of $\alpha$.

Depending on the context, we may wish to (or not to) distinguish two set partitions consisting of the same components but in a different order. To that end, we have the definition below; in what follows, $\mathfrak{S} (k)$ denotes the symmetric group on $k$ elements.

\begin{definition}
	
\label{reducedpartitions}

Two set partitions $\alpha_1 = \big( \alpha_1^{(1)}, \alpha_1^{(2)}, \ldots , \alpha_1^{(k_1)} \big)$ and $\alpha_2 = \big( \alpha_2^{(1)}, \alpha_2^{(2)}, \ldots , \alpha_2^{(k_2)} \big)$ are \emph{equivalent as reduced set partitions} if $k_1 = k_2$ and there exists a permutation $\sigma \in \mathfrak{S} (k_1)$ such that $\alpha_1^{(i)} = \alpha_2^{\sigma(i)}$ for each $1 \le i \le k_1$. However, we will consider them inequivalent as \emph{nonreduced set partitions} unless $\sigma = \Id$. For instance, if $S = \{1, 2, 3, 4 \}$, then the set partitions $\big( \{1, 2 \}, \{ 3, 4 \} \big)$ and $\big( \{3, 4\}, \{1, 2 \} \big)$ are equivalent as reduced set partitions but not as nonreduced ones.

For any positive integers $n$ and $k$, let $\mathcal{P}_n$ denote the family of (equivalence classes of) reduced set partitions of $\{1, 2, \ldots , n \}$ and $\mathcal{P}_{n; k}$ denote the family of (equivalence classes of) reduced set partitions of $\{ 1, 2, \ldots , n \}$ of length $k$. Similarly, let $\mathfrak{P}_n$ denote the family of nonreduced set partitions of $\{1, 2, \ldots , n \}$ and $\mathfrak{P}_{n; k}$ denote the family of nonreduced set partitions of $\{ 1, 2, \ldots , n \}$ of length $k$.

Furthermore, for any set of positive integers $A = (A_1, A_2, \ldots, A_k)$ with $\sum_{i = 1}^k A_i = n$, let $\mathfrak{P} (A) = \mathfrak{P} (A_1, A_2, \ldots , A_k) = \mathfrak{P}_{n; k} (A_1, A_2, \ldots , A_k)$ denote the family of nonreduced set partitions $\alpha = \big(\alpha^{(1)}, \alpha^{(2)}, \cdots , \alpha^{(k)} \big)$ of $\{1, 2, \ldots , n \}$ such that $ \alpha^{(i)}$ has $A_i$ elements for each $1 \le i \le k$.
\end{definition}

Observe in particular that
\begin{flalign}
\label{aipksize}
k! \big| \mathcal{P}_{n, k} \big| = \big| \mathfrak{P}_{n, k} \big|; \qquad \big| \mathfrak{P} (A) \big| = \binom{n}{A_1, A_2, \ldots , A_k}; \qquad \mathfrak{P}_{n, k} = \bigcup_{A \in \mathcal{C}_n (k)} \mathfrak{P} (A).
\end{flalign}

We say that a reduced set partition  $\alpha_1 \in \mathcal{P}_n$ \emph{refines} $\alpha_2 \in \mathcal{P}_n$ if, for each $\alpha_1^{(i)} \in \alpha_1$, there exists some $\alpha_2^{(j)} \in \alpha_2$ such that $\alpha_1^{(i)} \subseteq \alpha_2^{(j)}$. Then there exists a partial order on $\mathcal{P}_n$ (and thus one on $\mathfrak{P}_n$) defined by stipulating that $\alpha_1 \le \alpha_2$ if $\alpha_1$ refines $\alpha_2$. This allows one to define the notion of complementary partitions, given as follows.

\begin{definition}[{\cite[Definition 6.2]{ANBCTV}}]

\label{partitionscomplement}	
	
	Two reduced set partitions $\alpha_1, \alpha_2 \in \mathcal{P}_n$ are \emph{complementary} if $\ell (\alpha_1) + \ell (\alpha_2) = n + 1$ and the minimal element of $\mathcal{P}_n$ greater than or equal to both $\alpha_1$ and $\alpha_2$ is the maximal set partition $( \{ 1, 2, \ldots , n \} )$. For any $\gamma \in \mathcal{P}_n$, let $\mathscr{C} (\gamma)$ denote the set of reduced set partitions $\alpha \in \mathcal{P}_n$ that are complementary to $\gamma$.
\end{definition}

For instance if $n = 5$, then $\big( \{ 1 \}, \{2 \}, \{ 3, 4, 5 \} \big)$ and $\big( \{ 1, 3 \}, \{ 2, 4 \}, \{ 5 \} \big)$ are complementary. However, $\big( \{ 1, 2 \}, \{ 3 \}, \{ 4, 5 \} \big)$ and $\big( \{ 1, 2, 3 \}, \{ 4 \}, \{ 5 \} \big)$ are not since they both refine $\big( \{ 1, 2, 3 \}, \{ 4, 5 \} \big)$.

The following lemma indicates that two complementary set partitions $\alpha_1$ and $\alpha_2$ are \textit{transverse}, in that any component of $\alpha_1$ can intersect any component of $\alpha_2$ at most once.

\begin{lem}
	
	\label{transverse}
	
	If $\alpha_1 \in \mathcal{P}_n$ and $\alpha_2 \in \mathscr{C} (\alpha_1)$, then $\big| \alpha_1^{(i)} \cap \alpha_2^{(j)} \big| \le 1$ for each $i, j$.
\end{lem}

\begin{proof}

 Denote $\alpha_1 = \big( \alpha_1^{(1)}, \alpha_1^{(2)}, \ldots , \alpha_1^{(r)} \big)$ and $\alpha_2 = \big( \alpha_2^{(1)}, \alpha_2^{(2)}, \ldots , \alpha_2^{(s)} \big)$, and assume to the contrary that there exist $i \in [1, r]$ and $j \in [1, s]$ such that $\big| \alpha_1^{(i)} \cap \alpha_2^{(j)} \big| \ge 2$. For notational convenience, let us set $i = j = 1$.

 We will define distinct indices $k_1, k_2, \ldots , k_s \in [1, s]$ inductively as follows. First, set $k_1 = 1$. Now, suppose we have selected $k_1, k_2, \ldots , k_{m - 1}$ for some integer $m \in [2, s]$. Let $\mathcal{R}_{m - 1}$ denote the set of indices $i \in [1, r]$ such that $\alpha_1^{(i)}$ is not disjoint with $\bigcup_{j = 1}^{m - 1} \alpha_2^{(k_j)}$; observe that $1 \in \mathcal{R}_1$.

 Since $( \{ 1, 2, \ldots , n \} )$ is the minimal reduced set partition that is refined by both $\alpha_1$ and $\alpha_2$, it follows that $\bigcup_{j = 1}^{m - 1} \alpha_2^{(k_j)} \subset \bigcup_{i \in \mathcal{R}_{m - 1}} \alpha_1^{(i)}$ (and $\bigcup_{i \in \mathcal{R}_{m - 1}} \alpha_1^{(i)} \ne \bigcup_{j = 1}^{m - 1} \alpha_2^{(k_j)}$). Therefore, there exists some index $k \in [1, s]$ distinct from $k_1, k_2, \ldots , k_{m - 1}$ such that $\alpha_2^{(k)}$ is not disjoint with $\bigcup_{i \in \mathcal{R}_{m - 1}} \alpha_1^{(i)}$; set $k_m$ equal to any such $k$.

 Now observe that $\big| \mathcal{R}_1 \big| \le \ell \big( \alpha_2^{(1)} \big) - 1$, since $\big| \alpha_1^{(1)} \big| \cap \big| \alpha_2^{(1)} \big| \ge 2$ and there can be at most $\ell \big( \alpha_2^{(1)} \big) - 2$ indices $i \ne 1$ such that $\alpha_1^{(i)}$ intersects $\alpha_2^{(1)}$. We also have that $\big| \mathcal{R}_m \big| \le \big| \mathcal{R}_{m - 1} \big| + \ell \big( \alpha_2^{(k_m)} \big) - 1$ for each $m \ge 2$, since $\big| \alpha_2^{(k_m)} \cap \bigcup_{i \in \mathcal{R}_{m - 1}} \alpha_1^{(i)} \big| \ge 1$. Together these estimates yield $|\mathcal{R}_s| = \ell (\alpha_1) \le \sum_{i = 1}^s \big( \ell \big( \alpha_2^{(i)} \big) - 1 \big) = n - \ell (\alpha_2)$, which contradicts the fact that $\ell (\alpha_1) = n + 1 - \ell (\alpha_2)$ as $\alpha_1$ and $\alpha_2$ are complementary.
\end{proof}

\subsection{Estimates}

\label{Estimates1}

In this section we collect several estimates that will be used at various points throughout this paper. In the below, for any integer $k > 1$, we denote by $\zeta (k) = \sum_{j = 1}^{\infty} j^{-k}$ the Riemann zeta function. Moreover, if $c$ is some constant and $k$ is some integer variable, then we write $ck!$ to denote $c \cdot k!$ (instead of $(ck)!)$. For instance, $2k! = 2 \cdot k! \ne (2k)!$ and $2 (k - 4)! = 2 \cdot (k - 4)! \ne \big(2 (k - 4) \big)! = (2k - 8)!$.

We will repeatedly use the bounds
\begin{flalign}
\label{2ll}
\begin{aligned}
k \le & 2^{k - 1};  \qquad \displaystyle\frac{2 k^{k + 1 / 2}}{e^k}  \le k! \le \displaystyle\frac{4 k^{k + 1 / 2}}{e^k}; \qquad \binom{n}{k} \le 2^n; \\
& \displaystyle\sum_{i = 1}^{k - 1} i! (k - i)! \le 4(k - 1)!; \qquad \big| \zeta (k) - 1 \big| \le \displaystyle\frac{4}{2^k},
\end{aligned}
\end{flalign}

\noindent which hold for any nonnegative integer $k$ (and for the last estimate in \eqref{2ll} we additionally assume that $k > 1$). The first three estimates in \eqref{2ll} are quickly verified; let us explain how to derive the fourth and fifth. The fourth follows from the fact that
\begin{flalign*}
\displaystyle\sum_{i = 1}^{k - 1} i! (k - i)! = 2 (k - 1)! + (k - 1)! \displaystyle\sum_{i = 2}^{k - 2} k \binom{k}{i}^{-1} \le (k - 1)! \Bigg( 2 + k (k - 3) \binom{k}{2}^{-1} \Bigg) \le 4 (k - 1)!,
\end{flalign*}

\noindent where in the second statement above we used the fact that $\min_{2 \le i \le k - 2} \binom{k}{i} = \binom{k}{2}$. To deduce the fifth bound in \eqref{2ll}, observe for $k \in \mathbb{Z}_{> 1}$ that
\begin{flalign*}
\big| \zeta (k) - 1 \big| = 2^{-k} \Bigg( 1 + \displaystyle\sum_{j = 3}^{\infty} \bigg( \displaystyle\frac{2}{j} \bigg)^k  \Bigg) \le 2^{-k} \Bigg( 1 + \sum_{j = 3}^{\infty} \bigg( \displaystyle\frac{2}{j} \bigg)^2 \Bigg) = 2^{-k} \big( 4 \zeta (2) - 4 \big) < \displaystyle\frac{4}{2^k}.
\end{flalign*}

Let us state a further (known) bound that we will also often use throughout this paper. If $n, r$ are positive integers and $\{ A_i \}$ and $\{ A_{i, j} \}$ for $1 \le i \le n$ and $1 \le j \le r$ are nonnegative integers such that $\sum_{j = 1}^r A_{i, j} = A_i$ for each $i$, then one can quickly verify the multinomial coefficient estimate
\begin{flalign}
\label{aijai}
\displaystyle\prod_{i = 1}^n \binom{A_i}{A_{i, 1}, A_{i, 2}, \ldots , A_{i, r}} \le \binom{\sum_{j = 1}^n A_i}{\sum_{i = 1}^n A_{i, 1}, \sum_{i = 1}^n A_{i, 2}, \ldots , \sum_{i = 1}^n A_{i, r}}.
\end{flalign}

Now we have the following lemma bounds products of factorials and will be used several times throughout the proof of Theorem \ref{volume}.

\begin{lem}
	
	\label{aicisum}
	
	Let $k \ge 1$ and $C_1, C_2, \ldots , C_k$ be nonnegative integers with $C_1 = \max_{1 \le i \le k} C_i$. Fix some integer $N$, and let $A_1, A_2, \ldots , A_k$ be nonnegative integers such that $\sum_{i = 1}^k A_i = N$. Then,
	\begin{flalign}
	\label{aici1}
	\displaystyle\prod_{i = 1}^k (A_i + C_i)! \le (N + C_1)! \displaystyle\prod_{i = 2}^k C_i!.
	\end{flalign}

	\noindent Moreover, if we stipulate that that $A_i$ are all nonnegative integers; that at least two of the $A_i$ are positive; that $C_1 = \max_{1 \le i \le k} C_i$; and that $C_2 = \max_{2 \le i \le k} C_i$, then we have that
	\begin{flalign}
	\label{2aici}
	\begin{aligned}
	& \displaystyle\prod_{i = 1}^k (A_i + C_i)! \le (N + C_1 - 1)! (C_2 + 1)! \displaystyle\prod_{i = 3}^k C_i!.
	\end{aligned}
	\end{flalign}

	\noindent Furthermore, if we impose $k \ge 2$ and that the $A_i$ are all even positive integers (meaning that $N$ is even) with at least two of them greater than or equal to four, then
	\begin{flalign}
	\label{aiproduct1}
	\displaystyle\prod_{i = 1}^k (2 A_i - 3)!! \le 15 (2N - 4k - 3)!!.
	\end{flalign}

\end{lem}

\begin{rem}
	
	Equality in each of the bounds \eqref{aici1}, \eqref{2aici}, and \eqref{aiproduct1} can be achieved when one of the terms in the product on the left side is as large as possible. Specifically, equality in \eqref{aici1} is obtained when $A_1 = N$ and $A_i = 0$ for each $i > 1$; equality in \eqref{2aici} is obtained when $A_1 = N - 1$, $A_2 = 1$, and $A_i = 0$ for each $i > 2$; and equality in \eqref{aiproduct1} is obtained when $A_1 = N - 2k$, $A_2 = 4$, and $A_i = 2$ for each $i > 2$.
\end{rem}

\begin{proof}[Proof of Lemma \ref{aicisum}]

The proofs of \eqref{aici1} and \eqref{2aici} are very similar, so let us omit the proof of \eqref{2aici}. To establish \eqref{aici1}, we induct on $k$, observing that the statement holds if $k = 1$. Thus, let $m \ge 2$ be a positive integer and suppose that the statement is valid whenever $k \le m - 1$.

Let $C_1, C_2, \ldots , C_m$ be nonnegative integers with $C_1 = \max_{1 \le i \le m} C_i$ and let $A_1, A_2, \ldots , A_m$ be nonnegative integers such that $\sum_{i = 1}^m A_i = N$. Then,
\begin{flalign}
\label{1aici1estimate2}
\displaystyle\prod_{i = 1}^m (A_i + C_i)! = (A_1 + C_1)! \displaystyle\prod_{i = 2}^n (A_i + C_i)! \le (A_1 + C_1)! ( N + C_2 - A_1)! \displaystyle\prod_{i = 3}^m C_i!.
\end{flalign}

\noindent Since $C_1 = \max_{1 \le i \le m} C_i$ and $A_1 \le N$, we have that
\begin{flalign*}
\displaystyle\frac{(A_1 + C_1)!}{C_2!} = \displaystyle\prod_{j = 1}^{A_1 + C_1 - C_2} (j + C_2) & \le \displaystyle\prod_{j = 1}^{A_1 + C_1 - C_2} (j + N + C_2 - A_1)  = \displaystyle\frac{(N + C_1)!}{(N + C_2 - A_1)!},
\end{flalign*}

\noindent and so we deduce \eqref{aici1} from \eqref{1aici1estimate2}.

To establish \eqref{aiproduct1}, we again induct on $k$. To verify the statement for $k = 2$, let $A = A_1 \ge A_2 = B \ge 4$, and observe that \eqref{aiproduct1} holds if $B = 4$. If instead $B \ge 5$, then
\begin{flalign*}
(2 A - 3)!! (2 B - 3)!! = \displaystyle\prod_{i = 1}^{A - 1} (2i - 1) \displaystyle\prod_{i = 1}^{B - 1} (2i - 1) & < \displaystyle\prod_{i = 1}^{A - 1} (2i - 1) \displaystyle\prod_{i = 1}^{B - 1} (2i - 1) \displaystyle\prod_{i = 4}^{B - 1} \displaystyle\frac{2 (i + A - 4) - 1}{2i - 1} \\
& = 15 \displaystyle\prod_{i = 1}^{A + B - 5} (2i - 1) = 15 (2 A + 2 B - 11)!!,
\end{flalign*}

\noindent where to deduce the second statement we used the fact that $2 (i + A - 4) - 1 > 2i - 1$ for each $i \in [1, B - 1]$ (as $A \ge B \ge 5$). Thus, \eqref{aiproduct1} holds.

 Now, let $m \ge 3$ be a positive integer and suppose that \eqref{aiproduct1} is valid whenever $k \le m - 1$. Let $A_1, A_2, \ldots , A_m$ be positive even integers such that $\sum_{i = 1}^m A_i = N$ and such that at least two of the $A_i$ are greater than or equal to four; assume that $A_1 \ge A_2 \ge \cdots \ge A_m$, so that $A_2 \ge 4$. Then, we have that
\begin{flalign*}
\displaystyle\prod_{i = 1}^m (2A_i - 3)!! & = (2A_m - 3)!! \displaystyle\prod_{i = 1}^{m - 1} (2A_i - 3)!! 	 \\
& \le 15 (2 A_m - 3)!! \big( 2 (N - A_m) - 4 (m - 1) - 3 \big)!! \\
& \le 15 (2 A_m - 3)!! \big( 2 (N - A_m) - 4 (m - 1) - 3 \big)!! \displaystyle\prod_{i = 2}^{A_m - 1} \displaystyle\frac{2 (N - A_m - 2m) + 2i - 1}{2i - 1} \\
& = 15 (2 N - 4m - 3)!!,
\end{flalign*}

\noindent where we have used the fact that $N \ge A_m + 2m$ (since $A_1 \ge A_2 \ge 4$ and each of the $m - 2$ other $A_i$ are all positive, even integers and thus are at least equal to two); this verifies \eqref{aiproduct1}.
\end{proof}

The following lemma estimates sums of products of factorials and will be used in the proofs of Proposition \ref{innermanyestimate} and Proposition \ref{innermanyestimate2}. In what follows, we recall the sets $\mathcal{C}_n (k)$ of compositions and $\mathcal{G}_n (k)$ of nonnegative compositions, as explained in Section \ref{Partitions}.

\begin{lem}  	 	
	
	\label{hestimate}
	
	Let $L \ge a \ge 0$ and $b \ge 0$ be integers with $L$ positive. For any composition $A = (A_1, A_2, \ldots , A_{L - a + 1}) \in \mathcal{C}_L (L - a + 1)$ and nonnegative composition $B = (B_1, B_2, \ldots , B_{L - a + 1}) \in \mathcal{G}_b (L - a + 1)$, let $\mathfrak{s} = \mathfrak{s} (A, B) \in [1, L - a + 1]$ denote the minimal index such that $A_{\mathfrak{s}} + 2 B_{\mathfrak{s}} = \max_{1 \le i \le L - a + 1} (A_i + 2 B_i)$, and if $L \ne a$ then let $\mathfrak{h} = \mathfrak{h} (A, B) \in [1, L - a + 1] \setminus \{ \mathfrak{s} \}$ denote the minimal index such that $A_{\mathfrak{h}} + 2 B_{\mathfrak{h}} = \max_{i \ne \mathfrak{s}} (A_i + 2 B_i)$. In particular, $\mathfrak{h}$ is an index such that $A_{\mathfrak{h}} + 2 B_{\mathfrak{h}}$ is second largest among all $A_i + 2 B_i$.
	
	Then,
	\begin{flalign}
	\label{habl}
	\displaystyle\sum_{A \in \mathcal{C}_L (L - a + 1)} \displaystyle\sum_{B \in \mathcal{G}_b (L - a + 1)} \displaystyle\prod_{i = 1}^{L - a + 1} \displaystyle\frac{ (A_i + 2 B_i)!}{A_i! B_i!} \le \displaystyle\frac{2^{9 L + 5} (a + 2b)!}{a! b!},
	\end{flalign}
	
	\noindent and if $L \ne a$ then
	\begin{flalign}
	\label{gabl}
	\displaystyle\sum_{A \in \mathcal{C}_L (L - a + 1)} \displaystyle\sum_{B \in \mathcal{G}_b (L - a + 1)} (A_{\mathfrak{h}} + 2 B_{\mathfrak{h}} + 1) \displaystyle\prod_{i = 1}^{L - a + 1} \displaystyle\frac{ (A_i + 2 B_i)!}{A_i! B_i!} \le \displaystyle\frac{2^{9 L + 5} (a + 2 b)!}{a! b!}.
	\end{flalign}

\end{lem}

\begin{proof}
	
	If $L = a$, then the left side of \eqref{habl} is equal to $\frac{(a + 2b)!}{a! b!}$ and so \eqref{habl} holds. Thus, we may assume that $L > a$, in which case \eqref{habl} would follow from \eqref{gabl}. It therefore suffices to establish \eqref{gabl}, to which end we set
	\begin{flalign*}
	g (a, b; L) =  \displaystyle\sum_{A \in \mathcal{C}_L (L - a + 1)} \displaystyle\sum_{B \in \mathcal{G}_b (L - a + 1)} (A_{\mathfrak{h}} + 2 B_{\mathfrak{h}} + 1) \displaystyle\prod_{i = 1}^{L - a + 1} \displaystyle\frac{ (A_i + 2 B_i)!}{A_i! B_i!}.
	\end{flalign*}

	First observe that, for any composition $A = (A_1, A_2, \ldots , A_{L - a + 1}) \in \mathcal{C}_L (L - a + 1)$ and nonnegative composition $B = (B_1, B_2, \ldots , B_{L - a + 1}) \in \mathcal{G}_b (L - a + 1)$, we have that
	\begin{flalign*}
	\displaystyle\prod_{i = 1}^{L - a + 1} \binom{A_i + 2 B_i}{A_i, B_i, B_i} = \displaystyle\prod_{i = 1}^{L - a + 1} \displaystyle\binom{A_i + 2 B_i - 1}{A_i - 1, B_i, B_i} \displaystyle\frac{A_i + 2 B_i}{A_i} & \le \binom{a + 2b - 1}{a - 1, b, b} \displaystyle\prod_{i = 1}^{L - a + 1} \left( \displaystyle\frac{2 B_i}{A_i} + 1 \right),
	\end{flalign*}
	
	\noindent where in the second statement we used \eqref{aijai} (with the $n$ there equal to $L - a + 1$ here; $r$ there equal to $3$ here; the $A_{i, 1}$ there equal to $A_i - 1$ here; and the $A_{i, 2} = A_{i, 3}$ there equal to $B_i$ here) together with the facts that $\sum_{i = 1}^{L - a + 1} (A_i - 1) = a - 1$ and $\sum_{i = 1}^{L - a + 1} B_i = b$. Hence, since $\frac{2B_i}{A_i} + 1 \le 2 (B_i + 1)$ and $a \le L \le 2^L$ (recall the first estimate in \eqref{2ll}), it follows that
	\begin{flalign}
	\label{aibi2}
	\displaystyle\prod_{i = 1}^{L - a + 1} \binom{A_i + 2 B_i}{A_i, B_i, B_i} & \le \displaystyle\frac{2^{2L} (a + 2b - 1)!}{a! b!^2} \displaystyle\prod_{i = 1}^L (B_i + 1).
	\end{flalign}
	
	 Additionally, since $A_{\mathfrak{h}} \le L - 1 \le 2^{L - 1}$, we have $A_{\mathfrak{h}} + 2 B_{\mathfrak{h}} + 1 \le 2^L (B_{\mathfrak{h}} + 1)$. Together with \eqref{aibi2} and the fact that $\big| \mathcal{C}_L (L - a + 1) \big| = \binom{L - 1}{L - a} \le 2^{L - 1}$ (recall the first identity in \eqref{ynkestimate}), this yields  	
	\begin{flalign}
	\label{gablestimate1}
	\begin{aligned}
	g (a, b; L) & \le \displaystyle\frac{2^{2L} (a + 2 b - 1)!}{a! b!^2} \displaystyle\sum_{A \in \mathcal{C}_L (L - a + 1)} \displaystyle\sum_{B \in \mathcal{G}_b (L - a + 1)}  (A_{\mathfrak{h}} + 2 B_{\mathfrak{h}} + 1) \displaystyle\prod_{i = 1}^{L - a + 1} (B_i + 1)! \\
	& \le \displaystyle\frac{2^{4L} (a + 2 b - 1)!}{a! b!^2} \displaystyle\max_{A\in \mathcal{C}_L (L - a + 1)} \displaystyle\sum_{B \in \mathcal{G}_b (L - a + 1)}  ( B_{\mathfrak{h}} + 1) \displaystyle\prod_{i = 1}^{L - a + 1} (B_i + 1)!.
	\end{aligned}
	\end{flalign}
	
	 To estimate the right side of \eqref{gablestimate1}, observe that any $B = (B_1, B_2, \ldots , B_{L - a + 1}) \in \mathcal{G}_b (L - a + 1)$ is uniquely determined an integer $s \in [0, L - a + 1]$; a subset $T = \{ t_1, t_2, \ldots , t_s \} \subseteq \{ 1, 2, \ldots , L - a + 1 \}$; and a composition $C = (C_1, C_2, \ldots , C_s) \in \mathcal{C}_b (s)$. Indeed, given such an $s$, $T$, and $C$, one produces $B$ by setting $B_{t_i} = C_i$ for each $i \in [1, s]$ and $B_j = 0$ for each $j \notin T$.
	
	 Therefore, instead of summing the right side of \eqref{gablestimate1} over $B$, we can sum it over all $s$, $T$, and $C$. Denote by $C_h$ by the second-largest element in $C$ (unless $s < 2$, in which case set $C_h = 0$), and observe that $B_{\mathfrak{h}} \le L + C_h \le 2^{L - 1} (C_h + 1)$ since $a \le L \le 2^{L - 1}$. Thus, since there are $\binom{L - a + 1}{s} \le 2^{L + 1}$ possibilities for $T$ and $L - a + 2 \le 2^{L + 1}$ possibilities for $s$, we find that
	\begin{flalign}
	\label{gablestimate2}
	\begin{aligned}
	g (a, b; L) &  \le \displaystyle\frac{2^{5L} (a + 2 b - 1)!}{a! b!^2} \displaystyle\sum_{s = 0}^{L - a + 1} \displaystyle\sum_{|T| = s} \displaystyle\sum_{C \in \mathcal{C}_b (s)}  (C_h + 2) \displaystyle\prod_{i = 1}^s (C_i + 1)! \\
	& \le \displaystyle\frac{2^{5L} (a + 2 b - 1)!}{a! b!^2} \displaystyle\sum_{s = 0}^{L - a + 1} \displaystyle\sum_{|T| = s} \displaystyle\sum_{C \in \mathcal{C}_b (s)}  (C_h + 2)! \displaystyle\prod_{\substack{1 \le i \le s \\ i \ne h}} (C_i + 1)! \\
	& \le \displaystyle\frac{2^{7 L + 2} (a + 2 b - 1)!}{a! b!^2} \displaystyle\max_{s \in [0, L - a + 1]} \displaystyle\sum_{C \in \mathcal{C}_b (s)} (C_h + 2)! \displaystyle\prod_{\substack{ 1 \le i \le s \\ i \ne h}} (C_i + 1)!.
	\end{aligned}
	\end{flalign}
	
	 Next, if the maximum of the right side of \eqref{gablestimate2} is taken at $s = 0$, then the right side is bounded by $\frac{2^{7L + 3} (a + 2b - 1)!}{a! b!^2}$, and so the lemma holds. Similarly, if it is taken at $s = 1$, then $C_h = 0$ and $C = (b)$, meaning that the quantity on the right side of \eqref{gablestimate2} is equal to $\frac{2^{7L + 3} (b + 1)! (a + 2b - 1)!}{a! b!^2} \le \frac{2^{7L + 3} (a + 2b)!}{a! b!}$, and the lemma again holds. Thus, we may assume that the maximum on the right side of \eqref{gablestimate2} is taken at $s \ge 2$, so that $C_h \ge 1$.

	Now we apply \eqref{aici1} with the $k$ there equal to $s - 1$; the $C_i$ there each equal to $2$; and the $A_i$ there equal to the $C_i - 1$ here (for $i \ne h$). Since $\sum_{i \ne h} (C_i - 1) = b - s - C_h + 1$ and since $b - s - C_h + 3 \le a + 2b$, this yields
	\begin{flalign}
	\label{gablestimate3}
	\begin{aligned}
	 g(a, b; L) & \le \displaystyle\frac{2^{7L + s + 1} (a + 2 b - 1)!}{a! b!^2} \displaystyle\max_{s \in [2, L - a  +1]} \displaystyle\sum_{C \in \mathcal{C}_b (s)} (C_h + 2)! (b - C_h - s + 3)! \\
	 & \le \displaystyle\frac{2^{8 L + 1} (a + 2 b)!}{a! b!^2} \displaystyle\max_{s \in [2, L - a  +1]} \displaystyle\sum_{C \in \mathcal{C}_b (s)} (C_h + 2)! (b - C_h - s + 2)!.
	 \end{aligned}
	 \end{flalign}
	
	Let us estimate the right side of \eqref{gablestimate3}. To that end, observe that there are $s \le L + 1 \le 2^L$ possibilities for $h \in [1, s]$ and that $C_h \in \big[ 1, \frac{b}{2} \big]$ (since $C_h$ denotes the second largest element of the composition $C$, which has total size $b$). Thus, relabeling $C_h = D$, summing over all possible $D$ and $h$, and using the fact that $\big| \mathcal{C}_{b - D} (s - 1) \big| = \binom{b - D - 1}{s - 2} \le \binom{b - D}{s - 2}$ (again due to \eqref{ynkestimate}) yields
	\begin{flalign}
	\label{gabl1}
	\begin{aligned}
	g(a, b; L) & \le \displaystyle\frac{2^{9L + 1} (a + 2 b)!}{a! b!^2} \displaystyle\max_{s \in [2, L - a + 1]} \displaystyle\sum_{D = 1}^{\lfloor b / 2 \rfloor} \big| \mathcal{C}_{b - D} (s - 1) \big| (b - s + 2 - D)! (D + 1)! \\
	& = \displaystyle\frac{2^{9L + 1} (a + 2 b)!}{a! b!^2} \displaystyle\max_{s \in [2, L - a + 1]} \displaystyle\sum_{D = 1}^{\lfloor b / 2 \rfloor } \binom{b - D - 1}{s - 2} (b - s + 2 - D)! (D + 2)! \\
	& \le  \displaystyle\frac{2^{9L + 1} (a + 2 b)!}{a! b!^2} \displaystyle\sum_{D = 1}^{\lfloor b / 2 \rfloor } (b - D)! (D + 2)!.
	\end{aligned}
	\end{flalign}

	\noindent To bound the right side of \eqref{gabl1}, observe for $b \le 3$ that $\sum_{D = 1}^{\lfloor b / 2 \rfloor } (b - D)! (D + 2)! \le 3b!$ and for $b \ge 4$ that
	\begin{flalign*}
	\displaystyle\sum_{D = 1}^{\lfloor b / 2 \rfloor} (b - D)! (D + 2)! \le \displaystyle\sum_{D = 1}^{b - 2} (b - D)! (D + 2)!  \le 6 (b - 1)! + 2 b! + (b + 2)! \displaystyle\sum_{D = 2}^{b - 3} \binom{b + 2}{D + 2}^{-1} \\ \le 4b! + (b + 2)! (b - 4) \binom{b + 2}{3}^{-1} \le 10 b!,
	\end{flalign*}
	
	\noindent where we have used the fact that $\min_{4 \le i \le b - 1} \binom{b + 2}{i} = \binom{b + 2}{b - 1} = \binom{b + 2}{3}$. Together with \eqref{gabl1}, this establishes the lemma.
\end{proof}

 The following lemma also bounds sums of products and will be used in Section \ref{EstimateC}.

\begin{lem}
	
	\label{sumsb}
	
	Let $n \le r$ and $k_1, k_2, \ldots , k_n$ be positive integers; denote $k = \sum_{i = 1}^n k_i$. Then,
	\begin{flalign*}
	\displaystyle\sum_{\ell \in \mathcal{C}_r (n)} \displaystyle\prod_{i = 1}^n \displaystyle\frac{k_i^{2\ell_i - 2}}{(2 \ell_i - 2)!} \le \displaystyle\frac{k^{2(r - n)}}{(2r - 2n)!}.
	\end{flalign*}

\end{lem}

\begin{proof}
	
	This follows from the fact that
	\begin{flalign*}
	(2r - 2n)! \displaystyle\sum_{\ell \in \mathcal{C}_r (n)} \displaystyle\prod_{i = 1}^n \displaystyle\frac{k_i^{2\ell_i - 2}}{(2 \ell_i - 2)!} & \le (2r - 2n)! \displaystyle\sum_{A \in \mathcal{G}_{2r - 2n} (n)} \displaystyle\prod_{i = 1}^n \displaystyle\frac{k_i^{A_i}}{A_i!} \\
	& = \displaystyle\sum_{A\in \mathcal{G}_{2r - 2n} (n)} \binom{2r - 2n}{A_1, A_2, \ldots , A_n} \displaystyle\prod_{i = 1}^n k_i^{A_i} = k^{2r - 2n},
	\end{flalign*}
	
	\noindent where in the last equality we used the fact that $\sum_{i = 1}^n k_i = k$.
\end{proof}

We conclude with the following lemma, which also will be used in Section \ref{EstimateC}, that estimates factorials.

\begin{lem}
		
\label{krnestimate}

Let $k$ and $a$ be positive integers with $k \ge 2 a$. Then,
\begin{flalign}
\label{kaestimate}
k^a \le \displaystyle\frac{2^{8a} (k - 1)!}{(k - 2a)!}; \qquad k^{2a} \le \displaystyle\frac{2^{8a} (k - 1)!}{(k - 2a - 1)!}, \quad \text{if $k \ge 2a + 1$.}
\end{flalign}

\end{lem}

\begin{proof}
	
We only establish the second estimate in \eqref{kaestimate}, since the proof of the first is very similar. First observe that if $k \ge 4a$ then $\frac{k}{k - i} \le 2$ for each $1 \le i \le 2a$, meaning that
\begin{flalign*}
k^{2a} \displaystyle\prod_{i = 1}^{2a} \displaystyle\frac{1}{k - i} \le 2^{2a},
\end{flalign*}

\noindent from which we deduce the second estimate in \eqref{kaestimate}. If $k \le 4a$, then since $\frac{k}{k - i} \ge 1$ for any $i \in [2a + 1, k - 1]$ and since $(k - 1)! \ge 2 \big( \frac{k}{e} \big)^{k - 1}$ (due to the second estimate of \eqref{2ll}), we have that
\begin{flalign*}
k^{2a} \displaystyle\prod_{i = 1}^{2a} \displaystyle\frac{1}{k - i} \le k^{k - 1} \displaystyle\prod_{i = 1}^{k - 1} \displaystyle\frac{1}{k - i} = \displaystyle\frac{k^{k - 1}}{(k - 1)!} \le \displaystyle\frac{e^{k - 1}}{2} \le 2^{2k - 1} \le  2^{8a},	
\end{flalign*}

\noindent from which we again deduce \eqref{kaestimate}.
\end{proof}

\section{Evaluating the Volumes}

The goal of this section is to explain several ways to explicitly evaluate the strata volumes $\nu_1 \big( \mathcal{H}_1 (m) \big)$ for various partitions $m$. We begin in Section \ref{StratumPrincipal} by using an identity of Eskin-Okounkov \cite{ANBCTV} to establish Theorem \ref{volumeestimateasymptotic} in the case of the principal stratum $m = 1^{2g - 2}$. Then, in Section \ref{VolumeEvaluate} we recall the general algorithm of Eskin-Okounkov \cite{ANBCTV} that finds the stratum volume $\nu_1 \big(\mathcal{H}_1 (m) \big)$, for any given $m = (m_1, m_2, \ldots , m_n)$. In Section \ref{VolumeEstimateInitial} we outline how to use this algorithm to establish Theorem \ref{volumeestimateasymptotic} (or in fact the equivalent Theorem \ref{volume}).

\subsection{The Principal Stratum}

\label{StratumPrincipal}

In this section we establish \eqref{volumeestimateasymptotic} when $m = 1^{2g - 2}$ is the principal stratum using an identity of Eskin-Okounkov \cite{ANBCTV} that provides an explicit expression for the volume $\nu_1 \big( \mathcal{H}_1 (1^{2g - 2}) \big)$. Following the notation in \cite{ANBCTV}, we will use and estimate a quantity $\textbf{c} (m)$ instead of the Masur-Veech volume $\nu_1 \big( \mathcal{H}_1 (m) \big)$. In view of Remark 2 of \cite{VSDCLG}, the two quantities are related by
\begin{flalign}
\label{cnuh}
\nu_1 \big( \mathcal{H}_1 (m) \big) = 2 \textbf{c} (m + 1),
\end{flalign}

\noindent where if $m = (m_1, m_2, \ldots , m_n)$ then $m + 1 = (m_1 + 1, m_2 + 1, \ldots , m_n + 1)$; we can take \eqref{cnuh} to be the definition of $\textbf{c} (m + 1)$.

The below Lemma \ref{identityprincipal}, which originally appeared as Theorem 7.1 of \cite{ANBCTV}, yields an identity for $\textbf{c} (2, 2, \ldots, 2)$ (for any even positive integer $n$) that will be asymptotically analyzed in Proposition \ref{principalstratumasymptotic}; this will imply \eqref{volumeestimateasymptotic} in the case of the principal stratum. In what follows, we define the quantity (which was originally given by Definition 6.6 of \cite{ANBCTV} and will also appear later)
\begin{flalign}
\label{zk}
\mathfrak{z} (k) = \big( 2 - 2^{2 - k} \big) \zeta (k) \textbf{1}_{k \in 2 \mathbb{Z}_{\ge 0}},
\end{flalign}

\noindent where $\zeta (k)$ denotes the Riemann zeta function and $\textbf{1}_E$ denotes the indicator for any event $E$.

\begin{lem}[{\cite[Theorem 7.1]{ANBCTV}}]

\label{identityprincipal}
	
For any even positive integer $n$, let $\kappa = \kappa_n$ denote the partition $2^n = (2, 2, \ldots , 2)$, in which $2$ appears $n$ times. Then, we have that
\begin{flalign*}
\textbf{\emph{c}} (\kappa_n) = n! \displaystyle\sum_{\substack{\mu \in \mathbb{Y}_{n + 2} \\ \mu \text{\emph{ Even}}}} \displaystyle\frac{(-1)^{\ell (\mu) - 1}}{ \big( 2n - \ell (\mu) + 2)! \prod_{i = 2}^{\infty} M_i (\mu)!}  \displaystyle\prod_{i = 1}^{\ell (\mu)} (2 \mu_i - 3)!!  \mathfrak{z}(\mu_i),
\end{flalign*}

\noindent where $\mu$ is summed over all partitions of $n + 2$ with only even parts and we recall from Section \ref{Partitions} that $M_i (\mu)$ denotes the multiplicity of $i$ in $\mu$.

\end{lem}

Using Lemma \ref{identityprincipal} we will establish the below proposition, which verifies Theorem \ref{volumeestimateasymptotic} in the special case of the principal stratum.

\begin{prop}

\label{principalstratumasymptotic}

For any positive integer $g > 1$, we have that
\begin{flalign}
\Big| 2^{2g - 4} \nu_1 \big( \mathcal{H}_1 (1^{2g - 2}) \big) - 1 \Big|\le \displaystyle\frac{2^{20}}{g}.
\end{flalign}
\end{prop}

\begin{proof}
	
Throughout this proof, set $n = 2g - 2$. Combining \eqref{cnuh} and Lemma \ref{identityprincipal}, we deduce that
\begin{flalign}
\label{principal1}
2^{n - 2} \nu_1 \big( \mathcal{H}_1 (1^n) \big) = 2^{n - 1} \textbf{c} (\kappa_n) = 2^{n - 1} n! \displaystyle\sum_{\substack{|\mu| \in \mathbb{Y}_{n + 2} \\ \mu \text{ Even}}} \displaystyle\frac{(-1)^{\ell (\mu) - 1}}{\big( 2n - \ell (\mu) + 2)!  \prod_{i = 2}^{\infty} M_i (\mu)!} \displaystyle\prod_{i = 1}^{\ell (\mu)} (2 \mu_i - 3)!!  \mathfrak{z}(\mu_i) .
\end{flalign}

Let us begin by removing the one-part partition $\mu = (n + 2)$ from the sum on the right side of \eqref{principal1}. To that end, observe that the contribution of the $\mu = (n + 2)$ term (which satisfies $\ell (\mu) = 1$) is equal to
\begin{flalign*}
\frac{2^{n - 1} n! (2n + 1)!!}{(2n + 1)!} \mathfrak{z} (n + 2) = \displaystyle\frac{\mathfrak{z} (n + 2)}{2} = (1 - 2^{- n - 1}) \zeta (n + 2)  \in \left[ 1 - \displaystyle\frac{1}{2^{n - 1}}, 1 + \displaystyle\frac{1}{2^{n - 1}} \right],
\end{flalign*}

\noindent where to deduce the last statement we used the last estimate of \eqref{2ll}.

Thus, it follows from \eqref{principal1} that
\begin{flalign}
\label{principal2}
\Big| 2^{n - 2} \nu_1 \big( \mathcal{H}_1 (1^n) \big) - 1 \Big| \le \displaystyle\frac{1}{2^{n - 1}} + 2^{n - 1} n! \displaystyle\sum_{\substack{|\mu| \in \mathbb{Y}_{n + 2} \\ \mu \text{ Even} \\ \ell (\mu) \ge 2}} \displaystyle\frac{1}{\big( 2n - \ell (\mu) + 2)!} \displaystyle\prod_{i = 1}^{\ell (\mu)} (2 \mu_i - 3)!!  \mathfrak{z}(\mu_i) .
\end{flalign}

It remains to show that the sum on the right side of \eqref{principal2} is $\mathcal{O} \big( \frac{1}{n} \big)$, to which end we will divide this sum into two contributions. Specifically, for each integer $r \ge 2$, let $\xi^{(r)} = \xi = \big( \xi_1, \xi_2, \ldots , \xi_r \big)$ denote the partition of length $r$ such that $\xi_1 = n + 4 - 2r$ and $\xi_2 = \xi_3 = \cdots = \xi_r = 2$. Set $\Xi =	 \big\{ \xi^{(2)}, \xi^{(3)}, \ldots , \xi^{(n / 2 + 1)} \big\}$. Furthermore, for each $r \ge 2$, let $\Omega (r) = \Omega_n (r)$ denote the set of partitions $\mu = (\mu_1, \mu_2, \ldots , \mu_r) \in \mathbb{Y}_{n + 2}$ such that $\ell (\mu) = r$; each $\mu_i$ is even, and $\mu \notin \Xi$. The last condition is equivalent to $\mu_2 \ge 4$ and implies that $\ell (\mu) \le \frac{n}{2}$.

Then \eqref{principal2} implies that
\begin{flalign}
\label{principal3}
\Big| 2^{n - 2} \nu_1 \big( \mathcal{H}_1 (1^n) \big) - 1 \Big| \le \displaystyle\frac{1}{2^{n - 1}} + \mathfrak{E}_1 + \mathfrak{E}_2,
\end{flalign}

\noindent where
\begin{flalign*}
\mathfrak{E}_1 & = 2^{n - 1} n! \displaystyle\sum_{\mu \in \Xi} \displaystyle\frac{1}{\big( 2n - \ell (\mu) + 2)!} \displaystyle\prod_{i = 1}^{\ell (\mu)} (2 \mu_i - 3)!!  \mathfrak{z}(\mu_i); \\
\mathfrak{E}_2 & = 2^{n - 1} n! \displaystyle\sum_{r = 2}^{n / 2} \displaystyle\sum_{\mu \in \Omega (r)} \displaystyle\frac{1}{\big( 2n - r + 2)!} \displaystyle\prod_{i = 1}^r (2 \mu_i - 3)!!  \mathfrak{z}(\mu_i) .
\end{flalign*}

\noindent Let us first estimate $\mathfrak{E}_1$. Since $\ell \big( \xi^{(r)} \big) = r$ for each integer $r \ge 2$ and $\mathfrak{z} (k) \le 4$ for each $k \ge 2$ (in view of the last estimate in \eqref{2ll}), we find that
\begin{flalign}
\label{e1estimateprincipal}
\begin{aligned}
\mathfrak{E}_1 & \le 2^{n - 1} n! \displaystyle\sum_{r = 2}^{n / 2 + 1} \displaystyle\frac{2^{2r} (2n - 4r + 5)!!}{(2n - r + 2)!} \\
 & = 2^{n - 1} n! \displaystyle\sum_{r = 2}^{n / 2 + 1} \displaystyle\frac{2^{2r} (2n - 4r + 5)!}{2^{n - 2r + 2} (n - 2r + 2)! (2n - r + 2)!} \\
& = n! \displaystyle\sum_{s = 1}^{n / 2} \displaystyle\frac{2^{4s + 1} (2n - 4s + 1)!}{ (n - 2s)! (2n - s + 1)!} \\
& = \displaystyle\sum_{s = 1}^{n / 2} 2^{4s + 1} \displaystyle\prod_{i = 0}^{2 s - 1} \displaystyle\frac{n - i}{2n - 2s + 1 - i} \displaystyle\prod_{i = 0}^{s - 1} \displaystyle\frac{1}{2n - s + 1 - i} 	 \le 2 \displaystyle\sum_{s = 1}^{n / 2} \left( \displaystyle\frac{16}{n} \right)^s \le \displaystyle\frac{2^{10}}{n},
\end{aligned}
\end{flalign}

\noindent where we set $s = r - 1$ and used the facts that $\frac{n - i}{2n - 2 s + 1 - i} \le 1$ for each $i \in [0, 2s - 1]$ and that $2n - s + 1 - i \ge n$ for each $i \in [0, s - 1]$ and $s \le \frac{n}{2}$.

Next we bound $\mathfrak{E}_2$. To do this, let us apply \eqref{aiproduct1} with $A_i = \mu_i$, $k = r$, and $N = n + 2$ to deduce that $\max_{\mu \in \Omega (r)} \prod_{i = 1}^r (2 \mu_i - 3)!! \le 15 (2n - 4r + 1)!!$ (since $\mu_1 \ge \mu_2 \ge 4$). Combining this with the fact that $\mathfrak{z}(k) \le 4$ (which follows from the last estimate in \eqref{2ll}, as above) yields
\begin{flalign*}
\begin{aligned}
\mathfrak{E}_2 & \le 2^{n - 1} n! \displaystyle\sum_{r = 2}^{n / 2} \displaystyle\sum_{\mu \in \Omega (r)} \displaystyle\frac{2^{2r}}{\big( 2n - r + 2)!}  \displaystyle\prod_{i = 1}^r (2 \mu_i - 3)!! \le 2^{n + 3} n! \displaystyle\sum_{r = 2}^{n / 2}  \displaystyle\frac{2^{2r} \big| \Omega (r) \big| (2n - 4r + 1)!!}{\big( 2n - r + 2)!}.
\end{aligned}
\end{flalign*}

Since $\big| \Omega (r) \big| \le \big| \mathbb{Y}_{n / 2} (r) \big| \le \binom{n / 2}{r - 1} \le \frac{1}{(r - 1)!} \big( \frac{n}{2} \big)^{r - 1}$ (here, we recall from Section \ref{Partitions} that $\mathbb{Y}_n (k)$ denotes the number of partitions of size $n$ and length $k$, and we are using the first estimate in \eqref{ynkestimate}), it follows that
\begin{flalign*}
\mathfrak{E}_2 & \le 2^{n + 5} n! \displaystyle\sum_{r = 2}^{n / 2} \displaystyle\frac{(2n)^{r - 1} (2n - 4r + 1)!!}{(r - 1)! ( 2n - r + 2)!} \\
& = 2^{n + 5} n! \displaystyle\sum_{r = 2}^{n / 2} \displaystyle\frac{(2n)^{r - 1} (2n - 4r + 1)!}{2^{n - 2r} (r - 1)! (n - 2r)! ( 2n - r + 2)!} \\
& = 128 n! \displaystyle\sum_{r = 2}^{n / 2}  \displaystyle\frac{8^{r - 1}}{ (r - 1)!} \displaystyle\frac{n^{r - 1}}{(n - 2r)!} \displaystyle\prod_{j = 0}^{3r} \displaystyle\frac{1}{2n - r + 2 - j}.
\end{flalign*}

\noindent Therefore, since $\frac{n! n^{r - 1}}{(n - 2r)!} \le n^{-2} \prod_{j = 0}^{3r} (n + r + 1 - j)$ (since $n + r + 1 - j \ge n$ for $j \in [0, r]$), we obtain
\begin{flalign}
\label{e2estimateprincipal}
\begin{aligned}
\mathfrak{E}_2 & \le \displaystyle\frac{128}{n^2} \displaystyle\sum_{r = 2}^{n / 2}  \displaystyle\frac{8^{r - 1}}{ (r - 1)!} \displaystyle\prod_{j = 0}^{3r} \displaystyle\frac{n + r + 1 - j}{2n - r + 2 - j} \le \displaystyle\frac{128}{n^2} \displaystyle\sum_{r = 2}^{n / 2}  \displaystyle\frac{8^{r - 1}}{ (r - 1)!} \le  \displaystyle\frac{128 e^8}{n^2} \le \displaystyle\frac{2^{19}}{n^2}.
\end{aligned}
\end{flalign}

\noindent where we used the facts that $\frac{n + r + 1 - j}{2n - r + 2 - j} \le 1$ for each $r \le \frac{n}{2}$ and that $e \le 2^{3 / 2}$ to deduce the fourth and fifth inequalities, respectively. Now the proposition follows from \eqref{principal3}, \eqref{e1estimateprincipal}, \eqref{e2estimateprincipal}, and the first estimate in \eqref{2ll}.
\end{proof}

The method used to establish Proposition \ref{principalstratumasymptotic} will be used many times in the proof of Theorem \ref{volumeestimateasymptotic}. Upon encountering a large sum, such as the one that appears on the right side of \eqref{principal1}, we will sometimes remove a leading order term that should in principle dominate the sum.\footnote{In some cases this will not be done, if our goal is bound the sum instead of approximate it.} This is analogous to the removal of the $\mu = (n + 2)$ term used to establish \eqref{principal2} from \eqref{principal1}.

It then will then remain to estimate the error, which is still a sum with many summands. In some cases, we will remove a few ``exceptional summands'' from this sum, whose contribution can be estimated directly (in proof above, these were the $\xi^{(r)}$), and then partition the remaining summands according to a certain statistic. In the proof above, this statistic was the length of the partition (although it will not always be in the future), which led to the partition $\bigcup_{r = 2}^{n / 2} \Omega (r)$ of the ``non-exceptional summands.'' We then bound the sum over each part using the largest possible value of a summand, and then sum over all parts to estimate the error.

\begin{rem}
	
	Through	 a similar procedure as used in the proof of Proposition \ref{principalstratumasymptotic}, it is also possible to obtain the second order correction to $\nu_1 \big( \mathcal{H}_1 (1^{2g - 2}) \big)$, as in the asymptotic \eqref{12g2estimate} of Chen-M\"{o}ller-Zagier \cite{QLGL}. Although we will not pursue a complete proof here, it can be shown that the second order correction in the sum on the right side of \eqref{principal2} occurs at $\mu = \xi^{(2)} = (n, 2)$. In this case, $\ell (\mu) = 2$ and this correction becomes
	\begin{flalign*}
	- 2^{n - 1} n! \displaystyle\frac{(2n - 3)!! \mathfrak{z} (2) \mathfrak{z} (n)}{(2n)!} = - \displaystyle\frac{2^n n! (1 - 2^{1 - n}) (2n - 3)! \zeta (2) \zeta (n)}{2^{n - 2} (n - 2)! (2n)!} \sim - \displaystyle\frac{\zeta (2)}{2n} \sim - \displaystyle\frac{\pi^2}{24g}, 	
	\end{flalign*}
	
	\noindent where we used the fact that $\zeta (2) = \frac{\pi^2}{6}$ and $n = 2g - 2$. This matches the second order correction $- \frac{\pi^2}{24 g}$ appearing on the right side of \eqref{12g2estimate}.
\end{rem}

\subsection{The Eskin-Okounkov Algorithm}

\label{VolumeEvaluate}

In this section we explain the algorithm of \cite{ANBCTV} that evaluates the quantity $\textbf{c} (m)$ for any $m$ with $m_1, m_2, \ldots , m_n \ge 2$. Recall from \eqref{cnuh} that $2 \textbf{c} (m) = \nu_1 \big(\mathcal{H}_1 (m - 1) \big)$ and thus that any Masur-Veech volume can be directly expressed in terms of such a quantity. As mentioned in Section \ref{Outline}, the algorithm that determines $\textbf{c} (m)$ essentially proceeds through the composition of three identities.

We begin with a countably infinite set of indeterminates $\{ p_1, p_2, \ldots \}$ and consider the algebra $\Lambda = \mathbb{C} [ p_1, p_2, \ldots ]$ that they generate.\footnote{In \cite{ANBCTV}, the indeterminates $\{ p_i \}$ are \emph{shifted power sums}, and $\Lambda = \Lambda^*$ is the algebra of \emph{shifted symmetric functions}. However, these facts will not be necessary for us to state the algorithm.} Two of the three identities will define a multilinear form  $\langle \cdot \b| \ldots \b|  \cdot \rangle: \Lambda^n \rightarrow \mathbb{C}$, the first of which will define the form on the subset of $\Lambda$ given by the vector space spanned by $p_1, p_2, \ldots $.

In particular, we have the definition below, which essentially appears as Theorem 6.7 of \cite{ANBCTV}. In what follows, we recall the notions of reduced set partitions (as explained in Section \ref{Partitions}) and the definition \eqref{zk} of $\mathfrak{z}_k$.

\begin{definition}
	
	\label{innerp}
	
	For any reduced partition $\alpha \in \mathcal{P}_n$, let $\Delta (\alpha) = \mathcal{G}_{\ell (\alpha) - 2} \big( \ell (\alpha) \big)$ denote the set of $\ell (\alpha)$-tuples of nonnegative integers $(d_1, d_2, \ldots , d_{\ell (\alpha)})$ such that $\sum_{i = 1}^{\ell (\alpha)} d_i = \ell (\alpha) - 2$.

	For any sequence of $n$ positive integers $m = (m_1, m_2, \ldots , m_n) \in \mathbb{Z}_{\ge 1}^n$, define
	\begin{flalign}
	\label{singleinnermany}
	\langle m \rangle = \langle m_1 \b| m_2 \b| \cdots \b| m_n \rangle = \langle p_{m_1} \b| p_{m_2} \b| \cdots \b| p_{m_n} \rangle = |m|! \mathfrak{z} \big( |m| - n + 2 \big) + \mathcal{E} (m),
	\end{flalign}
	
	\noindent where $|m| = \sum_{i = 1}^n m_i$, and $\mathcal{E} (m) = \mathcal{E} (m_1, m_2, \ldots , m_n)$ is given by
	\begin{flalign}
	\label{singleinnermany2}
	\mathcal{E} (m) & =  \displaystyle\sum_{\substack{\alpha \in \mathcal{P}_n \\ \ell (\alpha) \ge 2}} (-1)^{\ell (\alpha) - 1} \big( \ell (\alpha) - 2 \big)!  \displaystyle\sum_{d \in \Delta (\alpha)} \displaystyle\prod_{i = 1}^{\ell (\alpha)} \displaystyle\frac{1}{d_i!}  \big| m_{\alpha^{(i)}} \big|! \mathfrak{z} \left( \big| m_{\alpha^{(i)}} \big| - \big| \alpha^{(i)} \big| - d_i + 1 \right).
	\end{flalign}
	
	\noindent  In \eqref{singleinnermany}, we have denoted $\alpha = \big( \alpha^{(1)} \cup \alpha^{(2)} \cup \cdots \cup \alpha^{(\ell (\alpha))} \big)$; $\big| \alpha^{(i)} \big|$ by the number of elements in the component $\alpha^{(i)}$; and $\big| m_{\alpha^{(i)}} \big| = \sum_{j \in \alpha^{(i)}} m_j$. Observe that each summand on the right side of \eqref{singleinnermany2} is well-defined since it does not depend on the representative of the equivalence class of $\alpha \in \mathcal{P}_n$.

\end{definition}

\begin{rem}
	
	In Theorem 6.7 of \cite{ANBCTV}, the $\ell (\alpha)$-tuple $d = (d_1, d_2, \ldots , d_{\ell (\alpha)})$ was not summed over $\Delta$, but instead only those elements of $\Delta$ such that $\big| m_{\alpha^{(i)}} \big| - \big| \alpha^{(i)} \big| - d_j + 1$ is even. Due to the definition \eqref{zk} of $\mathfrak{z}$, one quickly verifies that it is only these elements of $\Delta$ that contribute to the right side of \eqref{singleinnermany2}.
\end{rem}

\begin{rem}
	
	For any $\alpha \in \mathcal{P}_n$, observe from the second identity in \eqref{ynkestimate} that
	\begin{flalign}
	\label{deltaestimate}
	\big| \Delta (\alpha) \big| = \binom{2 \ell (\alpha) - 3}{\ell (\alpha) - 1} \le 2^{2 \ell (\alpha)}.
	\end{flalign}
	
\end{rem}

Now we must extend the inner product partly defined in Definition \ref{innerp} to all of $\Lambda^k$, which will be done through the second identity, given by definition below that essentially appears as Theorem 6.3 of \cite{ANBCTV} (under the name of a ``Wick-type identity''). In what follows we recall the notion of complementary set partitions explained in Definition \ref{partitionscomplement}, and we let $p_{\lambda} = \prod_{i = 1}^{\ell (\lambda)} p_{\lambda_i}$ for any partition $\lambda = (\lambda_1, \lambda_2, \ldots , \lambda_{\ell (\lambda)}) \in \mathbb{Y}$; observe that the $\{ p_{\lambda} \}_{\lambda \in \mathbb{Y}}$ generate $\Lambda$ and thus it suffices to define the inner product on any family of $p_{\lambda}$.

\begin{definition}

\label{innermanyp}

Fix partitions $\lambda^{(1)}, \lambda^{(2)}, \ldots , \lambda^{(n)} \in \mathbb{Y}$; set $L_j = \sum_{i = 1}^j \ell \big( \lambda^{(i)} \big)$ for each $j \in [1, n]$; and denote $L_0 = 0$ and $L = L_n$. Let $\rho = \big( \rho^{(1)}, \rho^{(2)},  \cdots , \rho^{(n)} \big) \in \mathcal{P}_{L; n}$ denote the reduced partition of $\{ 1, 2, \ldots , L \}$ such that $\rho^{(i)} = \big\{ L_{i - 1} + 1, L_{i - 1} + 2, \ldots , L_i \big\}$ for each $i \in [1, n]$. Define
\begin{flalign}
\label{manyinnermany}
\big\langle p_{\lambda^{(1)}} \b| p_{\lambda^{(2)}} \b| \cdots \b| p_{\lambda^{(n)}} \big\rangle = \displaystyle\sum_{\alpha \in \mathscr{C} (\rho)} \displaystyle\prod_{i = 1}^{L - n + 1} \big\langle \lambda_{\alpha^{(i)}} \big\rangle,
\end{flalign}

\noindent where the sum is over all reduced set partitions $\alpha = \big( \alpha^{(1)}, \alpha^{(2)}, \ldots , \alpha^{(L - n + 1)} \big) \in \mathcal{P}_n$ that are complementary to $\rho$, and $\lambda_{\alpha^{(i)}} \subset \mathbb{Z}_{\ge 1} $ is a set of $\big| \alpha^{(i)} \big|$ integers defined as follows. We stipulate that a positive integer $u$ is in $\lambda_{\alpha^{(i)}}$ if and only if there exist $j \in [1, n]$ and $k \in \big[ 1, \ell \big(\lambda^{(j)} \big) \big]$ such that $u = \lambda_k^{(j)}$ and $L_{j - 1} + k \in \alpha^{(i)}$. Observe that the product on the right side of \eqref{manyinnermany} does not depend on the representative of the equivalence class of $\alpha$. Now, using \eqref{manyinnermany}, extend the inner product $\langle \cdot \b| \ldots \b|  \cdot \rangle$ by linearity to all of $\Lambda^n$.
\end{definition}

For instance, if $n = 3$, $\rho = \big( \{ 1, 2, 3 \}, \{ 4 \}, \{ 5, 6 \} \big)$, and $\alpha = \big( \{ 1, 4 \}, \{ 2, 6 \}, \{ 3 \}, \{ 5 \} \big)$, then $(L_0, L_1, L_2, L_3) = (0, 3, 4, 6)$ and
\begin{flalign*}
\lambda_{\alpha^{(1)}} = \big( \lambda_1^{(1)}, \lambda_1^{(2)} \big); \qquad \lambda_{\alpha^{(2)}} = \big( \lambda_2^{(1)}, \lambda_2^{(3)} \big); \qquad \lambda_{\alpha^{(3)}} = \big( \lambda_3^{(1)} \big); \qquad \lambda_{\alpha^{(4)}} = \big( \lambda_1^{(3)} \big).
\end{flalign*}

\noindent The corresponding summand in \eqref{manyinnermany} is then $\big\langle \lambda_1^{(1)} \b| \lambda_1^{(2)} \big\rangle  \big\langle \lambda_2^{(1)} \b| \lambda_2^{(3)} \big\rangle \big\langle \lambda_3^{(1)} \big\rangle \big\langle \lambda_1^{(3)} \big\rangle$.

The quantities $\textbf{c} (m)$ will not be directly expressed in terms of inner products of the $p_{\lambda}$, but rather in terms of inner products of a different family of elements of $\mathfrak{f}_k \in \Lambda$. The third identity, which appears as Theorem 5.5 of \cite{ANBCTV}, expresses these $\mathfrak{f}_k$ in the $\{ p_{\lambda} \}$ basis. In what follows, we recall the notion of the weight of a partition from Definition \ref{partitionweight}.

\begin{definition}
	
	\label{fdefinition}

For any integer $k \ge 2$, define the function $\mathfrak{f}_k$ through\footnote{In \cite{ANBCTV}, the functions $\mathfrak{f}_k$ denote the highest weight part in the expansion of certain (normalized) characters of the symmetric group in the shifted power sum basis $\{ p_{\lambda} \}$. However, this fact is again not required to state the algorithm.}
\begin{flalign}
\label{kf}
\mathfrak{f}_k = \displaystyle\frac{1}{k} \displaystyle\sum_{\wt (\lambda) = k + 1} \displaystyle\frac{(-k)^{\ell (\lambda) - 1}}{\prod_{i = 1}^{\infty} M_i (\lambda) !} p_{\lambda}.
\end{flalign}

\end{definition}

Using the above definitions, we can express $\textbf{c} (m)$ as an inner product through the following proposition, which follows from combining equation (1.8), Theorem 5.5, Definition 6.1, Theorem 6.3, and Theorem 6.7 of \cite{ANBCTV}.

\begin{prop}[{\cite{ANBCTV}}]
	
	\label{cmuinner}
	
	Let $m = (m_1, m_2, \ldots , m_n) \in \mathbb{Y}$ be a partition with $m_1, m_2, \ldots , m_n \ge 2$. Then,
	\begin{flalign*}
	\textbf{\emph{c}} (m) = \displaystyle\frac{1}{|m|!} \big\langle \mathfrak{f}_{m_1} \b| \mathfrak{f}_{m_2} \b| \cdots \b| \mathfrak{f}_{m_n} \big\rangle.
	\end{flalign*}
	
\end{prop}

The goal of the remainder of this article is to establish the following theorem, which in view of \eqref{cnuh} implies Theorem \ref{volumeestimateasymptotic}.

\begin{thm}

\label{volume}

Let $g > 1$ be an integer and $m = (m_1, m_2, \ldots , m_n) \in \mathbb{Y}_{2g + n - 2}$ be a partition such that $m_i \ge 2$ for each $i \in [1, n]$. If we denote $\mathcal{F}_k = k \mathfrak{f}_k$ for each $k \ge 2$, then
\begin{flalign}
\label{fvolume}
\Big| \big\langle \mathcal{F}_{m_1} \b| \mathcal{F}_{m_2} \b| \cdots \b| \mathcal{F}_{m_n} \big\rangle - 2 |m|! \Big| \le 2^{2^{200}} \big( |m| - 1 \big)!.
\end{flalign}

\noindent In particular, since $|m| = 2g + n - 2 \ge g$, we have that
\begin{flalign*}
\left|  \textbf{\emph{c}} (m) \prod_{i = 1}^n m_i  - 2 \right| \le \displaystyle\frac{2^{2^{200}}}{g}.
\end{flalign*}

\end{thm}

\subsection{Outline of the Proof of Theorem \ref{volume}}

\label{VolumeEstimateInitial}

Let us briefly indicate why one might expect the estimate \eqref{fvolume} to hold.
	
First observe, using the fact that $\mathcal{F}_k = k \mathfrak{f}_k$ and the definition \eqref{kf} of $\mathfrak{f}_k$, that the inner product $\big\langle \mathcal{F}_{m_1} \b| \mathcal{F}_{m_2} \b| \cdots \b| \mathcal{F}_{m_n} \big\rangle$ can be expressed as a linear combination of terms of the form $\big\langle p_{\lambda^{(1)}} \b| p_{\lambda^{(2)}} \b| \cdots p_{\lambda^{(n)}} \big\rangle$. One of these terms is $\big\langle p_{m_1} \b| p_{m_2} \b| \cdots \b| p_{m_n} \big\rangle$, which occurs when $\lambda^{(i)} = (m_i)$ for each $i$; it is quickly verified that this is the term corresponding to the maximal value of the total size $\sum_{i = 1}^n \big| \lambda^{(i)} \big|$.

We will establish that this term in fact dominates $\big\langle \mathcal{F}_{m_1} \b| \mathcal{F}_{m_2} \b| \cdots \b| \mathcal{F}_{m_n} \big\rangle$, that is,
\begin{flalign*}
\big\langle \mathcal{F}_{m_1} \b| \mathcal{F}_{m_2} \b| \cdots \b| \mathcal{F}_{m_n} \big\rangle \approx \big\langle p_{m_1} \b| p_{m_2} \b| \cdots \b| p_{m_n} \big\rangle = \langle m \rangle.
\end{flalign*}

 To analyze the latter expression, recall from \eqref{singleinnermany} that $\langle m \rangle = |m|! \mathfrak{z} \big( |m| - n + 2 \big) + \mathcal{E} (m)$, where $\mathcal{E}$ is defined by \eqref{singleinnermany2}. We will show that, if $m$ does not contain any parts equal to one (which is the case in the setting of Theorem \ref{volume}), then $\mathcal{E} (m)$ is of smaller order than $|m|!$. Therefore, $\langle m \rangle \approx |m|! \mathfrak{z} \big( |m| - n + 2 \big)$; since $\mathfrak{z} (k) = ( 2 - 2^{2 - k} ) \zeta (k) \textbf{1}_{k \in 2 \mathbb{Z}_{\ge 0}} \approx 2$ for $k$ large and even, this would show that $\big\langle \mathcal{F}_{m_1} \b| \mathcal{F}_{m_2} \b| \cdots \b| \mathcal{F}_{m_n} \big\rangle \approx \langle m \rangle \approx 2 |m|!$, as in Theorem \ref{volume}.\footnote{Observe that this heuristic does not use the multi-fold inner product given by \eqref{manyinnermany} (in the generic case when at least one of the $\lambda^{(i)}$ there has at least two parts). Indeed, this will be due to the fact that the sum of these terms will not contribute in the large $|m|$ limit.}

To fully justify this procedure will require some additional bounds. Specifically, we will begin in Section \ref{EstimateM} by estimating the inner products $\langle m \rangle$ for partitions $m = (m_1, m_2, \ldots , m_n)$. If each $m_i \ge 2$, then Lemma \ref{mestimatelarge} will verify the above statement that $\mathcal{E} (m) = \mathcal{O} \big( (|m| - 1)! \big)$. However, this will not quite suffice for our purposes. Indeed, although the partition $m$ in the statement of Theorem \ref{volume} has all parts at least two, it is possible that when we use \eqref{kf} to express $\mathcal{F}$ as a linear combination of the $p_{\lambda}$ that some of these $p_{\lambda}$ will have some parts equal to one.

Therefore, we will still be required to bound $\langle m \rangle$ in the case when some parts of $m$ are equal to one. In this case, we are in fact not certain if $\mathcal{E} (m) = \mathcal{O} \big( (|m| - 1)! \big)$ holds, but we will establish a weaker bound for this quantity as Proposition \ref{mestimate}, which will suffice for our purposes.

Next, we must bound the more general inner product given by \eqref{manyinnermany}. To gain an initial idea for how these bounds should look, one might first attempt to understand the contribution of any one summand to the sum on the right side of \eqref{manyinnermany}. For simplicity, let us suppose as above that the ideal approximation $\langle m_{\alpha^{(i)}} \rangle \sim 2 \big| m_{\alpha^{(i)}} \big|!$ holds. In this case, each summand on the right side of \eqref{manyinnermany} becomes approximately $2^{L - n + 1} \prod_{i = 1}^{L - n + 1} \big| m_{\alpha^{(i)}} \big|!$, which can be shown to be bounded by $2^{L - n + 1} \big( |m| - L + n \big)!$.

This heuristic holds for any individual summand in \eqref{manyinnermany}. However, if the terms defining the sum on the right side of \eqref{manyinnermany} decay sufficiently quickly, then one might expect it to in fact be possible to estimate the inner product on the left side of \eqref{manyinnermany} by $C^{L - n + 1} \big( |m| - L + n \big)!$ for some constant $C$. We will be able to establish such an estimate through a more careful analysis, as we will see in Proposition \ref{innermanyestimate} and its refinement Proposition \ref{innermanyestimate2} below.

Once the multi-fold inner products $\big\langle p_{\lambda^{(1)}} \b| p_{\lambda^{(2)}} \b| \cdots \b| p_{\lambda^{(n)}} \big\rangle$ have been appropriately estimated, we will be able to justify the approximation $\big\langle \mathcal{F}_{m_1} \b| \mathcal{F}_{m_2} \b| \cdots \b| \mathcal{F}_{m_n} \big\rangle \approx \big\langle p_{m_1} \b| p_{m_2} \b| \cdots \b| p_{m_n} \big\rangle $ and conclude the proof of Theorem \ref{volume} in Section \ref{EstimateC}.

\section{Estimating \texorpdfstring{$\langle m \rangle$}{}}

\label{EstimateM}

In this section we estimate $\mathcal{E} (m)$ as $|m|$ tends to $\infty$. Specifically, in Section \ref{EstimateM1} we bound this quantity in the case when $m$ has no parts equal to one, and in Section \ref{EstimateM2} we establish a weaker bound for this quantity when some parts of $m$ equal one.

\subsection{The Case When Each \texorpdfstring{$m_i \ge 2$}{}}	

\label{EstimateM1}

Our goal in this section is to establish the following lemma, which estimates $\mathcal{E} (m)$ when each part of $m$ is at least two.
	
\begin{lem}

\label{mestimatelarge}

Let $n$ be a positive integer, and let $m = (m_1, m_2, \ldots , m_n)$ be an $n$-tuple of integers with each $m_i \ge 2$. Denoting $|m| = \sum_{i = 1}^n m_i$, we have that $\big| \mathcal{E} (m) \big| \le 2^{40} \big( |m| - 1 \big)!$.

\end{lem}

\begin{proof}

Observe by the definition \eqref{singleinnermany2} of $\mathcal{E} (m)$, we have that
	\begin{flalign*}
	\big| \mathcal{E} (m) \big| & \le   \displaystyle\sum_{\substack{\alpha \in \mathcal{P}_n \\ \ell (\alpha) \ge 2}} \big( \ell (\alpha) - 2 \big)!  \displaystyle\sum_{d \in \Delta (\alpha)} \displaystyle\prod_{j = 1}^{\ell (\alpha)}  \big| m_{\alpha^{(j)}} \big|! \mathfrak{z} \left( \big| m_{\alpha^{(j)}} \big| - \big| \alpha^{(j)} \big| - d_j + 1 \right).
	\end{flalign*}
	
	\noindent Applying the fact that $\big| \mathfrak{z} (k) \big| \le 4$ and \eqref{deltaestimate}, we deduce that
	\begin{flalign*}
	\big| \mathcal{E} (m) \big| & \le   \displaystyle\sum_{\substack{\alpha \in \mathcal{P}_n \\ \ell (\alpha) \ge 2}} 2^{4 \ell (\alpha)} \big( \ell (\alpha) - 2 \big)! \displaystyle\prod_{j = 1}^{\ell (\alpha)}  \big| m_{\alpha^{(j)}} \big|!.
	\end{flalign*}
	
	\noindent Setting $\ell (\alpha) = r$ and applying the first and third identities in \eqref{aipksize} yields
	\begin{flalign}
	\label{emestimate1}
	\begin{aligned}
	\big| \mathcal{E} (m) \big| \le  \displaystyle\sum_{r = 2}^n \displaystyle\sum_{\alpha \in \mathcal{P}_{n; r}} 2^{4r} \big( r - 2 \big)! \displaystyle\prod_{j = 1}^r  \big| m_{\alpha^{(j)}} \big|! & =  \displaystyle\sum_{r = 2}^n \displaystyle\frac{2^{4r}}{r (r - 1)} \displaystyle\sum_{\alpha \in \mathfrak{P}_{n; r}} \displaystyle\prod_{j = 1}^{r}  \big| m_{\alpha^{(j)}} \big|! \\
	& = \displaystyle\sum_{r = 2}^n \displaystyle\frac{2^{4r}}{r (r - 1)} \displaystyle\sum_{\ell \in \mathcal{C}_n (r)} \displaystyle\sum_{\alpha \in \mathfrak{P} (\ell)} \displaystyle\prod_{j = 1}^{r}  \big| m_{\alpha^{(j)}} \big|!,
	\end{aligned}
	\end{flalign}
	
	 Now, for any composition $\ell \in \mathcal{C}_n (r)$, let $\mathfrak{s} = \mathfrak{s} (\ell)$ denote the minimal index $\mathfrak{s} \in [1, r]$ such that $\ell_{\mathfrak{s}}	 = \max_{1 \le j \le r} \ell_j$. Then, since $\sum_{i = 1}^r \big| m_{\alpha^{(i)}} \big| = |m|$; $\sum_{i = 1}^r \ell_i = n$; and $\big| m_{\alpha^{(i)}} \big| \ge 2 \ell_i$ (since $m_i \ge 2$ for each $i \in [1, r]$), \eqref{aici1} applied with $C_i = 2 \ell_i$ and $A_i = \big| m_{\alpha^{(i)}} \big| - 2 \ell_i$ yields
	\begin{flalign}
	\label{malphajproduct}
	\displaystyle\max_{\alpha \in \mathfrak{P} (\ell)} \displaystyle\prod_{j = 1}^{r}  \big| m_{\alpha^{(j)}} \big|! \le \big( |m| - 2n + 2 \ell_{\mathfrak{s}} \big)! \displaystyle\prod_{\substack{1 \le i \le r \\ i \ne \mathfrak{s}}} (2 \ell_i)!.
	\end{flalign}
	
	\noindent Applying the second identity in \eqref{aipksize} and \eqref{malphajproduct} in \eqref{emestimate1} yields
	\begin{flalign}
	\label{eestimate2}
	\begin{aligned}
	\big| \mathcal{E} (m) \big| & \quad \le \displaystyle\sum_{r = 2}^n \displaystyle\frac{2^{4r}}{r (r - 1)} \displaystyle\sum_{\ell \in \mathcal{C}_n (r)} \binom{n}{\ell_1, \ell_2, \ldots , \ell_r}  \big( |m| - 2n + 2 \ell_{\mathfrak{s}} \big)! \displaystyle\prod_{\substack{1 \le i \le r \\ i \ne \mathfrak{s}}} (2 \ell_i)! \\
	& \quad = \big| m \big|! \displaystyle\sum_{r = 2}^n \displaystyle\frac{2^{4 r}}{r (r - 1)} \displaystyle\sum_{\ell \in \mathcal{C}_n (r)} \displaystyle\frac{n!}{\ell_{\mathfrak{s}}!} \displaystyle\prod_{i = 0}^{2n - 2 \ell_{\mathfrak{s}} - 1} \displaystyle\frac{1}{|m| - i}  \displaystyle\prod_{\substack{1 \le i \le r \\ i \ne \mathfrak{s}}} \displaystyle\frac{(2 \ell_i)!}{\ell_i!}   \\
	& \quad \le 2n \big( |m| - 1 \big)! \displaystyle\sum_{r = 2}^n \displaystyle\frac{2^{4 r}}{r (r - 1)} \displaystyle\sum_{\ell \in \mathcal{C}_n (r)} \displaystyle\frac{n! \prod_{j = 1}^r \ell_i! }{(2n)!}  \displaystyle\prod_{i = 1}^r \binom{2 \ell_i}{\ell_i},
	\end{aligned}
	\end{flalign}
	
	\noindent where in the last estimate we used the fact that
	\begin{flalign*}
	\displaystyle\prod_{i = 0}^{2n - 2 \ell_{\mathfrak{s}} - 1} \displaystyle\frac{1}{|m| - i} \le \displaystyle\frac{2 n}{|m|} \displaystyle\prod_{i = 0}^{2n - 2 \ell_{\mathfrak{s}} - 1} \displaystyle\frac{1}{2n - i} = \displaystyle\frac{2n (2 \ell_{\mathfrak{s}})!}{m (2n)!}, \quad \text{since $|m| \ge 2n$.}
	\end{flalign*}
	
	\noindent Since $\prod_{i = 1}^r \binom{2 \ell_i}{\ell_i} \le \binom{2n}{n}$ (due to \eqref{aijai}), we deduce from \eqref{eestimate2} that
	\begin{flalign}
	\label{eestimatesum2}
	\begin{aligned}
	\big| \mathcal{E} (m) \big| & \le 2n \big( |m| - 1 \big)! \displaystyle\sum_{r = 2}^n 2^{4r} \displaystyle\sum_{\ell \in \mathcal{C}_n (r)} \displaystyle\frac{1}{n!} \displaystyle\prod_{j = 1}^r \ell_i!  \\
	& \le 2 n \big( |m| - 1 \big)! \displaystyle\sum_{r = 2}^n \displaystyle\frac{2^{4r} r (n - r + 1)!}{n!} + 2n \big( |m| - 1 \big)! \displaystyle\sum_{r = 2}^n 2^{4r}  \displaystyle\sum_{\substack{\ell \in \mathcal{C}_n (r) \\ \max_{i \ne \mathfrak{s}} \ell_i \ge 2}} \displaystyle\frac{\prod_{j = 1}^r \ell_i! }{n!},
	\end{aligned}
	\end{flalign}
	
	\noindent where the first sum on the right side of \eqref{eestimatesum2} corresponds to ``exceptional'' compositions $\ell \in \mathcal{C}_n (r)$ with one part equal to $n - r + 1$ and the remaining $r - 1$ parts equal to one, and the second sum corresponds to the remaining compositions (which must satisfy $\max_{i \ne \mathfrak{s}} \ell_i \ge 2$).
	
	Applying \eqref{eestimatesum2}, \eqref{ynkestimate}, and \eqref{2aici} (with $A_i = \ell_i - 1$ and $C_i = 1$) we obtain
	\begin{flalign}
	\label{2eestimatesum2}
	\begin{aligned}
	\big| \mathcal{E} (m) \big| & \le 2 n \big( |m| - 1 \big)! \displaystyle\sum_{r = 2}^n \displaystyle\frac{2^{4r} r (n - r + 1)!}{n!} + 2 n \big( |m| - 1 \big)! \displaystyle\sum_{r = 2}^n 2^{4r} \big| \mathcal{C}_n (r) \big| \displaystyle\max_{\substack{\ell \in \mathcal{C}_n (r) \\ \max_{i \ne \mathfrak{s}} \ell_i \ge 2}} \displaystyle\frac{\prod_{j = 1}^r \ell_i! }{n!} \\
	&  \le 512 \big( |m| - 1 \big)! \displaystyle\sum_{r = 2}^n \displaystyle\frac{16^{r - 2} r}{(r - 2)!} + 4 n \big( |m| - 1 \big)! \displaystyle\sum_{r = 2}^n \displaystyle\frac{16^r (n - r)!}{n!} \binom{n - 1}{r - 1} \le 2^{15} e^{16} \big(|m| - 1 \big)!,
	\end{aligned}
	\end{flalign}
	
	\noindent from which we deduce the lemma, since $e \le 2^{3 / 2}$.
\end{proof}

\subsection{The Case When \texorpdfstring{$m$}{} Has Parts Equal to \texorpdfstring{$1$}{}}

\label{EstimateM2}

Our goal in this section is to establish the following proposition that estimates $\big| \mathcal{E} (m) \big|$ when $m$ has some parts equal to one.

\begin{prop}
	
\label{mestimate}

Let $k \le n$ be positive integers, and let $m = (m_1, m_2, \ldots , m_n)$ be an $n$-tuple of positive integers with at most $k$ parts equal to $1$. Denoting $|m| = \sum_{i = 1}^n m_i$, we have that $\big| \mathcal{E} (m) \big| \le 2^{78 k} |m|!$.
\end{prop}

For the remainder of this section we will for notational convenience assume that $m_{n - k + 1} = m_{n - k + 2} = \cdots = m_n = 1$ and that $m_i \ge 2$ for each $1 \le i \le n - k$. We begin with the bound below.

In what follows, for any nonnegative integers $u, r \le n$, let $\mathfrak{V}_{n; r; u}$ denote the set of nonreduced set partitions $\alpha = \big( \alpha^{(1)}, \alpha^{(2)}, \ldots,  \alpha^{(r)} \big) \in \mathfrak{P}_{n; r}$ with the property that $\alpha^{(i)}$ contains at least one element in $\{ 1, 2, \ldots , n - u \}$ for each $i \in [1, r]$ or, equivalently, no $\alpha^{(i)}$ is a subset of $\{ n - u + 1, n - u + 2, \ldots ,  n \}$.

\begin{lem}
	
	\label{mestimate1st}
	
	Let $k \le n$ be positive integers, and let $m = (m_1, m_2, \ldots , m_n)$ be an $n$-tuple of positive integers with $m_{n - k + 1} = m_{n - k + 2} = \cdots = m_n = 1$ and $m_i \ge 2$ for each $1 \le i \le n - k$. Then,
		\begin{flalign}
	\label{12estimatee4}
	\begin{aligned}
	\big| \mathcal{E} (m) \big| & \le \displaystyle\sum_{r = 2}^n \displaystyle\sum_{s = 0}^r 2^{4 (r - s)} \binom{r}{s} \displaystyle\sum_{t = s}^k \displaystyle\frac{k!}{(k - t)!}  \binom{t - 1}{s - 1}	 \displaystyle\sum_{\alpha \in \mathfrak{V}_{n - t; r - s; k - t} } \displaystyle\prod_{j = 1}^{r - s} \big| m_{\alpha^{(j)}} \big|!.
	\end{aligned}
	\end{flalign}
\end{lem}

\begin{proof}
	
	Recalling the definition \eqref{singleinnermany2} of $\mathcal{E}$, applying the first identity in \eqref{aipksize}, and setting $r = \ell (\alpha)$ yields
	\begin{flalign}
	\label{12estimatee}
	\big| \mathcal{E} (m) \big| \le \displaystyle\sum_{r = 2}^n \displaystyle\frac{1}{r (r - 1)}  \displaystyle\sum_{\alpha \in \mathfrak{P}_{n; r} }  \displaystyle\sum_{d \in \Delta (\alpha)} \displaystyle\prod_{j = 1}^r  \big| m_{\alpha^{(j)}} \big|! \mathfrak{z} \left( \big| m_{\alpha^{(j)}} \big| - \big| \alpha^{(j)} \big| - d_j + 1 \right).
	\end{flalign}
	
	In order to analyze the right side of \eqref{12estimatee}, we will fix which components of $\alpha$ are subsets of $\{ n - k + 1, n - k + 2, \ldots , n \}$; this will correspond to understanding when $\big| m_{\alpha^{(k)}} \big| = \big| \alpha^{(k)} \big|$ (the cardinality of $\alpha^{(k)}$). To that end, let $s \le r$ and $t \le k$ be nonnegative integers; $s$ will denote the number of components in $\alpha$ that are contained in $\{ n - k + 1, n - k + 2, \ldots , n \}$, and $t$ will denote the total number of elements in these components. Also let $C = (C_1, C_2, \ldots , C_s) \in \mathcal{C}_t (s)$, and let $\mathcal{I} = (i_1, i_2, \ldots , i_s)$ denote an $s$-tuple of positive integers such that $1 \le i_1 < i_2 < \cdots < i_s \le r$. The sequence $\mathcal{I}$ will specify which $\alpha^{(i)}$ are contained in $\{ n - k + 1, n - k + 2, \ldots , n \}$, and the composition $C$ will specify how many elements each of these $\alpha^{(i)}$ has.
	
	Let $\mathfrak{K}_{n; r} (C; \mathcal{I})$ denote the family of nonreduced partitions $\alpha = \big( \alpha^{(1)}, \alpha^{(2)}, \ldots , \alpha^{(r)} \big) \in \mathfrak{P}_{n; r}$ such the following holds. First, for each $1 \le j \le s$, we have that $\alpha^{(i_j)} \subseteq \{ n - k + 1, n - k + 2, \ldots , n \}$ and $\big| \alpha^{(i_j)} \big| = C_j$. Second, for each $i \in \{ 1, 2, \ldots , r \} \setminus \{ i_1, i_2, \ldots , i_s \}$, the component $\alpha^{(i)}$ contains at least one element less than $n - k + 1$. Thus, $\mathfrak{K}_{n; r} (C; \mathcal{I})$ identifies which components of $\alpha$ are subsets of $\{ n - k + 1, n - k + 2, \ldots , n \}$ and also identifies how many elements they have.
	
	 Observe that
	\begin{flalign*}
	\mathfrak{P}_{n; r} = \bigcup_{s = 0}^r \bigcup_{t = s}^k \bigcup_{C \in \mathcal{C}_t (s)}  \bigcup_{|\mathcal{I}| = s} \mathfrak{K}_{n; r} (C; \mathcal{I}),
	\end{flalign*}
	
	\noindent which upon insertion into \eqref{12estimatee} yields
	\begin{flalign}
	\label{12estimatee2}
	\begin{aligned}
	\big| \mathcal{E} (m) \big| & \le \displaystyle\sum_{r = 2}^n \displaystyle\sum_{s = 0}^r \displaystyle\sum_{t = s}^k \displaystyle\sum_{C \in \mathcal{C}_t (s)} \displaystyle\sum_{|\mathcal{I}| = s} \displaystyle\sum_{\alpha \in \mathfrak{K}_{n; r} (C; \mathcal{I}) }  \displaystyle\sum_{d \in \Delta (\alpha)} \displaystyle\prod_{j = 1}^{r}  \big| m_{\alpha^{(j)}} \big|! \mathfrak{z} \left( \big| m_{\alpha^{(j)}} \big| - \big| \alpha^{(j)} \big| - d_j + 1 \right) \\
	& = \displaystyle\sum_{r = 2}^n \displaystyle\sum_{s = 0}^r \displaystyle\sum_{t = s}^k \displaystyle\sum_{C \in \mathcal{C}_t (s)} \displaystyle\sum_{|\mathcal{I}| = s} \displaystyle\sum_{\alpha \in \mathfrak{K}_{n; r} (C; \mathcal{I}) }  \displaystyle\sum_{d \in \Delta (\alpha)}  \displaystyle\prod_{i = 1}^s C_i! \displaystyle\prod_{i \in \mathcal{I}} \mathfrak{z} ( 1 - d_i ) \\
	& \qquad \qquad \qquad \qquad \qquad \qquad \qquad \qquad \quad \times \displaystyle\prod_{\substack{1 \le i \le r \\ i \notin \mathcal{I}}}  \big| m_{\alpha^{(i)}} \big|! \mathfrak{z} \left( \big| m_{\alpha^{(i)}} \big| - \big| \alpha^{(i)} \big| - d_i + 1 \right),
	\end{aligned}
	\end{flalign}
	
	\noindent where we have used the fact that $m_{n - k + 1} = m_{n - k + 2} = \cdots = m_n = 1$.

	To further estimate the right side of \eqref{12estimatee2}, first observe that the summand on the right side of \eqref{12estimatee2} does not depend on the choice of $\mathcal{I} \subseteq \{ 1, 2, \ldots , r \}$ with $|\mathcal{I}| = s$. Thus we can fix $\mathcal{I} = \mathcal{J} = \mathcal{J}_s = \{ 1, 2, \ldots , s \}$ and multiply the summand by $\binom{r}{s}$.
	
	Further observe that $\mathfrak{z} (1 - d_i) = \textbf{1}_{d_i = 1}$, since $\mathfrak{z} (k) = 0$ if $k$ is either odd or negative and $\mathfrak{z} (0) = 1$. Thus, let $\Delta_{\mathcal{J}} (\alpha) \subseteq \Delta (\alpha)$ denote the subset of $(d_1, d_2, \ldots , d_{\ell (\alpha)}) \in \Delta (\alpha)$ such that $d_i = 1$ for $i \in \mathcal{J}_s$.
	
	Inserting these two facts and the additional fact that $\big| \Delta_{\mathcal{J}} (\alpha)	 \big| = \binom{2 \ell (\alpha) - 2s - 3}{\ell (\alpha) - s - 2} \le 2^{2 \ell (\alpha) - 2 s} = 2^{2 (r - s)}$ (see \eqref{deltaestimate}) into \eqref{12estimatee2} yields
	\begin{flalign}
	\label{12estimatee3}
	\begin{aligned}
	\big| \mathcal{E} (m) \big| & \le \displaystyle\sum_{r = 2}^n \displaystyle\sum_{s = 0}^r 2^{4 (r - s)} \binom{r}{s} \displaystyle\sum_{t = s}^k \displaystyle\sum_{C \in \mathcal{C}_t (s)} \displaystyle\sum_{\alpha \in \mathfrak{K}_{n; r} (C; \mathcal{J}) }  \displaystyle\prod_{i = 1}^s C_i! \displaystyle\prod_{i = s + 1}^r \big| m_{\alpha^{(i)}} \big|!,
	\end{aligned}
	\end{flalign}
	
	\noindent where we used the bound (due to the last estimate in \eqref{2ll}) $\mathfrak{z} (k) \le 4$ when $i \notin \mathcal{J}$.
	
	To proceed, observe that any $\alpha \in \mathfrak{K}_{n; r} (C; \mathcal{J})$ can be identified as an ordered union $\alpha' \cup \mathcal{U}_1 \cup \mathcal{U}_2 \cup \cdots \cup \mathcal{U}_s$, where the $\mathcal{U}_i$ are disjoint subsets of $\{ n - k + 1, n - k + 2 , \ldots , n \}$ such that $\big| \mathcal{U}_i \big| = C_i$ for each $i \in [1, s]$, and $\alpha'$ is a nonreduced partition of $\{ 1, 2, \ldots , n \} \setminus \bigcup_{i = 1}^s \mathcal{U}_i$, none of whose components is a subset of $\{ n - k + 1, n - k + 2, \ldots , n \}$. Since the rightmost summand in \eqref{12estimatee3} does not depend on the explicit choice of the $\mathcal{U}_i$ satisfying these properties, we can fix some choice of the $\mathcal{U}_i$ and multiply the summand on the right side of \eqref{12estimatee3} by the number of such choices, which is $\binom{k}{k - t, C_1, C_2, \ldots , C_s}$. If we fix $\bigcup_{i = 1}^s \mathcal{U}_i = \{ n - t + 1, n - t + 2, \ldots , n \}$, then $\alpha'$ becomes a member of $\mathfrak{V}_{n - t; r - s, k - t}$.
	
	It follows upon insertion into \eqref{12estimatee3} that
	\begin{flalign}
	\label{12estimatee41}
	\begin{aligned}
	\big| \mathcal{E} (m) \big| & \le \displaystyle\sum_{r = 2}^n \displaystyle\sum_{s = 0}^r 2^{4 (r - s)} \binom{r}{s} \displaystyle\sum_{t = s}^k \displaystyle\sum_{C \in \mathcal{C}_t (s)} \displaystyle\sum_{\alpha \in \mathfrak{V}_{n - t; r - s; k - t} } \binom{k}{k - t, C_1, C_2, \ldots , C_s} \displaystyle\prod_{i = 1}^s C_i! \displaystyle\prod_{i = 1}^{r - s}\big| m_{\alpha^{(i)}} \big|!.
	\end{aligned}
	\end{flalign}
	
	Now the lemma follows from \eqref{12estimatee41}, the fact that the summand on the right side of \eqref{12estimatee41} does not depend on $C$, and the fact (from \eqref{ynkestimate}) that $\big| \mathcal{C}_s (t) \big|\le \binom{t - 1}{s - 1}$.
	\end{proof}

	Now we can establish Proposition \ref{mestimate} in a similar way to how we established Proposition \ref{mestimatelarge}.

	\begin{proof}[Proof of Proposition \ref{mestimate}]
		
	We will begin by rewriting the sum over $\alpha$ in \eqref{12estimatee4}. To that end, for any $\alpha = \big( \alpha^{(1)}, \alpha^{(2)}, \ldots , \alpha^{(r - s)} \big) \in \mathfrak{V}_{n - t; r - s; k - t}$, define $\beta = \beta (\alpha) = \big( \beta^{(1)}, \beta^{(2)}, \ldots , \beta^{(r - s)} \big) \in \mathfrak{P}_{n - k; r - s}$ by $\beta^{(i)} = \alpha^{(i)} \cap \{ 1, 2, \ldots ,  n - k \}$ for each $1 \le i \le r - s$; observe that no $\beta^{(i)}$ is empty since $\alpha \in \mathfrak{V}_{n - t; r - s; k - t}$.  Further define the (possibly empty) sets $\mathcal{T}_1, \mathcal{T}_2, \ldots , \mathcal{T}_{r - s}$ by $\mathcal{T}_i = \mathcal{T}_i (\alpha) = \alpha^{(i)} \cap \{ n - k + 1, n - k + 2, \ldots , n - t \}$; then the $\mathcal{T}_i$ are disjoint and satisfy $\bigcup_{i = 1}^{r - s} \mathcal{T}_i = \{n - k + 1, n - k + 2, \ldots , n - t \}$.
	
	Any $\alpha \in \mathfrak{V}_{n - t; r - s; k - t}$ can be uniquely recovered from $\beta (\alpha) \in \mathfrak{P}_{n - k; r - s}$ and disjoint family of sets $\mathcal{T} = (\mathcal{T}_1, \mathcal{T}_2, \ldots , \mathcal{T}_{r - s})$ such that $\bigcup_{i = 1}^{r - s} \mathcal{T}_i = \{ n - k + 1, n - k + 2, \ldots , n - t \}$. Thus, instead of having the sum on the right side of \eqref{12estimatee4} be over all $\alpha$ we can therefore take it over all $\beta$ and $\mathcal{T}$ satisfying the above conditions. More precisely, let $\mathfrak{T}_{r - s} (n; k; t)$ denote the family of all $(r - s)$-tuples of disjoint sets $\mathcal{T} = (\mathcal{T}_1, \mathcal{T}_2, \ldots , \mathcal{T}_{r - s})$ such that $\bigcup_{i = 1}^{r - s} \mathcal{T}_i = \{ n - k + 1, n - k + 2, \ldots , n - t \}$. We find from \eqref{12estimatee4} (and the fact that $\frac{k!}{(k - t)!} = t! \binom{k}{t}$) that
	\begin{flalign}
	\label{12estimatee5}
	\big| \mathcal{E} (m) \big| & \le \displaystyle\sum_{r = 2}^n \displaystyle\sum_{s = 0}^r 2^{4 (r - s)} \binom{r}{s} \displaystyle\sum_{t = s}^k t! \binom{k}{t}  \binom{t - 1}{s - 1}	 \displaystyle\sum_{\beta \in \mathfrak{P}_{n - k; r - s} } \displaystyle\sum_{\mathcal{T} \in \mathfrak{T}_{r - s} (n; k; t)} \displaystyle\prod_{i = 1}^{r - s} \Big( \big| m_{\beta^{(i)}} \big| + |\mathcal{T}_i| \big)!,
	\end{flalign}
	
	\noindent where we have used the fact that $\big| m_{\alpha^{(i)}} \big| = \big| m_{\beta^{(i)}} \big| + |\mathcal{T}_i|$ since $m_{n - k + 1} = m_{n - k + 2} = \cdots = m_{n - t} = 1$.
	
	Now observe that, for fixed $\beta$, the summand on the right side of \eqref{12estimatee5} does not depend on the explicit choice of $\mathcal{T}$ but only on the sizes $|\mathcal{T}_i|$. Thus, for any nonnegative composition $A = (A_1, A_2, \ldots , A_{r - s}) \in \mathcal{G}_{k - t} (r - s)$, let $\mathfrak{T} (A) = \mathfrak{T} (A; n; k)$ denote the set of all $\mathcal{T} \in \mathfrak{T}_{r - s} (n; k; t)$ such that $|\mathcal{T}_i| = A_i$ for each $1 \le i \le r - s$. Using the last identity in \eqref{aipksize} and the fact that $\mathfrak{T}_{r - s} (n; k; t) = \bigcup_{A \in \mathcal{G}_{k - t} (r - s)} \mathfrak{T} (A)$, we deduce from \eqref{12estimatee5} that
	\begin{flalign*}
	\big| \mathcal{E} (m) \big| & \le \displaystyle\sum_{r = 2}^n \displaystyle\sum_{s = 0}^r 2^{4 (r - s)} \binom{r}{s} \displaystyle\sum_{t = s}^k t! \binom{k}{t}  \binom{t - 1}{s - 1} \\
	& \qquad \qquad \times	\displaystyle\sum_{B \in \mathcal{C}_{n - k} (r - s)} \displaystyle\sum_{\beta \in \mathfrak{P} (B) } \displaystyle\sum_{A \in \mathcal{G}_{k - t} (r - s)} \displaystyle\sum_{\mathcal{T} \in \mathfrak{T} (A)} \displaystyle\prod_{i = 1}^{r - s} \Big( \big| m_{\beta^{(i)}} \big| + A_i \big)! \\
	& = \displaystyle\sum_{r = 2}^n \displaystyle\sum_{s = 0}^r 2^{4 (r - s)} \binom{r}{s} \displaystyle\sum_{t = s}^k t! \binom{k}{t}  \binom{t - 1}{s - 1} \\
	& \qquad \qquad \times	\displaystyle\sum_{B \in \mathcal{C}_{n - k} (r - s)} \displaystyle\sum_{\beta \in \mathfrak{P} (B) } \displaystyle\sum_{A \in \mathcal{G}_{k - t} (r - s)} \binom{k - t}{A_1, A_2, \ldots , A_{r - s}} \displaystyle\prod_{i = 1}^{r - s} \Big( \big| m_{\beta^{(i)}} \big| + A_i \Big)! \\
	& \le 2^k k! \displaystyle\sum_{r = 2}^n \displaystyle\sum_{s = 0}^r 2^{4 (r - s)} \binom{r}{s}  \displaystyle\sum_{t = s}^k \displaystyle\sum_{A \in \mathcal{G}_{k - t} (r - s)}  \displaystyle\sum_{B \in \mathcal{C}_{n - k} (r - s)} \displaystyle\sum_{\beta \in \mathfrak{P} (B)} \displaystyle\prod_{i = 1}^{r - s} \displaystyle\frac{\Big( \big| m_{\beta^{(i)}} \big| + A_i \Big)! }{A_i!},
	\end{flalign*}
	
	\noindent where we have used the fact that $\big| \mathfrak{T} (A) \big| = \binom{k - t}{A_1, A_2, \ldots , A_{r - s}}$ and the estimate $\binom{t - 1}{s - 1} \le 2^t \le 2^k$.
	
	It follows that
	\begin{flalign}
	\label{12estimatee51}
	\begin{aligned}
	\big| \mathcal{E} (m) \big| & \le 2^k k! \displaystyle\sum_{r = 2}^n \displaystyle\sum_{s = 0}^r 2^{4 (r - s)} \binom{r}{s}  \displaystyle\sum_{t = s}^k \displaystyle\sum_{A \in \mathcal{G}_{k - t} (r - s)} \displaystyle\sum_{B \in \mathcal{C}_{n - k} (r - s)} \binom{n - k}{B_1, B_2, \ldots , B_{r - s}} \\
	& \qquad \qquad \qquad \qquad \qquad \qquad \qquad \qquad \qquad \qquad \times \displaystyle\max_{\beta \in \mathfrak{P} (B)} \displaystyle\prod_{i = 1}^{r - s} \displaystyle\frac{\Big( \big| m_{\beta^{(i)}} \big| + A_i \Big)!}{A_i!},
	\end{aligned}
	\end{flalign}
	
	\noindent where we have used the fact that $\big| \mathfrak{P} (B) \big| = \binom{n - k}{B_1, B_2, \ldots , B_{r - s}}$ (recall the second identity in \eqref{aipksize}).
	
	Next we use \eqref{aici1} with their $A_i$ and $C_i$ equal to our $\big| m_{\beta^{(i)}} \big| - 2 B_i$ and $A_i + 2B_i$, respectively (which we may do since $m_i \ge 2$ for each $i \in [1, n - k]$). Setting $\mathfrak{h} = \mathfrak{h} (A, B) \in [1, r - s]$ to be the minimal index such that $A_{\mathfrak{h}} + 2 B_{\mathfrak{h}} = \max_{1 \le i \le r - s} (A_i + 2B_i)$, we deduce from \eqref{12estimatee51} that
	\begin{flalign}
	\label{12estimatee6}
	\begin{aligned}
	\big| \mathcal{E} (m) \big| & \le 2^k k! \displaystyle\sum_{r = 2}^n \displaystyle\sum_{s = 0}^r 2^{4 (r - s)} \binom{r}{s}  \displaystyle\sum_{t = s}^k \displaystyle\sum_{A \in \mathcal{G}_{k - t} (r - s)} \displaystyle\sum_{B \in \mathcal{C}_{n - k} (r - s)} \binom{n - k}{B_1, B_2, \ldots , B_{r - s}} \\
	& \qquad \qquad \qquad \times \big( |m| - k - 2 (n - k) + A_{\mathfrak{h}} + 2 B_{\mathfrak{h}} \big)! \displaystyle\prod_{i = 1}^{r - s} \displaystyle\frac{1}{A_i!} \displaystyle\prod_{\substack{1 \le i \le r - s \\ i \ne \mathfrak{h}}} ( A_i + 2 B_i )!,
	\end{aligned}
	\end{flalign}
	
	\noindent where we have used the fact that $\sum_{i = 1}^{r - s}  \big( \big|m_{\beta^{(i)}} \big| - 2 B_i \big) = |m| - k - 2 (n - k)$.
	
	Since $|m| - t \ge 2n - k - t \ge A_{\mathfrak{h}} + 2 B_{\mathfrak{h}}$, we have that
	\begin{flalign*}
	\big( |m| - k - 2 (n - k) + A_{\mathfrak{h}} + 2 B_{\mathfrak{h}} \big)! & = \big( |m| - t \big)! \displaystyle\prod_{i = 0}^{2n - k - t - A_{\mathfrak{h}} - 2 B_{\mathfrak{h}} - 1} \displaystyle\frac{1}{|m| - t - i} \\
	& \le \big( |m| - t \big)! \displaystyle\prod_{i = 0}^{2n - k - t - A_{\mathfrak{h}} - 2 B_{\mathfrak{h}} - 1} \displaystyle\frac{1}{2n - k - t - i} \\
	&  =  \displaystyle\frac{\big( |m| - t \big)! (A_{\mathfrak{h}} + 2 B_{\mathfrak{h}})!}{(2n - k - t)!},
	\end{flalign*}
	
	\noindent so it follows from \eqref{12estimatee6} that
	\begin{flalign*}
	\big| \mathcal{E} (m) \big| & \le 2^k k! \displaystyle\sum_{r = 2}^n \displaystyle\sum_{s = 0}^r 2^{4 (r - s)} \binom{r}{s} \displaystyle\sum_{t = s}^k  \displaystyle\frac{(n - k)! \big( |m| - t \big)!}{(2n - k - t)!}  \displaystyle\sum_{A \in \mathcal{G}_{k - t} (r - s)} \displaystyle\sum_{B \in \mathcal{C}_{n - k} (r - s)}\displaystyle\prod_{i = 1}^{r - s} \displaystyle\frac{( A_i + 2 B_i )!}{A_i! B_i!}.
	\end{flalign*}
	
	\noindent Using the fact that $k! = t! (k - t)! \binom{k}{t} \le 2^k t! (k - t)!$ for any $0 \le t \le k$, we deduce that
	\begin{flalign}
	\label{12estimatee7}
	\begin{aligned}
	\big| \mathcal{E} (m) \big| & \le 4^k \displaystyle\sum_{r = 2}^n \displaystyle\sum_{s = 0}^r 2^{4 (r - s)} \binom{r}{s} \displaystyle\sum_{t = s}^k t! \big( |m| - t \big)! \\
	& \qquad \qquad \times \displaystyle\sum_{A \in \mathcal{G}_{k - t} (r - s)} \displaystyle\sum_{B \in \mathcal{C}_{n - k} (r - s)} \displaystyle\frac{(k - t)! (n - k)!}{(2n - k - t)!} \displaystyle\prod_{i = 1}^{r - s} B_i! \displaystyle\prod_{i = 1}^{r - s} \binom{A_i + 2 B_i}{A_i, B_i, B_i} \\
	& \le 4^k \displaystyle\sum_{r = 2}^n \displaystyle\sum_{s = 0}^r 2^{4 (r - s)} \binom{r}{s} \displaystyle\sum_{t = s}^k t! \big( |m| - t \big)! \displaystyle\sum_{A \in \mathcal{G}_{k - t} (r - s)} \displaystyle\sum_{B \in \mathcal{C}_{n - k} (r - s)} \displaystyle\frac{1}{(n - k)!} \displaystyle\prod_{i = 1}^{r - s} B_i!,
	\end{aligned}
	\end{flalign}
	
	\noindent where we have used the fact that $\prod_{i = 1}^{r - s} \binom{A_i + 2B_i}{A_i, B_i, B_i} \le \binom{2n - k - t}{k - t, n - k, n - k}$, which holds due to \eqref{aijai}, since $\sum_{i = 1}^{r - s} A_i = k - t$ and $\sum_{i = 1}^{r - s} B_i = n - k$.
	
	Now, since $\max_{B \in \mathcal{C}_{n - k} (r - s)} \prod_{i = 1}^{r - s} B_i! \le (n - k - r + s + 1)!$ by \eqref{aici1} (applied with each $A_i$ equal to $B_i - 1$ and each $C_i$ equal to $1$); since $\big| \mathcal{C}_{n - k} (r - s) \big| = \binom{n - k - 1}{r - s - 1}$ (where if $r = s$ we replace this quantity by $1$) from the first statement of \eqref{ynkestimate}; and since $\big| \mathcal{G}_{k - t} (r - s) \big| = \binom{k - t + r - s - 1}{r - s - 1} \le 2^{k + r - s}$ from the second statement of \eqref{ynkestimate}, we deduce that
	\begin{flalign}
	\label{2k1nk2mestimate}
	\begin{aligned}
	\big| \mathcal{E} (m) \big|	&  \le 2^{3k} \displaystyle\sum_{r = 2}^n \displaystyle\sum_{s = 0}^r 2^{5(r - s)} \binom{r}{s} \displaystyle\sum_{t = s}^k t! \big( |m| - t \big)! \binom{n - k - 1}{r - s - 1} \displaystyle\frac{(n - k - r - s + 1)!}{ (n - k)!} \\
	&  \le 2^{3 k + 2} |m|! \displaystyle\sum_{s = 0}^k \displaystyle\sum_{r = s}^n  2^{5 (r - s)} \binom{r}{s} \displaystyle\frac{1}{ (r - s - 1)!} \\
	& = 2^{3 k + 2} |m|! (k + 1) + 2^{3 k + 2} |m|! \displaystyle\sum_{s = 0}^k \displaystyle\sum_{r = s + 1}^n  2^{5 (r - s)} \binom{r}{s} \displaystyle\frac{1}{ (r - s - 1)!} \\
	& \le 2^{4 k + 2} |m|! + 2^{3 k + 7} |m|! (\mathcal{E}_1 + \mathcal{E}_2),
	\end{aligned}
	\end{flalign}
	
	\noindent where we have used the facts that $\sum_{t = s}^{|m|} \big( |m| - t \big)! t! \le 4 m!$ and $k + 1 \le 2^k$ (which follow from the fourth and first estimates in \eqref{2ll}, respectively), and we have denoted
	\begin{flalign*}
	\mathcal{E}_1 =  \displaystyle\sum_{s = 0}^k \displaystyle\sum_{r = s + 1}^{3k}  2^{5 (r - s - 1)} \binom{r}{s} \displaystyle\frac{1}{ (r - s - 1)!}; \qquad \mathcal{E}_2 = \displaystyle\sum_{s = 0}^k \displaystyle\sum_{r = 3k + 1}^{\infty}  2^{5 (r - s - 1)} \binom{r}{s} \displaystyle\frac{1}{ (r - s - 1)!}.	
	\end{flalign*}
	
	\noindent Now, since $\binom{r}{s} \le 2^r$ and $k + 1 \le 2^k$ (recall the first bound in \eqref{2ll}), we have that
	\begin{flalign}
	\label{e1estimate}
	\mathcal{E}_1 \le 2^{3k}  \displaystyle\sum_{s = 0}^k \displaystyle\sum_{r = s + 1}^{3k}  2^{5 (r - s - 1)} \displaystyle\frac{1}{ (r - s - 1)!} \le 2^{3k}  \displaystyle\sum_{s = 0}^k \displaystyle\sum_{r = 0}^{\infty}  \displaystyle\frac{32^r}{r!} \le e^{32} (k + 1) 2^{3k} \le e^{32} 2^{4k} \le 2^{4k + 48}.
	\end{flalign}
	
	\noindent Furthermore, since $\binom{r}{s} \le \binom{r}{k} \le r^k$ for $s \le k \le \frac{r}{2}$ and $r^k \le 3^k \binom{r - k - 1}{k} k!$ for $r \ge 3k$, we have that
	\begin{flalign}
	\label{e2estimate}
	\begin{aligned}
	\mathcal{E}_2 & \le \displaystyle\sum_{s = 0}^k \displaystyle\sum_{r = 3k + 1}^{\infty}  2^{5 (r - s - 1)} \binom{r}{s} \displaystyle\frac{1}{ (r - s - 1)!} \\
	& \le (k + 1) \displaystyle\sum_{r = 3k}^{\infty}    \displaystyle\frac{32^r r^k}{(r - k - 1)!} \le  (k + 1) 3^k \displaystyle\sum_{r = 3k}^{\infty}  \displaystyle\frac{32^r}{ (r - 2k - 1)!}  \le 2^{13 k + 5} \displaystyle\sum_{r = 0}^{\infty}  \displaystyle\frac{32^r}{ r!} \le 2^{13 k + 53}.
	\end{aligned}
	\end{flalign}
	
	\noindent Since $2^{4k + 2} + 2^{3k + 7} (2^{4k + 48} + 2^{13k + 53}) \le 2^{16k + 62} \le 2^{78k}$, the proposition follows from \eqref{2k1nk2mestimate}, \eqref{e1estimate}, \eqref{e2estimate}\footnote{This in fact shows $\big| \mathcal{E} (m) \big| \le 2^{16k + 62} |m|! \le 2^{78k} |m|!$. However, due to the way in which we will use Proposition \ref{mestimate} in Section \ref{ManyEstimate} below, the constants $2^{78k}$ and $2^{16k + 62}$ will be of similar efficiency, and so we use the former.}.	
\end{proof}

\section{Proof of Theorem \ref{volume}}

\label{VolumeProof}

In this section we establish Theorem \ref{volume}. In Section \ref{ManyEstimate} we provide bounds on the multi-fold inner product $\big\langle p_{\lambda^{(1)}} \b| p_{\lambda^{(2)}} \b| \cdots \b| p_{\lambda^{(n)}} \big\rangle$. These estimates will be used in Section \ref{EstimateC} to conclude the proof of Theorem \ref{volume}.

\subsection{Estimating the Multi-Fold Inner Product}

\label{ManyEstimate}

Our goal in this section is to provide two estimates for the multi-fold inner product given by \eqref{manyinnermany}. The first, stated as Proposition \ref{innermanyestimate} below, provides estimates on such inner products in general; the second, Proposition \ref{innermanyestimate2} provides a stronger estimate if we make an additional assumption on the partition sequence $\big\{ \lambda^{(i)} \big\}$.

\begin{prop}
	
	\label{innermanyestimate}
	
	Let $a \ge 1$ and $b \ge 0$ be integers; $\lambda^{(1)}, \lambda^{(2)}, \ldots , \lambda^{(a)}$ be partitions of lengths at least two; and $D_1, D_2, \ldots , D_b \ge 2$ be integers. Denote $\ell_i = \ell \big(\lambda^{(i)} \big)$ for each $i \in [1, a]$, $L = \sum_{i = 1}^a \ell_i$, $|\lambda| = \sum_{i = 1}^a \big| \lambda^{(i)} \big|$, and $B = \sum_{i = 1}^b D_i$. Then,
	\begin{flalign}
	\label{innerp1}
	\Big| \langle p_{\lambda^{(1)}} \b| p_{\lambda^{(2)}} \b| \cdots \b| p_{\lambda^{(a)}} \b| p_{D_1} \b| p_{D_2} \b| \cdots \b| p_{D_b} \rangle \Big| \le 2^{89 L + 5} \big( |\lambda| + a + B - L \big)!.
	\end{flalign}

\end{prop}

\begin{rem}
	
	If we define the one-part partitions $\lambda^{(a + j)} = (D_j)$ for each $j \in [1, b]$, then the expression $|\lambda| + a + B - L$ appearing on the right side of \eqref{innerp1} can be rewritten as $\sum_{i = 1}^{a + b} \big( | \lambda^{(i)} | - \ell ( \lambda^{(i)} ) + 1 \big) $.
\end{rem}

\begin{proof}[Proof of Proposition \ref{innermanyestimate}]
	
As in Definition \ref{innermanyp}, define the set partition $\rho = \big( \rho^{(1)}, \rho^{(2)}, \ldots , \rho^{(a + b)} \big)$ of $\{ 1, 2, \ldots , L + b \}$ as follows. For each integer $i \in [1, a]$, define the partial sum $L_i = \sum_{j = 1}^i \ell_j$ (with $L_0 = 0$); then set $\rho^{(i)} = \{ L_{i - 1} + 1, L_{i - 1} + 2, \ldots , L_i \}$ for each $i \in [1, a]$, and set $\rho^{(j)} = \{ L + j - a \}$ for each $j \in [a + 1, a + b]$.

In view of the definition \eqref{manyinnermany}, we have that
\begin{flalign}
\label{21manyinnerestimate}
\begin{aligned}
\big\langle  p_{\lambda^{(1)}} \b| p_{\lambda^{(2)}} \b| \cdots \b| p_{\lambda^{(a)}} \b| p_{D_1} \b| p_{D_2} \b| \cdots \b| p_{D_b} \big\rangle  = \displaystyle\sum_{\alpha \in \mathscr{C} (\rho)} \displaystyle\prod_{i = 1}^{L - a + 1} \big\langle \omega_{\alpha^{(i)}} \big\rangle,
\end{aligned}
\end{flalign}

\noindent where the sum is over all reduced set partitions $\alpha = \big( \alpha^{(1)}, \alpha^{(2)}, \ldots , \alpha^{(L - a + 1)} \big) \in \mathcal{P}_{L + b}$ that are complementary to $\rho$, and $\omega_{\alpha^{(i)}} \subset \mathbb{Z}_{\ge 1} $ is a set of $\big| \alpha^{(i)} \big|$ integers defined as follows. We stipulate that positive integer $u \in \omega_{\alpha^{(i)}}$ if and only if either $a + j \in \alpha^{(i)}$ and $u = D_j$ for some $j \in [1, b]$ or there exist $j \in [1, a]$ and $k \in \big[ 1, \ell \big(\lambda^{(j)} \big) \big]$ such that $u = \lambda_k^{(j)}$ and $L_{j - 1} + k \in \alpha^{(i)}$. Let $\big| \omega_{\alpha^{(i)}} \big|$ denote the sum of the elements in $\omega_{\alpha^{(i)}}$ for each $i \in [1, L - a + 1]$.

Now let $\alpha = \big( \alpha^{(1)}, \alpha^{(2)}, \ldots , \alpha^{(r)} \big)$ be a reduced set partition complementary to $\rho$. Then, we must have that $r = \ell (\alpha) = L + 1 - a$ due to Definition \ref{partitionscomplement}. Furthermore, each $\alpha^{(i)}$ must contain at least one element from $\{ 1, 2, \ldots , L \}$. Indeed, otherwise, there would exist some $\alpha^{(i)} \subseteq \{ L + 1, L + 2, \ldots , L + b \}$, meaning that both $\alpha$ and $\rho$ would be refinements of $\big( \alpha^{(i)}, \{ 1, 2, \ldots , L + b \} \setminus \alpha^{(i)} \big)$, which is a contradiction.

Now, for any $A = (A_1, A_2, \ldots,  A_{L - a + 1}) \in \mathcal{C}_L (L - a + 1)$ and $B = (B_1, B_2, \ldots , B_{L - a + 1}) \in \mathcal{G}_b (L - a + 1)$, let $\mathfrak{R} (A, B)$ denote the set of non-reduced set partitions of $\alpha = \big( \alpha^{(1)}, \alpha^{(2)}, \ldots , \alpha^{(L - a + 1)} \big) \in \mathfrak{P}_{L + b; L - a + 1}$ satisfying the following three properties. First, we have that $\big|\alpha^{(i)} \cap \{ 1, 2, \ldots , L \} \big| = A_i$; second, that $\big| \alpha^{(i)} \cap \{ L + 1, L + 2, \ldots , L + b \} \big| = B_i$; and third, that $\alpha$ and $\rho$ are transverse, meaning that $\big| \rho^{(i)} \cap \alpha^{(j)} \big| \le 1$ for each $i, j$. Observe that $\mathscr{C} (\rho) \subseteq \bigcup_{A \in \mathcal{C}_L (L - a + 1)} \bigcup_{B \in \mathcal{G}_b (L - a + 1)} \mathfrak{R} (A, B)$ in view of Lemma \ref{transverse}. Further observe that $\big| \omega_{\alpha^{(i)}} \big| \ge A_i + 2 B_i$, since $\omega_{\alpha^{(i)}}$ has $A_i + B_i$ (positive) elements, $B_i$ elements of which are in $\{ D_1, D_2, \ldots , D_b \}$ (and therefore bounded below by $2$).

In view of \eqref{21manyinnerestimate} and the third identity in \eqref{aipksize}, we have that
\begin{flalign}
\label{1manyinnerestimate}
\begin{aligned}
\big\langle p_{\lambda^{(1)}} \b| p_{\lambda^{(2)}} \b| & \cdots \b| p_{\lambda^{(a)}} \b| p_{D_1} \b| p_{D_2} \b| \cdots \b| p_{D_b} \big\rangle \\
& \le \displaystyle\frac{1}{(L - a + 1)!} \displaystyle\sum_{A \in \mathcal{C}_L (L - a + 1)} \displaystyle\sum_{B \in \mathcal{G}_b (L - a + 1)} \displaystyle\sum_{\alpha \in \mathfrak{R} (A; B)} \displaystyle\prod_{i = 1}^{L - a + 1} \langle \omega_{\alpha^{(i)}} \rangle \\
& \le \displaystyle\frac{2^{79 L}}{(L - a + 1)!} \displaystyle\sum_{A \in \mathcal{C}_L (L - a + 1)} \displaystyle\sum_{B \in \mathcal{G}_b (L - a + 1)} \displaystyle\sum_{\alpha \in \mathfrak{R} (A; B)} \displaystyle\prod_{i = 1}^{L - a + 1} \big| \omega_{\alpha^{(i)}} \big|!,
\end{aligned}
\end{flalign}

\noindent where we used \eqref{singleinnermany}, the fact that $\mathfrak{z} (k) \le 4$, Proposition \ref{mestimatelarge}, Proposition \ref{mestimate}, and the fact that the total number of ones among the $\omega_{\alpha^{(i)}}$ is at most equal to $L$. In \eqref{1manyinnerestimate}, $\big| \omega_{\alpha^{(i)}} \big| = \sum_{j \in \omega_{\alpha^{(i)}}} j$ denotes the sum of the elements in $\omega_{\alpha^{(i)}}$.

Now let $\mathfrak{s} \in [1, L - a + 1]$ denote the minimal index such that $A_{\mathfrak{s}} + 2 B_{\mathfrak{s}} = \max_{1 \le i \le L - a + 1} (A_i + 2 B_i)$. Then, apply \eqref{aici1} with the $A_i$ and $C_i$ there equal to our $\big| \omega_{\alpha^{(i)}} \big| - A_i - 2B_i \ge 0$ and $A_i + 2 B_i$, respectively (observe that  we may do since each $D_i \ge 2$). Since $\sum_{i = 1}^{L - a + 1} A_i = L$; $\sum_{i = 1}^{L - a + 1} B_i = b$; and $\sum_{i = 1}^{L - a + 1} \big| \omega_{\alpha^{(i)}} \big| = |\lambda| + B$, this yields
\begin{flalign}
\label{manylambdainnerestimate1}
\begin{aligned}
\big\langle & p_{\lambda^{(1)}} \b| p_{\lambda^{(2)}} \b| \cdots \b| p_{\lambda^{(a)}} \b| p_{D_1} \b| p_{D_2} \b| \cdots \b| p_{D_b} \big\rangle \\
& \le \displaystyle\frac{2^{79 L}}{(L - a + 1)!}  \displaystyle\sum_{A \in \mathcal{C}_L (L - a + 1)} \displaystyle\sum_{B \in \mathcal{G}_b (L - a + 1)} \displaystyle\sum_{\alpha \in \mathfrak{R} (A; B)} \big( |\lambda| + B + A_{\mathfrak{s}} + 2 B_{\mathfrak{s}} - L - 2 b \big)! \\
& \qquad \times \displaystyle\prod_{\substack{1 \le i \le L - a + 1 \\ i \ne \mathfrak{s}}} (A_i + 2 B_i)! \\
& \le \displaystyle\frac{2^{79 L}}{(L - a + 1)!} \displaystyle\sum_{A \in \mathcal{C}_L (L - a + 1)} \displaystyle\sum_{B \in \mathcal{G}_b (L - a + 1)} \binom{L}{A_1, A_2, \ldots , A_{L - a + 1}} \binom{b}{B_1, B_2, \ldots , B_{L - a + 1}} \\
& \qquad \times \big( |\lambda| + B + A_{\mathfrak{s}} + 2 B_{\mathfrak{s}} - L - 2 b \big)! \displaystyle\prod_{\substack{1 \le i \le L - a + 1 \\ i \ne \mathfrak{s}}} (A_i + 2 B_i)!,
\end{aligned}
\end{flalign}

\noindent where we have used the fact that $\big| \mathfrak{R} (A, B) \big| \le \binom{L}{A_1, A_2, \ldots , A_{L - a + 1}} \binom{b}{B_1, B_2, \ldots , B_{L - a + 1}}$. The latter fact holds by first ignoring the transversality condition between $\alpha$ and $\rho$, and then by using the second identity in \eqref{aipksize}, which implies that there are at most $\binom{L}{A_1, A_2, \ldots , A_{L - a + 1}}$ possibilities for $A$ and at most $\binom{b}{B_1, B_2, \ldots , B_{L - a + 1}}$ possibilities for $B$.

 Observe that since $|\lambda| + B \ge L + 2b$, we have that
\begin{flalign}
\label{productlambda1}
\begin{aligned}
\big( |\lambda| + B + A_{\mathfrak{s}} + 2 B_{\mathfrak{s}} - L - 2 b \big)! & = \big( |\lambda| + a + B - L \big)! \displaystyle\prod_{i = 0}^{2b + a - A_{\mathfrak{s}} - 2B_{\mathfrak{s}} - 1} \displaystyle\frac{1}{|\lambda| + a + B - L - i} \\
& \le \big( |\lambda| + a + B - L \big)! \displaystyle\prod_{i = 0}^{2b + a - A_{\mathfrak{s}} - 2B_{\mathfrak{s}} - 1} \displaystyle\frac{1}{a + 2b - i} \\
& = \displaystyle\frac{\big( |\lambda| + a + B - L \big)! (A_{\mathfrak{s}} + 2 B_{\mathfrak{s}})!}{(a + 2b)!}.
\end{aligned}
\end{flalign}

\noindent Inserting \eqref{productlambda1} into \eqref{manylambdainnerestimate1}, applying \eqref{habl}, and using the fact that $L! = a! (L - a)! \binom{L}{a} \le 2^L a! (L - a + 1)!$, we obtain
\begin{flalign*}
\big\langle & p_{\lambda^{(1)}} \b| p_{\lambda^{(2)}} \b| \cdots \b| p_{\lambda^{(a)}} \b| p_{D_1} \b| p_{D_2} \b| \cdots \b| p_{D_b} \big\rangle \\
& \le \displaystyle\frac{2^{79 L} \big( |\lambda| + a + B - L \big)! L! b!}{(a + 2b)! (L - a + 1)!} \displaystyle\sum_{A \in \mathcal{C}_L (L - a + 1)} \displaystyle\sum_{B \in \mathcal{G}_b (L - a + 1)} \displaystyle\prod_{i = 1}^{L - a + 1} \displaystyle\frac{ (A_i + 2 B_i)!}{A_i! B_i!} \\
& \le \displaystyle\frac{2^{88 L + 5} \big( |\lambda| + a + B - L \big)! L! }{a! (L - a + 1)!} \le 8^{88 L + 5} \big( |\lambda| + a + B - L \big)! \binom{L}{a} \le 2^{89 L + 5} \big( |\lambda| + a + B - L \big)!,
\end{flalign*}

\noindent from which the proposition follows.
\end{proof}

If at least one of the $\lambda^{(i)}$ has at least two parts that are at least equal to two, then the following proposition indicates that it is possible to improve upon the bound of Proposition \ref{innermanyestimate}. 	

\begin{prop}

\label{innermanyestimate2}

Adopt the notation of Proposition \ref{innermanyestimate} and additionally suppose that there exists some $j_0 \in [1, a]$ such that at least two parts of $\lambda^{(j_0)}$ are at least $2$. Then,
\begin{flalign*}
\Big| \langle p_{\lambda^{(1)}} \b| p_{\lambda^{(2)}} \b| \cdots \b| p_{\lambda^{(a)}} \b| p_{D_1} \b| p_{D_2} \b| \cdots \b| p_{D_b} \rangle \Big|  \le 2^{89 L + 5} \big( |\lambda| + a + B - L - 1 \big)!.
\end{flalign*}

\end{prop}

\begin{proof}
	The proof of this proposition will be similar to that of Proposition \ref{innermanyestimate}, except that we will be able to use the existence of some $j_0$ such that $\lambda^{(j_0)}$ has two parts not equal to $1$ to improve the estimate \eqref{manylambdainnerestimate1}.
	
	To explain further, we begin in the same way as we did in the proof of Proposition \ref{innermanyestimate}; in particular, adopt the notation of that proof. Then, the estimate \eqref{1manyinnerestimate} still holds.
	
	Now let $\mathfrak{s} \in [1, L - a + 1]$ denote the minimal index such that $A_{\mathfrak{s}} + 2 B_{\mathfrak{s}} = \max_{1 \le i \le L - a + 1} (A_i + 2 B_i)$, and let $\mathfrak{h} \in [1, L - a + 1]$ denote the minimal index such that $A_{\mathfrak{h}} + 2 B_{\mathfrak{h}} = \max_{i \ne \mathfrak{s}} (A_i + 2 B_i)$; in particular, $\mathfrak{h}$ is an index such that $A_{\mathfrak{h}} + 2 B_{\mathfrak{h}}$ is second largest among all $A_i + 2 B_i$. Set $\mathfrak{A}_i = \big| \omega_{\alpha^{(i)}} \big| - A_i - 2B_i$ and $\mathfrak{C}_i = A_i + 2 B_i$ for each $i \in [1, L - a + 1]$; since each $D_i \ge 2$, each $\mathfrak{A}_i$ is nonnegative.
	
	Furthermore, since $\rho$ and $\alpha$ are transverse, there exist two distinct indices $u, v \in [1, L - a + 1]$ such that $\lambda_1^{(j_0)} \in \omega_{\alpha^{(u)}}$ and $\lambda_2^{(j_0)} \in \omega_{\alpha^{(v)}}$. Since $\lambda_1^{(j_0)} \ge \lambda_2^{(j_0)} \ge 2$, it follows that $\big| \omega_{\alpha^{(u)}} \big| \ge A_u + 2 B_u + 1$ and $\big| \omega_{\alpha^{(v)}} \big| \ge A_v + 2 B_v + 1$. Therefore, $\mathfrak{A}_u$ and $\mathfrak{A}_v$ are positive, so applying \eqref{1manyinnerestimate}, \eqref{2aici} (with the $A_i$ and $C_i$ there equal to the $\mathfrak{A}_i$ and $\mathfrak{C}_i$ here, respectively) and using the facts that $\sum_{i = 1}^{L - a + 1} A_i = L$; $\sum_{i = 1}^{L - a + 1} B_i = b$; and $\sum_{i = 1}^{L - a + 1} \big| \omega_{\alpha^{(i)}} \big| = |\lambda| + B$ yields
	\begin{flalign}
	\label{manylambdainnerestimate2}
	\begin{aligned}
	\big\langle & p_{\lambda^{(1)}} \b| p_{\lambda^{(2)}} \b| \cdots \b| p_{\lambda^{(a)}} \b| p_{D_1} \b| p_{D_2} \b| \cdots \b| p_{D_b} \big\rangle \\
	& \le \displaystyle\frac{2^{79 L}}{(L - a + 1)!}  \displaystyle\sum_{A \in \mathcal{C}_L (L - a + 1)} \displaystyle\sum_{B \in \mathcal{G}_b (L - a + 1)} \displaystyle\sum_{\alpha \in \mathfrak{R} (A; B)} \big( |\lambda| + B + A_{\mathfrak{s}} + 2 B_{\mathfrak{s}} - L - 2 b - 1 \big)! \\
	& \qquad \times (A_{\mathfrak{h}} + 2 B_{\mathfrak{h}} + 1)! \displaystyle\prod_{\substack{1 \le i \le L - a + 1 \\ i \ne \mathfrak{s}, \mathfrak{h}}} (A_i + 2 B_i)! \\
	& \le \displaystyle\frac{2^{79 L}}{(L - a + 1)!} \displaystyle\sum_{A \in \mathcal{C}_L (L - a + 1)} \displaystyle\sum_{B \in \mathcal{G}_b (L - a + 1)} \binom{L}{A_1, A_2, \ldots , A_{L - a + 1}} \binom{b}{B_1, B_2, \ldots , B_{L - a + 1}} \\
	& \qquad \times \big( |\lambda| + B + A_{\mathfrak{s}} + 2 B_{\mathfrak{s}} - L - 2 b - 1 \big)! (A_{\mathfrak{h}} + 2 B_{\mathfrak{h}} + 1)  \displaystyle\prod_{\substack{1 \le i \le L - a + 1 \\ i \ne \mathfrak{s}}} (A_i + 2 B_i)!.
	\end{aligned}
	\end{flalign}
	
	\noindent Observe that since at least one partition in $\lambda$ has at least two parts equal to $2$, we have that $|\lambda| \ge L + 1$; moreover, since each $D_i \ge 2$, we also have that $B \ge 2b$. Therefore $|\lambda| + B - 1 \ge L + 2b$, and so
	\begin{flalign}
	\label{productlambda2}
	\begin{aligned}
	\big( |\lambda| + B + A_{\mathfrak{s}} & + 2 B_{\mathfrak{s}} - L - 2 b - 1 \big)! \\
	& = \big( |\lambda| + a + B - L - 1 \big)! \displaystyle\prod_{i = 0}^{2b + a - A_{\mathfrak{s}} - 2B_{\mathfrak{s}} - 1} \displaystyle\frac{1}{|\lambda| + a + B - L - i - 1} \\
	& \le \big( |\lambda| + a + B - L - 1 \big)! \displaystyle\prod_{i = 0}^{2b + a - A_{\mathfrak{s}} - 2B_{\mathfrak{s}} - 1} \displaystyle\frac{1}{a + 2b - i} \\
	& \le \displaystyle\frac{\big( |\lambda| + a + B - L - 1\big)! (A_{\mathfrak{s}} + 2 B_{\mathfrak{s}})!}{(a + 2b)!}.
	\end{aligned}
	\end{flalign}
	
	\noindent Inserting \eqref{productlambda2} into \eqref{manylambdainnerestimate2}, applying \eqref{gabl} and using the fact that $\frac{L!}{a! (L - a + 1)!} \le \binom{L}{a} \le 2^L$, we find that
	\begin{flalign*}
	\big\langle & p_{\lambda^{(1)}} \b| p_{\lambda^{(2)}} \b| \cdots \b| p_{\lambda^{(a)}} \b| p_{D_1} \b| p_{D_2} \b| \cdots \b| p_{D_b} \big\rangle \\
	& \le \displaystyle\frac{2^{79 L} \big( |\lambda| + a + B - L - 1 \big)! L! b!}{(a + 2b)! (L - a + 1)!} \\
	& \qquad \times \displaystyle\sum_{A \in \mathcal{C}_L (L - a + 1)} \displaystyle\sum_{B \in \mathcal{G}_b (L - a + 1)} (A_{\mathfrak{h}} + 2 B_{\mathfrak{h}} + 1)  \displaystyle\prod_{i = 1}^{L - a + 1} \displaystyle\frac{ (A_i + 2 B_i)!}{A_i! B_i!} \\
	& \le 2^{88 L + 5} \big( |\lambda| + a + B - L - 1 \big)! \displaystyle\frac{L!}{a! (L - a + 1)!} \le 2^{89 L + 5} \big( |\lambda| + a + B - L - 1 \big)!,
	\end{flalign*}
	
	\noindent from which we deduce the proposition.
\end{proof}

\subsection{Estimating \texorpdfstring{$\textbf{c} (m)$}{}}

\label{EstimateC}

Using Lemma \ref{mestimatelarge}, Proposition \ref{innermanyestimate}, and Proposition \ref{innermanyestimate2}, we can now establish Theorem \ref{volume}.

\begin{proof}[Proof of Theorem \ref{volume}]
	
Recalling the fact that $\mathcal{F}_k = k \mathfrak{f}_k$ and
the definition \eqref{kf} of $\mathfrak{f}_k$, we deduce that
\begin{flalign}
\label{mifinner}
\begin{aligned}
 \big\langle \mathcal{F}_{m_1} & \b| \mathcal{F}_{m_2} \b| \cdots \b| \mathcal{F}_{m_n} \big\rangle \\
& = \Bigg\langle \displaystyle\sum_{\wt (\lambda^{(1)}) = m_1 + 1} \displaystyle\frac{(-m_1)^{\ell (\lambda^{(1)}) - 1}}{\prod_{i = 1}^{\infty} M_i \big(\lambda^{(1)} \big) !} p_{\lambda^{(1)}} \Bigg| \cdots \Bigg|  \displaystyle\sum_{\wt (\lambda^{(n)}) = m_n + 1} \displaystyle\frac{(-m_n)^{\ell (\lambda^{(n)}) - 1}}{\prod_{i = 1}^{\infty} M_i \big(\lambda^{(n)} \big) !} p_{\lambda^{(n)}} \Bigg\rangle \\
& = \displaystyle\sum_{\wt (\lambda^{(1)}) = m_1 + 1} \cdots \displaystyle\sum_{\wt (\lambda^{(n)}) = m_n + 1} \big\langle p_{\lambda^{(1)}} \b| p_{\lambda^{(2)}} \b| \cdots \b| p_{\lambda^{(n)}} \big\rangle \displaystyle\prod_{j = 1}^n \displaystyle\frac{(-m_j)^{\ell (\lambda^{(j)}) - 1}}{\prod_{i = 1}^{\infty} M_i \big(\lambda^{(j)} \big) !}.
\end{aligned}
\end{flalign}

Now let us rewrite the right side of \eqref{mifinner}. For each integer $1 \le j \le n$, set $l_j = \ell \big(\lambda^{(j)} \big)$, and denote $r = \sum_{j = 1}^n l_j \in \big[ n, |m| \big]$. Then \eqref{mifinner} can be alternatively expressed as
\begin{flalign}
\label{mifinner2}
\begin{aligned}
\big\langle & \mathcal{F}_{m_1}  \b| \mathcal{F}_{m_2} \b| \cdots \b| \mathcal{F}_{m_n} \big\rangle \\
& = \displaystyle\sum_{r = n}^{|m|} \displaystyle\sum_{l \in \mathcal{C}_r (n)} \displaystyle\sum_{\substack{ \ell (\lambda^{(1)}) = l_1 \\ |\lambda^{(1)}| = m_1 - l_1 + 1}} \cdots \displaystyle\sum_{\substack{\ell (\lambda^{(n)}) = l_n \\ |\lambda^{(n)}| = m_n - l_n + 1 }}  \big\langle p_{\lambda^{(1)}} \b| p_{\lambda^{(2)}} \b| \cdots \b| p_{\lambda^{(n)}} \big\rangle \displaystyle\prod_{j = 1}^n \displaystyle\frac{(-m_j)^{l_j - 1}}	{\prod_{i = 1}^{\infty} M_i \big(\lambda^{(j)} \big) !}.
\end{aligned}
\end{flalign}

There is one $l = (l_1, l_2, \ldots , l_n) \in \mathcal{C}_r (n)$ when $r = n$, namely $l = 1^n$. Thus, if $r = n$, we must have that each $l_i = 1$, so that $\lambda^{(i)} = (m_i)$ for each $1 \le i \le n$. The corresponding summand is then $\big\langle p_{m_1} \b| p_{m_2} \b| \cdots \b| p_{m_n} \big\rangle$. Subtracting this term from both sides of \eqref{mifinner2} yields
\begin{flalign}
\label{mifinner3}
\begin{aligned}
\Big| \big\langle & \mathcal{F}_{m_1}  \b| \mathcal{F}_{m_2} \b| \cdots \b| \mathcal{F}_{m_n} \big\rangle - \big\langle p_{m_1} \b| p_{m_2} \b| \cdots \b| p_{m_n} \big\rangle \Big| \\
& = \Bigg| \displaystyle\sum_{r = n + 1}^{|m|} \displaystyle\sum_{l \in \mathcal{C}_r (n)} \displaystyle\sum_{\substack{ \ell (\lambda^{(1)}) = l_1 \\ |\lambda^{(1)}| = m_1 - l_1 + 1}} \cdots \displaystyle\sum_{\substack{\ell (\lambda^{(n)}) = l_n \\ |\lambda^{(n)}| = m_n - l_n + 1 }}  \big\langle p_{\lambda^{(1)}} \b| p_{\lambda^{(2)}} \b| \cdots \b| p_{\lambda^{(n)}} \big\rangle \displaystyle\prod_{j = 1}^{|m|} \displaystyle\frac{(-m_j)^{l_j - 1}}{\prod_{i = 1}^{\infty} M_i \big(\lambda^{(j)} \big) !} \Bigg| \\
& \le  \displaystyle\sum_{r = n + 1}^{|m|} \displaystyle\sum_{l \in \mathcal{C}_r (n)} \displaystyle\sum_{\substack{ \ell (\lambda^{(1)}) = l_1 \\ |\lambda^{(1)}| = m_1 - l_1 + 1}} \cdots \displaystyle\sum_{\substack{\ell (\lambda^{(n)}) = l_n \\ |\lambda^{(n)}| = m_n - l_n + 1 }}  \big\langle p_{\lambda^{(1)}} \b| p_{\lambda^{(2)}} \b| \cdots \b| p_{\lambda^{(n)}} \big\rangle \displaystyle\prod_{j = 1}^n \displaystyle\frac{m_j^{l_j - 1}}{\prod_{i = 1}^{\infty} M_i \big(\lambda^{(j)} \big) !},
\end{aligned}
\end{flalign}

\noindent where in the inequality we removed the signs (which will be irrelevant in the estimates to follow). 

To proceed, we will divide the sum on the right side of \eqref{mifinner2} into two parts; the first will consist of ``exceptional'' sequences of partitions $\lambda = \big( \lambda^{(1)}, \lambda^{(2)}, \ldots , \lambda^{(n)} \big)$, in which all of the $\lambda^{(i)}$ are of a specific form $\xi (k, l)$ to be defined below. The second will consist of all of the remaining sequences of partitions.

More specifically, for any nonnegative integers $k \ge s \ge 1$, define $\xi^{(k, s)} = (k - 2 s + 2, 1^{s - 1}) \in \mathbb{Y}_{k - s + 1}  (s)$ denote the partition with one part equal to $k - 2 s + 2$ and $s - 1$ parts equal to one. For any sequence $l = (l_1, l_2, \ldots , l_n)$, let $\Omega (l) = \Omega (l, m)$ denote the set of sequences of partitions $\lambda = \big(\lambda^{(1)}, \lambda^{(2)}, \ldots , \lambda^{(n)} \big)$ such that $\big| \lambda^{(i)} \big| = m_i - l_i + 1$; such that $\ell \big( \lambda^{(i)} \big) = l_i$ for each $i \in [1, n]$; and such that there exists an $j \in [1, n]$ such that $\lambda^{(j)}$ is not of the form $\xi^{(k, s)}$ for any integers $k \ge s \ge 1$. The latter condition is equivalent to stipulating that there exists a $j \in [1, n]$ such that $\lambda^{(j)}$ has at least two parts equal to two.

In view of \eqref{mifinner3}, we have that
\begin{flalign}
\label{mifinner4}
\begin{aligned}
\Big| \big\langle & \mathcal{F}_{m_1}  \b| \mathcal{F}_{m_2} \b| \cdots \b| \mathcal{F}_{m_n} \big\rangle - \big\langle p_{m_1} \b| p_{m_2} \b| \cdots \b| p_{m_n} \big\rangle \Big| \le \mathfrak{E}_1 + \mathfrak{E}_2,
\end{aligned}
\end{flalign}

\noindent where
\begin{flalign}
\label{e1e2inner}
\begin{aligned}
\mathfrak{E}_1 & =   \displaystyle\sum_{r = n + 1}^{|m|} \displaystyle\sum_{l \in \mathcal{C}_r (n)}   \big\langle p_{\xi^{(m_1, l_1)}} \b| p_{\xi^{(m_2, l_2)}} \b| \cdots \b| p_{\xi^{(m_n, l_n)}} \big\rangle \displaystyle\prod_{j = 1}^n \displaystyle\frac{m_j^{l_j - 1}}{(l_j - 1) !}; \\
\mathfrak{E}_2 & =  \displaystyle\sum_{r = n + 1}^{|m|} \displaystyle\sum_{l \in \mathcal{C}_r (n)} \displaystyle\prod_{j = 1}^n m_j^{l_j - 1} \displaystyle\sum_{\lambda \in \Omega (l)}  \big\langle p_{\lambda^{(1)}} \b| p_{\lambda^{(2)}} \b| \cdots \b| p_{\lambda^{(n)}} \big\rangle \displaystyle\prod_{j = 1}^n \displaystyle\frac{1}{\prod_{i = 1}^{\infty} M_i \big(\lambda^{(j)} \big) !}.
\end{aligned}
\end{flalign}

To estimate $\mathfrak{E}_1$, let $l = (l_1, l_2, \ldots , l_n) \in \mathcal{C}_r (n)$ with $k$ of the $l_i$ equal to $1$ (and the remaining $n - k$ of the $l_i$ at least equal to $2$). Since $\frac{l_i}{2} \le l_i - 1$ when $l_i \ge 2$, we have that $n + \frac{r - k}{2} \le n + \sum_{i = 1}^n (l_i - 1) = \sum_{i = 1}^n l_i = r$, and so $r - k \le 2(r - n)$. Since each $m_i \ge 2$, we can apply Proposition \ref{innermanyestimate} with the $a$ there equal to our $n - k$, the $b$ there equal to our $k$, the $L$ there equal to our $r - k$, the $\{ \lambda^{(i)} \}$ there equal to our $\big\{ \xi^{(m_i, l_i)} \big\}_{l_i \ge 2}$, and the $\{ D_i \}$ there equal to our $\{ m_i \}_{l_i = 1}$. Using the facts that $\sum_{i = 1}^n l_i = r$; $r - k \le 2 (r - n)$; and $\sum_{i = 1}^n \big| \lambda^{(i)} \big| = |m| - r + n$, this proposition yields
\begin{flalign}
\label{pxiinnerestimate}
\big\langle p_{\xi^{(m_1, l_1)}} \b| p_{\xi^{(m_2, l_2)}} \b| \cdots \b| p_{\xi^{(m_n, l_n)}} \big\rangle \le 2^{89 (r - k) + 5} \big( |m| - 2r + 2n \big)! \le 2^{178 (r - n) + 5} \big( |m| - 2r + 2n \big)!.
\end{flalign}

\noindent Inserting \eqref{pxiinnerestimate} into the definition \eqref{e1e2inner} of $\mathfrak{E}_1$, and then applying the fact (since $\sum_{i = 1}^n m_i = |m|$ and $\sum_{i = 1}^n (l_i - 1) = r - n$) that
\begin{flalign*}
\displaystyle\sum_{l \in \mathcal{C}_r (n)} \displaystyle\prod_{i = 1}^n \displaystyle\frac{m_i^{l_i - 1}}{(l_i - 1)!} = \displaystyle\frac{|m|^{r - n}}{(r - n)!},
\end{flalign*}

\noindent yields
\begin{flalign*}
\mathfrak{E}_1 & \le 32 \displaystyle\sum_{r = n + 1}^{|m|}  \displaystyle\sum_{l \in \mathcal{C}_r (n)}  2^{178 (r - n)} \big( |m| - 2 r + 2n \big)! \displaystyle\prod_{i = 1}^n \displaystyle\frac{m_i^{l_i - 1}}{(l_i - 1)!} \\
& = 32 \displaystyle\sum_{r = n + 1}^k \displaystyle\frac{2^{178 (r - n) } \big( |m| - 2r + 2n \big)! |m|^{r - n}}{(r - n)!}.
\end{flalign*}

\noindent Using the first estimate in \eqref{kaestimate} and the fact that $r > n$, we deduce that $\big( |m| - 2r + 2n \big)! |m|^{r - n} \le 2^{8 (r - n)} \big( |m| - 1 \big)!$, from which it follows that
\begin{flalign}
\label{e1estimateinner}
\begin{aligned}
\mathfrak{E}_1 & \le 32 \big( |m| - 1 \big)! \displaystyle\sum_{r = n + 1}^k \displaystyle\frac{2^{186 (r - n)}}	{(r - n)!} \le 32 e^{2^{186}} \big( |m| - 1 \big) ! \le 2^{2^{187} } \big( |m| - 1 \big)!.
\end{aligned}
\end{flalign}

 Next we estimate $\mathfrak{E}_2$. Recall that for each $\lambda = \big( \lambda^{(1)}, \lambda^{(2)}, \ldots , \lambda^{(n)} \big) \in \Omega (l)$ there exists some $j \in [1, n]$ such that $\lambda^{(j)}$ has at least two parts that are at least equal to two. Therefore, if $k$ of the $\lambda^{(i)}$ have length one, we can apply Proposition \ref{innermanyestimate2} with the $a$ there equal to our $n - k$, the $b$ there equal to our $k$, the $L$ there equal to our $r - k$, the $\{ \lambda^{(i)} \}$ there equal to our $\big\{ \lambda^{(i)} \big\}_{l_i \ge 2}$, and the $\{ D_i \}$ there equal to our $\{ m_i \}_{l_i = 1}$. This yields
\begin{flalign}
\label{plambdainnerestimatemany}
\displaystyle\max_{\lambda \in \Omega (l)} \big\langle p_{\lambda^{(1)}} \b| p_{\lambda^{(2)}} \b| \cdots \b| p_{\lambda^{(n)}} \big\rangle 	\le 2^{89 (r - k) + 5} \big( |m| - 2r + 2n - 1 \big)! \le 2^{178 (r - n) + 5} \big( |m| - 2r + 2n - 1\big)!.
\end{flalign}

\noindent Inserting \eqref{plambdainnerestimatemany} into the definition \eqref{e1e2inner} of $\mathfrak{E}_2$, we find that
\begin{flalign*}
\mathfrak{E}_2 & \le \displaystyle\sum_{r = n + 1}^{|m|} \displaystyle\sum_{l \in \mathcal{C}_r (n)}  \bigg( \displaystyle\max_{\lambda \in \Omega (l)}  \big\langle p_{\lambda^{(1)}} \b| p_{\lambda^{(2)}} \b| \cdots \b| p_{\lambda^{(n)}} \big\rangle \bigg) \displaystyle\prod_{j = 1}^n m_j^{l_j - 1} \displaystyle\sum_{\lambda \in \Omega (l)} \displaystyle\prod_{j = 1}^n \displaystyle\frac{1}{\prod_{i = 1}^{\infty} M_i \big(\lambda^{(j)} \big) !} \\
& \le 32 \displaystyle\sum_{r = n + 1}^k 2^{178 (r - n)} \big( |m| - 2r + 2n - 1 \big)! \displaystyle\sum_{l \in \mathcal{C}_r (n)} \displaystyle\prod_{j = 1}^n m_j^{l_j - 1} \displaystyle\prod_{j = 1}^n \displaystyle\sum_{\lambda^{(j)} \in \mathbb{Y}_{m_j - l_j + 1} (l_i)} \displaystyle\frac{1}{\prod_{i = 1}^{\infty} M_i \big(\lambda^{(j)} \big) !}  \\
& = 32 \displaystyle\sum_{r = n + 1}^k 2^{178 (r - n)} \big( |m| - 2r + 2n - 1 \big)! \displaystyle\sum_{l \in \mathcal{C}_r (n)} \displaystyle\prod_{j = 1}^n  \displaystyle\frac{m_j^{l_j - 1}}{l_j!} \binom{m_j - l_j}{l_j - 1},
\end{flalign*}

\noindent where to establish the last equality we used \eqref{mlambda}. Therefore, since $\binom{m_j - l_j}{l_j - 1} \le \frac{m_j^{l_j - 1}}{(l_j - 1)!}$, $(2 l_j - 2)! \le (l_j - 1)! l_j! \binom{2 l_j - 2}{l_j - 1} \le 2^{2(l_j - 1)} (l_j - 1)! l_j!$, and $\sum_{j = 1}^n (l_j - 1) = r - n$, we obtain
\begin{flalign*}
\mathfrak{E}_2 & \le 32 \displaystyle\sum_{r = n + 1}^k 2^{180 (r - n)}  \big( |m| - 2r + 2n - 1 \big)! \displaystyle\sum_{l \in \mathcal{C}_r (n)}  \displaystyle\prod_{j = 1}^n \displaystyle\frac{m_j^{2 l_j - 2}}{(2 l_j - 2)!}  \\
& \le 32 \displaystyle\sum_{r = n + 1}^k \displaystyle\frac{2^{180 (r - n)} |m|^{2r - 2n} \big( |m| - 2r + 2n - 1 \big)!}{(2r - 2n)!}
\end{flalign*}

\noindent where we have applied Lemma \ref{sumsb}. Applying the second estimate in \eqref{kaestimate} then implies
\begin{flalign}
\label{e2estimateinner}
\mathfrak{E}_2  \le 32 \big( |m| - 1 \big)! \displaystyle\sum_{r = n + 1}^k \displaystyle\frac{2^{188 (r - n)}}{(2r - 2n)!}   \le 64 e^{2^{188}} \big( |m| - 1 \big)! \le 2^{2^{189}} \big( |m| - 1 \big)!.
\end{flalign}

 Now from \eqref{mifinner4}, \eqref{e1estimateinner}, \eqref{e2estimateinner}, the definition \eqref{singleinnermany} of the inner product $\big\langle p_{m_1} \b| p_{m_2} \b| \cdots \b| p_{m_n} \big\rangle$, and Lemma \ref{mestimatelarge} (using the fact that $m$ has no parts equal to one), we deduce that
 \begin{flalign}
 \label{fmiestimate}
 \Big| \big\langle & \mathcal{F}_{m_1}  \b| \mathcal{F}_{m_2} \b| \cdots \b| \mathcal{F}_{m_n} \big\rangle - |m|! \mathfrak{z} \big( |m| - n + 2 \big)  \Big| \le 2^{2^{190}} \big( |m| - 1 \big)!.
 \end{flalign}

 \noindent Thus the theorem, with the $C$ there equal to $2^{2^{191}} < 2^{2^{200}}$, follows from \eqref{fmiestimate} and the fact (which holds due to the first and last estimates in \eqref{2ll}) that $\Big| \mathfrak{z} \big( |m| - n + 2 \big) - 2 \Big| \le \frac{8}{|m| - n} \le \frac{16}{|m|}$.
\end{proof}

\appendix

	\section*{Appendix: Asymptotic values of Siegel--Veech constants \\ By Anton Zorich}
	
	\label{Constants}

	\setcounter{equation}{0}	
	
	\counterwithout{equation}{section}

	\subsection*{Siegel--Veech constants.}
\begin{figure}[ht]
%

\includegraphics[height=4cm, width=5.5cm, angle = -90]{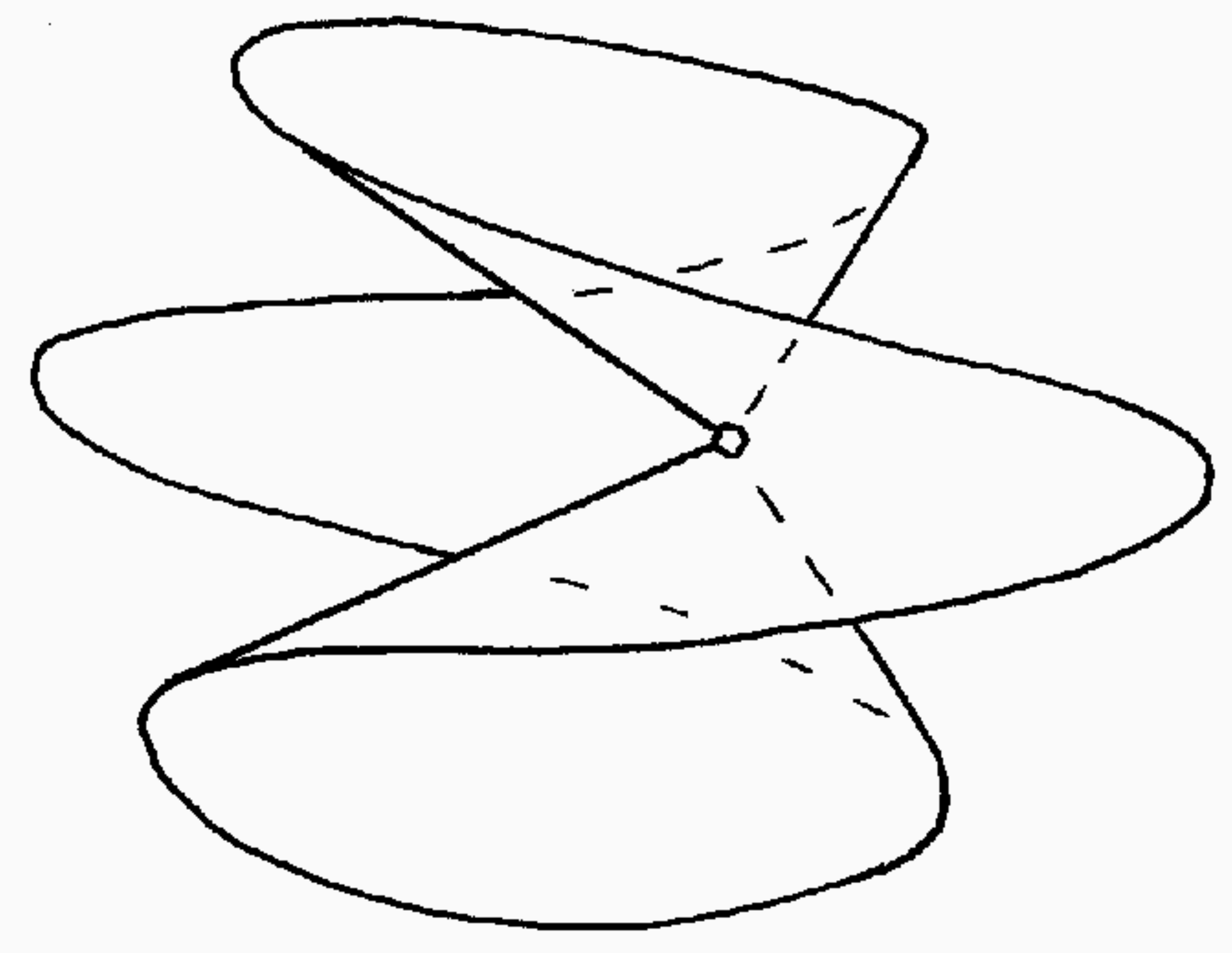} \qquad \qquad  \qquad	 \includegraphics[, height=4cm, width=5.5cm, angle = -90]{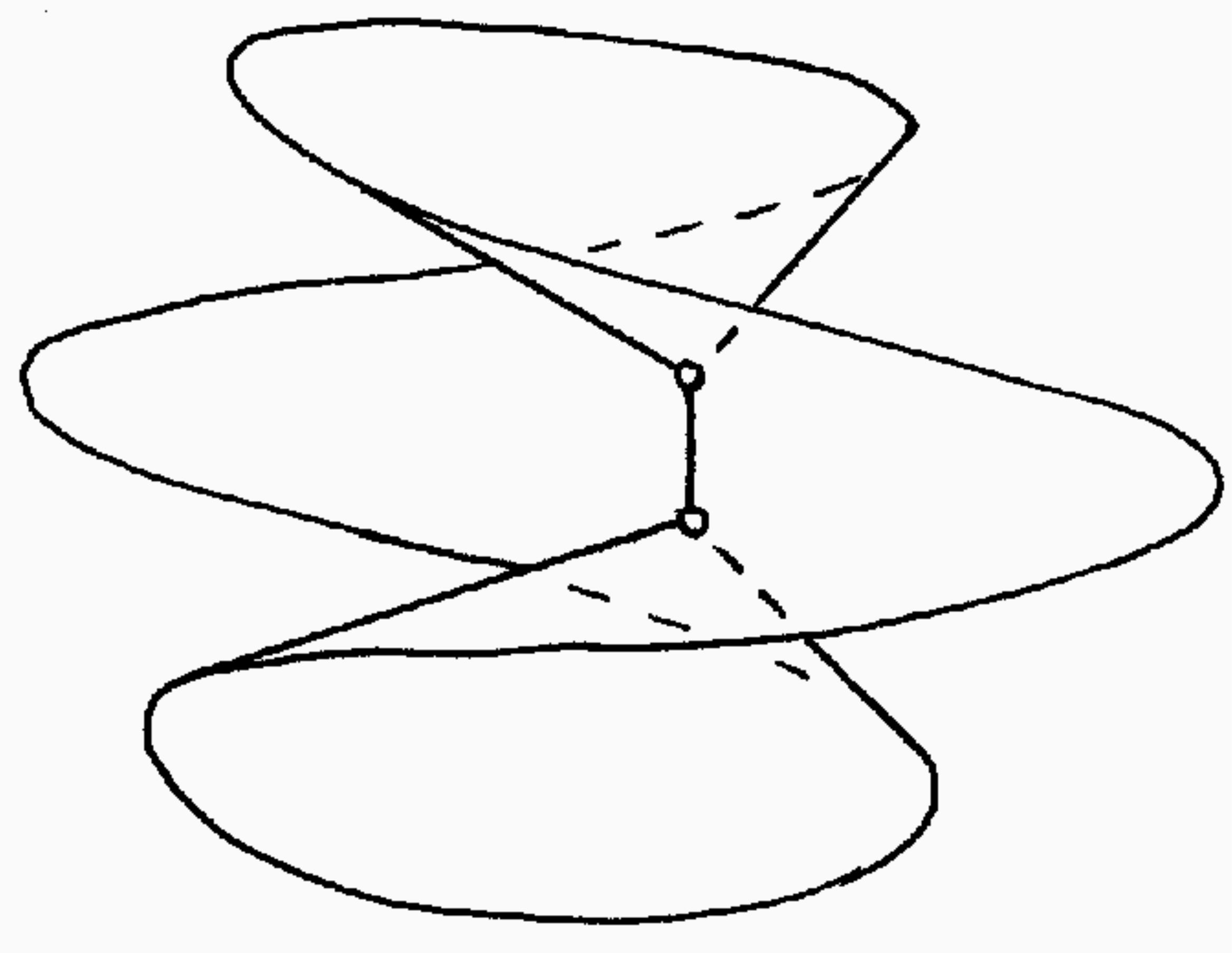}
\caption{
\label{pic:degsad}
Saddle point with cone angle $6\pi$ on the left and two saddle points
with cone angles $4\pi$ on the right.
}
\end{figure}

A holomorphic $1$-form $\omega$ on a Riemann surface defines a canonical flat metric with conical singularities located at the zeroes of $\omega$. Namely,
in the complement of a finite collection of zeroes of $\omega$, the form
$\omega$ can be represented in an appropriate local holomorphic coordinate $z$ as $\omega=dz$. 
In the associated real coordinates $(x,y)$, such that $z=x+\mathrm{i}y$, the flat metric
has the form $dx^2+dy^2$.
The cone angle of the resulting flat metric at a zero of $\omega$ of degree $m$ is $2\pi(m+1)$.
The conical singularities
are often called \textit{saddle points} or just \textit{saddles}.
Figure~\ref{pic:degsad} illustrates
a saddle point associated to a zero of degree two of the $1$-form in the left
picture and two distinct saddle points
associated to two simple zeroes  of the $1$-form in the right picture (see Figure~3 in~\cite{PBC} for more details on breaking a zero into two).
In certain situations it is convenient to interpret a regular
marked point on a translation surface as a saddle point.

The resulting flat metric has trivial linear holonomy: the parallel transport 
of a tangent vector along any closed loop on the Riemann surface brings the vector to itself. 
Note that the holomorphic $1$-form $\omega$ also defines the distinguished \textit{vertical direction} (direction of $y$-axes
in flat coordinates $(x,y)$ as above) equivariant under the parallel transport.
A closed orientable surface endowed with a flat metric
with isolated conical singularities
having trivial linear holonomy
and endowed with a distinguished direction in the tangent space
at some point (and hence at all points)
is called a \textit{translation surface}.
Similar to in the torus case, geodesics on translation surfaces do not have self-intersections at regular points.

A geodesic segment joining two saddle points (or a saddle point to itself) and having no saddle points in its interior is called \textit{saddle connection}.
The right picture in Figure~\ref{pic:degsad} illustrates
a saddle connection joining two saddle points.
The choice of the vertical direction incorporated in the structure
of translation surface endows any oriented saddle connection $\gamma$
with a direction. In this way, we can
consider the corresponding \textit{affine holonomy} vector
as a complex number in $\mathbb{C}\simeq\mathbb{R}^2$.
By construction, this complex number coincides with the integral
of the holomorphic $1$-form $\omega$ along $\gamma$. Since
both endpoints of the saddle connection $\gamma$
are located at zeroes of the $1$-form, $\gamma$ defines
an element of the relative homology group
$H_1(C,\{P_1,\dots,P_n\})$, where $C$ is the Riemann surface,
and $\{P_1,\dots,P_n\}$ is the set of zeroes of $\omega$. Thus,
the integral $\int_\gamma\omega$ defines a \textit{relative period} of $\omega$.

\begin{figure}[ht]
%
\includegraphics{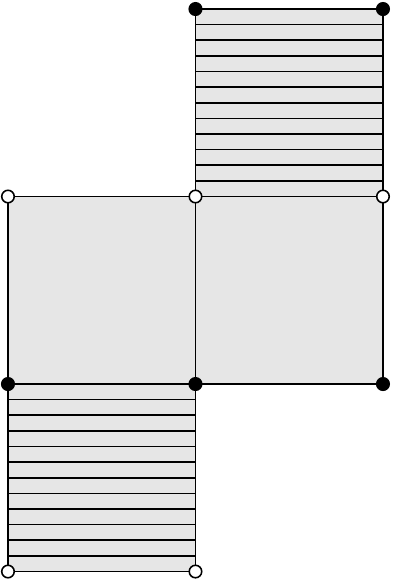} \qquad \qquad \includegraphics{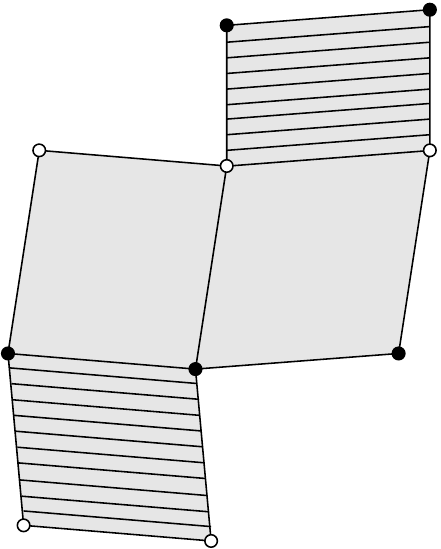}

\begin{picture}(0,0)(0,-10)
\put(-10,-10)
{\begin{picture}(0,0)(0,0)
	\put(-53,182){$\tau$}
	\put(-18,150){$\eta$}
	\put(-85,150){$\eta$}
	\put(-107,130){$\rho$}
	\put(-140,95){$\gamma$} 
	\put(-72,95){$\gamma_1$}
	\put(-18,95){$\gamma$} 
	\put(-53,60){$\tau$}
	\put(-140,40){$\beta$}
	\put(-72,40){$\beta$}
	\put(-107,5){$\rho$}
	\end{picture}}
\put(152,-10)
{\begin{picture}(0,0)(0,0)
	\put(-48,170){$\tau$} 
	\put(-10,145){$\eta$} 
	\put(-80,145){$\eta$}
	\put(-107,130){$\rho$}   %
	\put(-140,95){$\gamma$} %
	\put(-72,95){$\gamma_1$}%
	\put(-15,95){$\gamma$}  
	\put(-48,57){$\tau$} 
	\put(-142,40){$\beta$}
	\put(-75,40){$\beta$}
	\put(-107,5){$\rho$}
	\end{picture}}
\end{picture}

\caption{
\label{pic:m2simple}
Nonhomologous saddle  connections  which  have  the same holonomy
lose this property after a  generic  deformation  of the surface,
while homologous ones, $\gamma\sim\gamma_1$, share the same affine holonomy.
}
\end{figure}
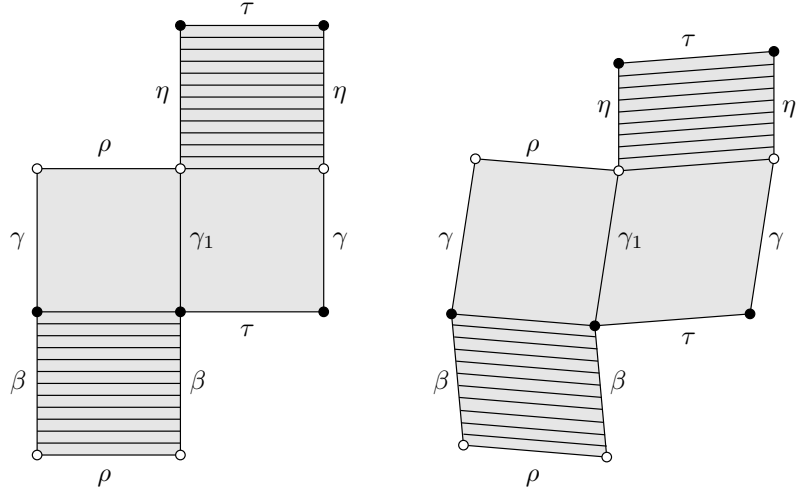

The same period may be represented
by several saddle connections $\gamma_1,\dots,\gamma_k$.
Any finite collection
$\gamma_1,\dots,\gamma_k$ of saddle connections
persists under small deformations of the translation surface.
If
the initial saddle connections are homologous as elements of
$H_1(C,\{P_1,\dots,P_n\})$, then the deformed saddle connections
stay homologous, and hence define the same period
of the deformed $1$-form.
Figure~\ref{pic:m2simple} (copied from Figure~2 in~\cite{PBC})
presents an example of a
\textit{configuration of homologous saddle connections of multiplicity $2$}.
The translation surfaces are obtained from the corresponding polygons by
gluing together pairs of sides marked by the same symbol.
The relative periods of the translation surface in the left picture
along the saddle connections represented by the positively oriented
horizontal (vertical) sides of the squares are equal to $1$ (respectively to $\mathrm{i}$).
However, after a generic small deformation of the translation surface, the periods along non-homologous
saddle connections become different,
while periods along homologous saddle connections $\gamma$ and
$\gamma_1$ coincide.

We refer to~\cite{PBC} for detailed combinatorial description of the notion
of \textit{configuration of homologous saddle connections}. The case when
such saddle connections join distinct saddle points is illustrated in
Figure~\ref{pic:dist:sad:mult}
borrowed from~\cite{PBC}. The number $k$ of homologous saddle connections
in such configuration is called the \textit{multiplicity} of the configuration.
Cutting the surface along $k$ homologous saddle connections we decompose
the surface into $k$ connected components. Each connected component
has boundary in a form of a slit composed of two geodesic segments
having the same length and the same direction. Gluing together the two
sides of each slit as in Figure~\ref{pic:dist:sad:mult}
we get $k$ translation surfaces without boundary of smaller genera
each endowed with a distinguished saddle connection.
For example, applying this operation
to homologous saddle connections $\gamma$ and $\gamma_1$ in any of the
two surfaces as in Figure~\ref{pic:m2simple}
we get two flat tori with slits of the same length and direction.
The combinatorial geometry of the corresponding \textit{configuration of homologous saddle connections} 
is described by the geometry of the resulting geometric configuration.

\begin{figure}[htb!]
%
 %

\begin{picture}(100, 12)(150, 62)
\put(0,-20){\includegraphics[height=2.75cm, width=6.25cm]{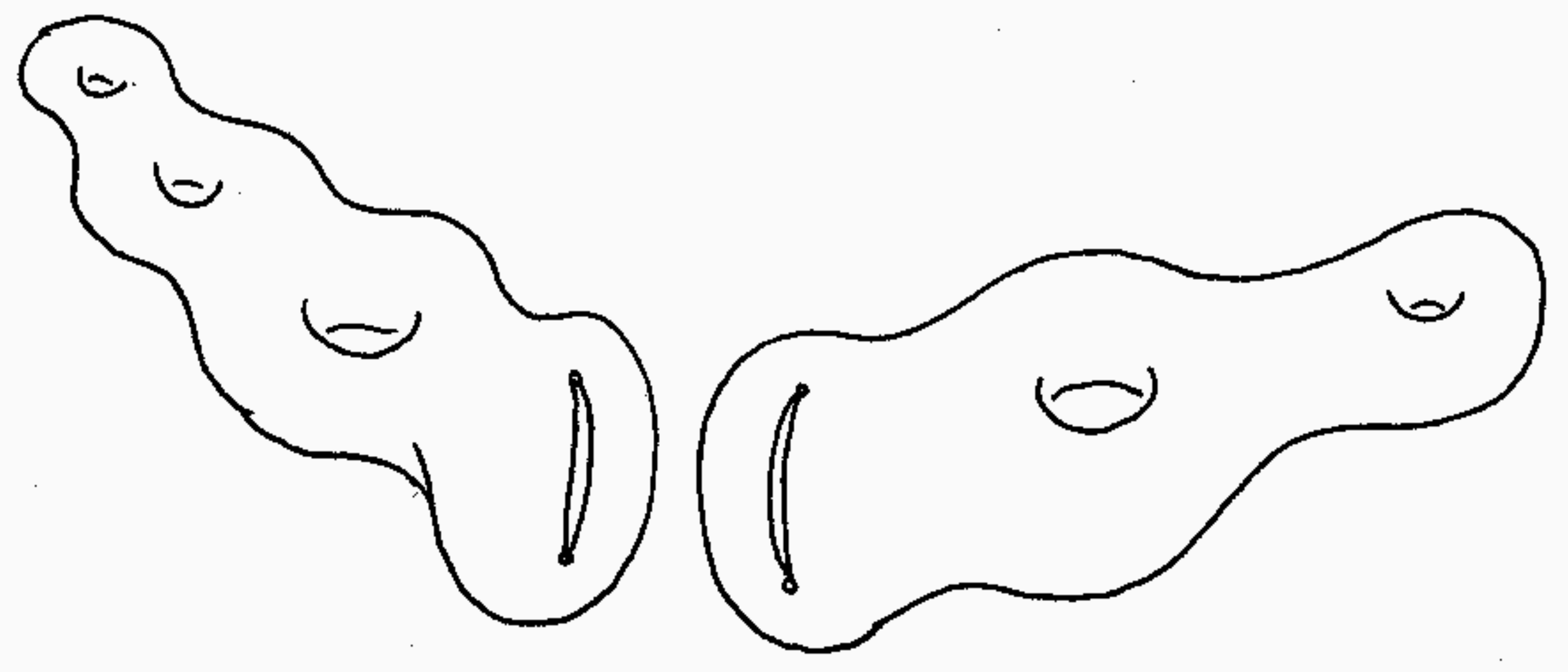}} \put(200,-10){\includegraphics[height=3cm,width=6.5cm]{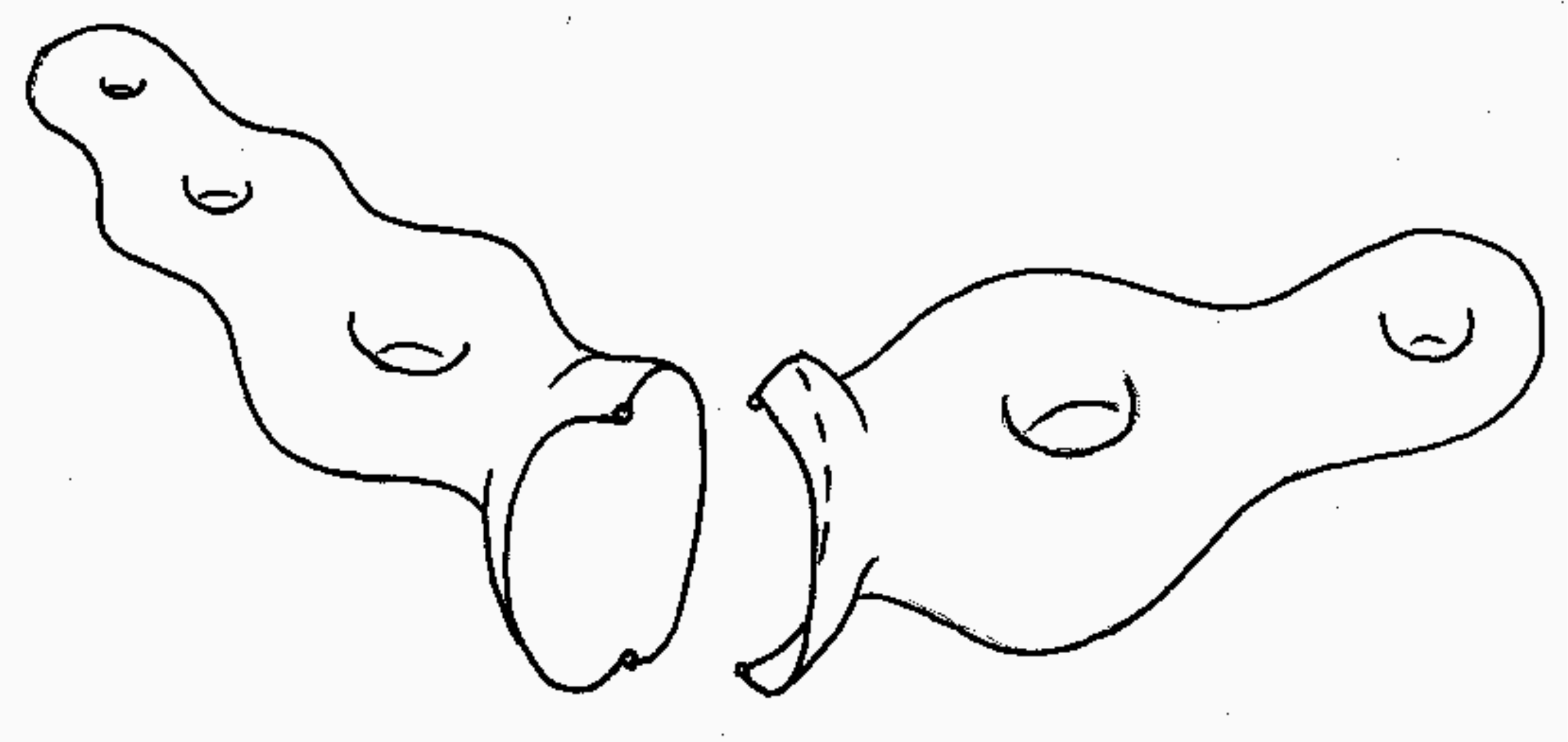}}

\put(10,-120){\includegraphics[height=3.5cm, width=3cm]{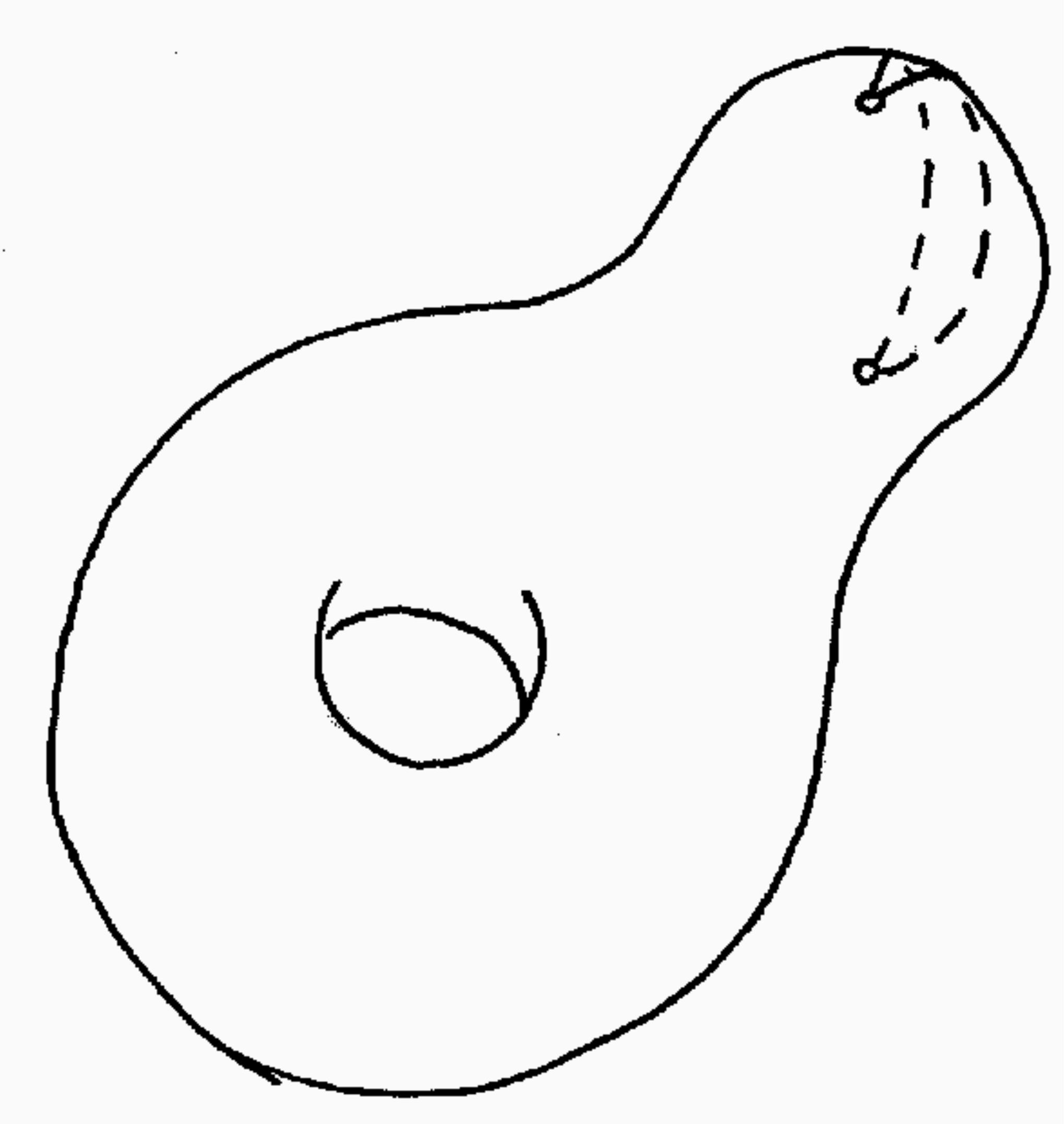}} \put(180,-175){\qquad \includegraphics[height=5cm, width=6cm]{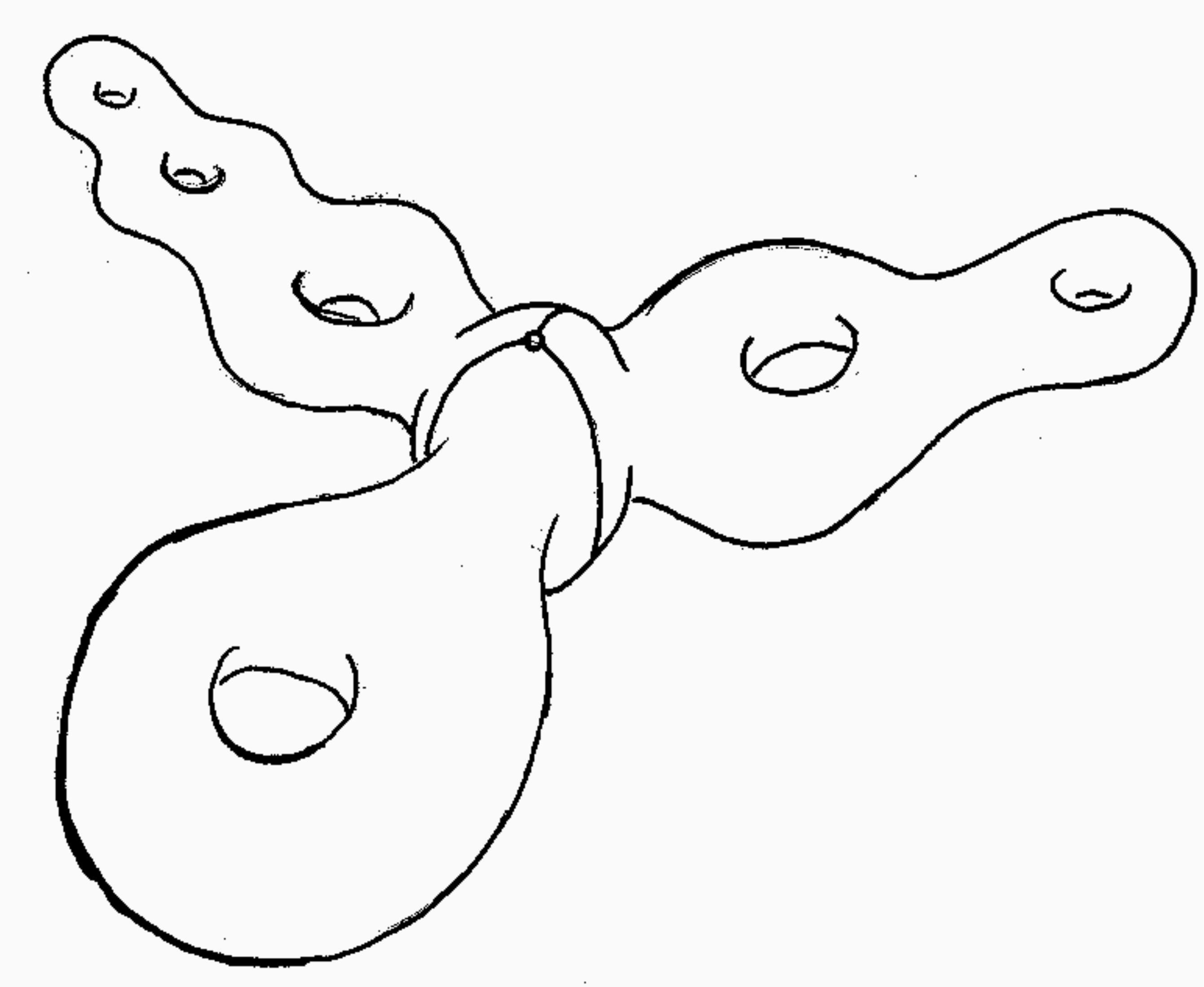} }
\end{picture}

\begin{picture}(12,-38)(12,-38)
\put(0,0){
	\begin{picture}(0,0)(0,0) 
	\put(-155,-35){$S_3$}
	\put(-35,-45){$S_2$}
	\put(-129,-84){$\scriptscriptstyle \gamma'_3$}
	\put(-117,-84){$\scriptscriptstyle \gamma''_3$}
	\put(-106,-84){$\scriptscriptstyle \gamma'_2$}
	\put(-95,-84){$\scriptscriptstyle \gamma''_2$}
	\end{picture}
}
\put(5,0){
	\begin{picture}(0,0)(0,0) 
	\put(-125,-203){$S_1$}
	\put(-116,-130){$\scriptscriptstyle \gamma''_1$}
	\put(-101,-133){$\scriptscriptstyle \gamma'_1$}
	\end{picture}
}
\put(0,0){
	\begin{picture}(0,0)(0,0) 
	\put(50,-24){$S_3$}
	\put(152,-34){$S_2$}
	\put(95,-52){$z''$}
	\put(89,-97){$z'$}
	\end{picture}
}
\put(0,0){
	\begin{picture}(0,0)(0,0) 
	\put(85,-257){$S_1$}
	\put(55,-136){$S_3$}
	\put(163,-147){$S_2$}
	\put(89,-161){$z_2$}
	\put(89,-185){$\scriptstyle \gamma_2$}
	\put(80,-181){$\scriptstyle \gamma_1$}
	\end{picture}
}
\end{picture}

\vspace{215bp} 

\caption{
\label{pic:dist:sad:mult}
Multiple homologous saddle connections. }
\end{figure}
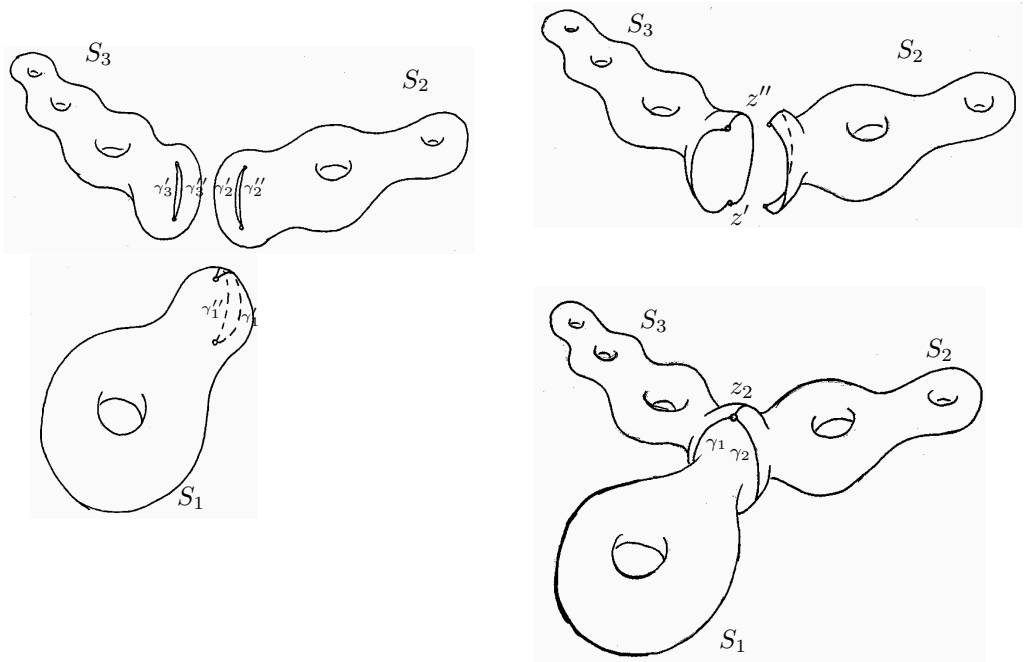

	Consider a flat torus of unit area. The  number  of
	geodesic segments of length  at most $L$  joining a generic
	pair of distinct  points on the torus grows quadratically
	as the number of lattice points in a  disc of radius $L$,
	so we get asymptotics $\pi L^2$. The number of (homotopy
	classes) of closed  geodesics of length at most $L$ has
	different  asymptotics.  Since  we want to count  only
	\textit{primitive} geodesics (those  which  do not repeat
	themselves) now we have to count  only \textit{coprime}
	lattice points in a  disc of radius $L$,  considered up to
	a symmetry of  the torus. Therefore we get the
	asymptotics     $\cfrac{1}{2\zeta(2)}\cdot      \pi
	L^2=\cfrac{3}{\pi^2}\cdot \pi L^2$.
  	
	It is proved in~\cite{AFS} that  the  growth  rate
	of the number of saddle connections for a generic
	translation surface corresponding to any stratum $\cH(m)$
	in the moduli space of Abelian differentials also has
	quadratic asymptotics  $c\cdot  (\pi L^2)$, and, moreover,
	almost all flat surfaces of unit area in any connected
	component of any stratum share the same  constant $c$ in
	the  asymptotics. The constant $c$ is called the
	\textit{Siegel--Veech constant}. It depends on the
	connected component of the stratum and on the geometric
	type of geodesic segments which we count. In the two
	examples for the torus, the Siegel--Veech constant
	corresponding to the count of geodesic segments joining a
	generic pair of distinct points is equal to $1$ while the
	Siegel--Veech constant corresponding to the count of
	primitive geodesic segments joining a fixed point to itself
	equals $\frac{3}{\pi^2}$.

	\subsection*{Volume asymptotics.}
	Let $m=(m_1,\dots,m_n)$  be  an  unordered  partition  of a
	positive even number  $2g-2$,  i.e., let
	$|m|=m_1+\dots+m_n=2g-2$.  Denote  by $\mathbb{Y}_{2g - 2}$ the set
	of all partitions. Denote by $\nu_1 \big(\cH(m_1,\dots,m_n)\big)$ the
	Masur--Veech volume of the stratum $\cH(m_1,\dots,m_n)$ in
	normalization of~\cite{PBC}.

	Theorem \ref{volumeestimateasymptotic} can be rephrased as follows.
	
	\begin{NoNumberTheorem}
		For any $m\in\mathbb{Y}_{2g - 2}$ one has
		\begin{equation}
		\label{eq:asymptotic:formula:for:the:volume}
		\nu_1 \big(\cH_1(m_1,\dots,m_n)\big)=\cfrac{4}{(m_1+1)\cdot\dots\cdot(m_n+1)}
		\cdot (1+\varepsilon(m)),
		\end{equation}
		where
		\begin{equation}
		\label{eq:asymptotic:formula:for:the:volume:bound}
		\max_{m\in\mathbb{Y}_{2g - 2}} |\varepsilon(m)| \le \frac{2^{2^{200}}}{g}\,.
		\end{equation}
	\end{NoNumberTheorem}
	
	The results in~\cite{PBC} combined with the
	bound~\eqref{eq:asymptotic:formula:for:the:volume:bound}
	for the error term
	in~\eqref{eq:asymptotic:formula:for:the:volume} immediately
	imply asymptotics of certain Siegel--Veech constants for
	connected strata in large genus. Recall that saddle
	connections might appear in tuples, triples etc of
	homologous saddle connections having the same direction and
	the same length (see~\cite{PBC} for
	details). The asymptotic formulae for Siegel--Veech
	constants become particularly simple in the case when one
	restricts the count to saddle connections of
	\textit{multiplicity one}. 


    The original preprint version of this note stated as a conjecture that	
    the Siegel--Veech constants for higher
	multiplicities become negligibly small with respect to the
	Siegel--Veech constants for multiplicity one computed below.
    This conjecture was proved in the recent paper of A.~Aggarwal~\cite{LGAC} as 
    part of the proof of the conjectures of A.~Eskin and the author on large genus asymptotics of Siegel--Veech constants.
    The very recent preprint 
    of D.~Chen, M.~M\"oller, A.~Sauvaget and D.~Zagier~\cite{VSCC} suggests an alternative proof
    of the conjectures on large genus asymptotics of Siegel--Veech constants.
    Combined with 
    the computations below it implies an alternative proof of the conjecture that 
    the Siegel--Veech constants for higher multiplicities become negligibly small in large genera.
    \medskip

	\subsection*{Saddle connections joining distinct zeroes.}
	Consider any connected stratum of the form $\cH(m_1,m_2,\dots)$,
	i.e. one which has at least two distinct zeroes, where $m_1,m_2$
	denote their degrees. The situation when $m_1=m_2$
	is not excluded. In the case when one (or
	both) of $m_1,m_2$ is equal to $0$ the ``zero of degree
	$0$'' should be seen as a generic marked point (generic pair
	of marked points respectively).
	
	\begin{Corollary}
		\label{cor:sc}
		There exists a universal constant $\ubound^{sc} > 0$ such that
		the Siegel---Veech constant
		$c^{sc}_{m_1,m_2}(\cH(m_1,m_2,\dots))$ corresponding to the
		count  of  saddle connections of multiplicity one joining a
		fixed zero of degree $m_1$ to a distinct zero of degree
		$m_2$ satisfies
		$$
		c^{sc}_{m_1,m_2}(\cH(m_1,m_2,\dots))
		= (m_1+1)(m_2+1)
		\cdot\big(1+\varepsilon^{sc}_{m_1,m_2}(m)\big)
		$$
		where
		\begin{equation}
		\label{eq:saddle:connections}
		\max_{m\in\mathbb{Y}_{2g - 2}} |\varepsilon^{sc}_{m_1,m_2}(m)|
		\le \frac{\ubound^{sc}}{g}\,.
		\end{equation}
	\end{Corollary}
	
	\begin{proof}
		By the formula preceding formula (17) in~\cite{PBC}
		the corresponding Siegel--Veech constant equals
		\begin{equation}
		\label{eq:sc:from:EMZ}
		c^{sc}_{m_1,m_2}(\cH(m_1,m_2,\dots))
		= \frac{
			(m_1+m_2+1)\nu_1 \big(\cH_1(m')\big)}{
			\nu_1\big(\cH_1(m)\big)}\,,
		\end{equation}
		where $m=\{m_1,m_2,\dots, \}$ and $m'$ is
		obtained from $m$ by replacing the first two entries
		with the single entry $m_1+m_2$.
		Applying~\eqref{eq:asymptotic:formula:for:the:volume}
		to the ratio of volumes we get
        the desired asymptotic expression.
	\end{proof}
	
	\begin{Remark}
		The answer  matches the following extremely naive
		interpretation (which should be taken with a reservation).
		Normalization of Masur--Veech volumes as in~\cite{PBC}
		implies that
		$$
		\nu_1\big(\cH(0,0,m_1,\dots,m_n)\big)=
		\nu_1\big(\cH(0,m_1,\dots,m_n)\big)=
		\nu_1\big(\cH(m_1,\dots,m_n)\big)\,.
		$$
		Thus, by~\eqref{eq:sc:from:EMZ}, the Siegel--Veech constant
		$c^{sc}_{0,0}(0,0,m_1,\dots)$ corresponding to the  number
		of  saddle connections of multiplicity one joining a
		generic marked point $P_1$ to a distinct generic marked
		point $P_2$ identically equals to $1$. When the total angle
		at $P_1$ is $m_1+1$ times bigger and the total angle  at
		$P_2$ is $m_2+1$   times  bigger  we  get an    extra
		factor $(m_1+1)(m_2+1)$.
		
		
		By the same formula~\eqref{eq:sc:from:EMZ}, the Siegel--Veech constant
		corresponding to the  number  of  saddle connections of
		multiplicity one joining a generic marked point $P_1$ to a
		fixed zero $P_2$ of degree $m_1$ identically equals to $(m_1+1)$
		$$
		c^{sc}_{0,m_1}(0,m_1,\dots)=(m_1+1)\,.
		$$
	\end{Remark}
	
	The preprint version of this appendix stated a conjecture that the condition ``multiplicity one'' in the
	statement of Corollary~\ref{cor:sc} can be omitted: the
	contribution of all higher multiplicities becomes
	negligible in large genus.
    Meanwhile, this conjecture was proved first by A.~Aggarwal in~\cite{LGAC} and then by
    D.~Chen, M.~M\"oller, A.~Sauvaget and D.~Zagier in~\cite{VSCC} by completely different methods. Moreover, article~\cite{VSCC} proves
    that counting multiple homologous saddle connections as a single one, the corresponding Siegel--Veech constant
    equals $(m_1+1)(m_2+1)$ identically for any nonhyperelliptic component of any stratum. Both proofs are
    quite involved, so for the sake of completeness we keep the original proof
    in the simplest
	case of the principal stratum, where the only higher
	multiplicity is two.
	
	\begin{Corollary}
		There exists a universal constant $\ubound^{sc}_2 > 0$ such that
		the Siegel---Veech constant
		$c^{sc;2}_{1,1}(\cH(1^{2g-2}))$ corresponding to the
		count  of  pairs of homologous saddle connections joining
		a fixed pair of distinct zeroes satisfies
		\begin{equation}
		\label{eq:principal:sc:2}
		c^{sc;2}_{1,1}(\cH(1^{2g-2}))
		\le \frac{\ubound^{sc}_2}{g}\,.
		\end{equation}
	\end{Corollary}
	\begin{proof}
		This configuration of homologous saddle connections is
		discussed in details in section 9.6
		of~\cite{PBC}. The two homologous saddle
		connections joining two fixed distinct simple zeroes cut the
		surface into two subsurfaces of positive genera $g_1, g_2$
		where $g_1+g_2=g$. Formula 9.2 in~\cite{PBC}
		gives the value of the corresponding Siegel--Veech constant
		for \textit{all} possible pairs of $2g-2$ simple zeroes.
		Dividing the corresponding expression by the number
		$(2g-2)(2g-1)/2$ of possible pairs we get
		\begin{multline*}
		c^{sc;2}_{1,1}(\cH(1^{2g-2}))
		=\frac{1}{4}\cdot\sum_{g_1+g_2=g}
		\frac{(2g-4)!\,(4g_1-3)!\,(4g_2-3)!}{(2g_1-2)!\,(2g_2-2)!\,(4g-5)!}
		\cdot
		\frac{\nu_1\big(\cH(1^{2g_1-2})\big)\cdot\nu_1\big(\cH(1^{2g_2-2})\big)}{\nu_1\big(\cH(1^{2g-2})\big)}
		\end{multline*}
		where $g_1,g_2\ge 1$.
		
		Applying~\eqref{eq:asymptotic:formula:for:the:volume} and
		taken into consideration that $g_1+g_2=g$ we conclude that
		the ratio containing the volumes is uniformly bounded from
		above uniformly in $g,g_1,g_2$.
		
		Consider the following expression as a function of $g_1$ depending on the parameter $g$, where $g_1+g_2=g$:
		$$
		a_{g_1}:=\frac{(2g-4)!\,(4g_1-3)!\,(4g_2-3)!}{(2g_1-2)!\,(2g_2-2)!\,(4g-5)!}
		$$
		Then,
		$$
		a_{1}:=\frac{1}{(4g-5)(4g-6)}
		$$
		and
		$$
		a_{g_1+1}=a_{g_1}\cdot\frac{(4g_1+1)(4g_1-1)}{(4g_2-3)(4g_2-5)}\,.
		$$
		Hence, we have $a_{g_1+1}\le a_{g_1}$
		as soon as $g_2>g_1$. Note that $a_{g-g_1}=a_{g_1}$.
		Thus,
		$$
		\sum_{g_1=1}^{g-1} a_{g_1}\le (g-1) a_1 =
		\frac{g-1}{(4g-5)(4g-6)}
		$$
		and~\eqref{eq:principal:sc:2} follows.
	\end{proof}
	\subsection*{Isolated saddle connection joining a zero to itself.}
	Consider a connected stratum $\cH(m_1,\dots)$. Let us count
	saddle connections joining a zero of degree $m_1$ to
	itself. 

\begin{figure}[htb]

 \includegraphics[height=2.25cm,width=3.75cm]{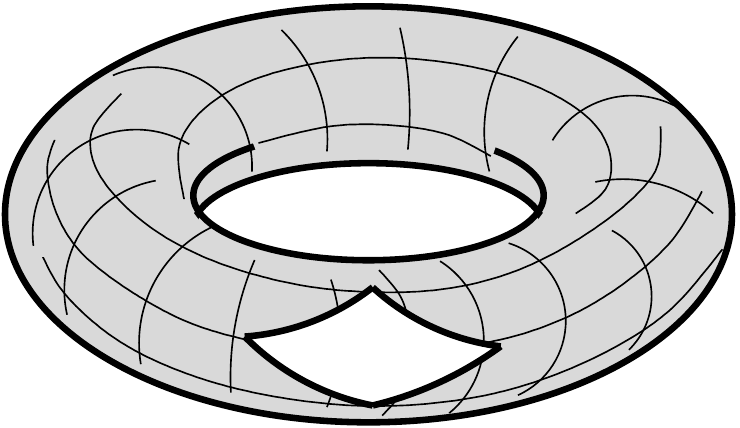} \qquad \qquad \includegraphics[height=2.25cm, width=3.75cm]{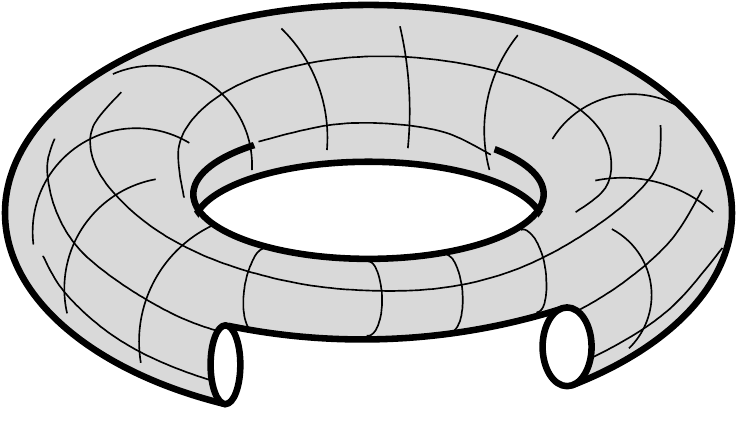} \qquad \qquad \includegraphics[height=2.25cm, width=3.75cm]{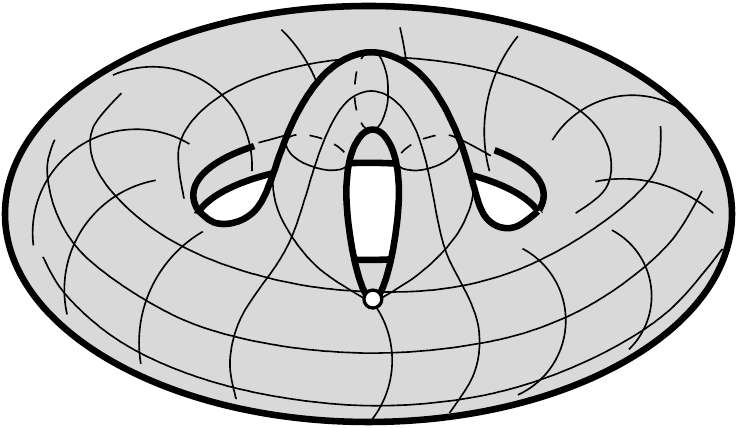}
 
\caption{
\label{fig:parallelogram:construction}
A saddle connection joining a zero to itself and not bounding a cylinder.
}
\end{figure}

	We start with saddle connections which do not bound
    a cylinder and do not have any homologous saddle connections. They can
    be obtained from a translation surface of genus $g-1$ by the following construction.
    Remove a parallelogram out of a translation surface (as in the left picture
    in Figure~\ref{fig:parallelogram:construction}). Glue one pair of 
    opposite sides of the parallelogram by parallel translation. We get
    a translation surface with two parallel geodesic boundary components 
    of the same length (as in the middle picture
    in Figure~\ref{fig:parallelogram:construction}). 
    Gluing them together we get a translation surface in genus $g$
    without boundary. By construction, the four corners of the initial 
    parallelogram are identify in one point, which is necessarily a saddle point,
    and the two geodesic boundary components of the intermediate surface become
    a single saddle connection joining this saddle point to itself.
    
    One can apply this construction backwards: cut a surface 
    along a saddle connection joining a zero to itself getting a connected surface 
    with two disjoint geodesic boundary components; join the two points on the
    boundary components coming from the original saddle point by a non self-intersecting
    path; cut the surface with boundary along this path to get a surface of 
    genus $g-1$ with a single hole in a form of curvilinear parallelogram with
    two opposite sides (coming from the original saddle connection) represented
    by parallel segments of the same length.

	\begin{Corollary}
		\label{cor:zero:to:itself:all:angles}
		There exists a universal constant $\ubound^{loop} > 0$ such
		that the Siegel---Veech constant
		$c^{loop}_{m_1}(\cH(m_1,\dots))$ corresponding to the number  of
		saddle connections of multiplicity one joining a fixed zero
		of degree $m_1$ to itself and not bounding a cylinder
		satisfies
		\begin{equation}
		\label{eq:saddle:connections:to:itself:all}
		c^{loop}_{m_1}(\cH(m_1,m_2,\dots))=
		\cfrac{(m_1+1)(m_1-1)}{2}
		\cdot\big(1+\varepsilon_{m_1}(m)\big)
		\end{equation}
		where
		\begin{equation}
		\label{eq:saddle:connections:to:itself:all:bound}
		\max_{\substack{{m\in\mathbb{Y}_{2g - 2}}\\
				m_1\in m\\
				\cH(m)\text{ is connected}}}
		|\varepsilon_{m_1}(m)|
		\le \frac{\ubound^{loop}}{g}\,.
		\end{equation}
	\end{Corollary}
	\begin{proof}
		Note that by geometric reasons any
		closed saddle connection joining a simple zero
		to itself bounds a cylinder filled with closed
		regular flat geodesics. Thus, for $m_1=1$ we
		get
		$$
		c^{loop}_{1}(\cH(m))=0\,,
		$$
		which justifies~\eqref{eq:saddle:connections:to:itself:all}
		for $m_1=1$. From now on we exclude this trivial case and
		assume that $m_1\ge 2$.
		
		We start with a more restrictive count. Namely, fix any
		integer $j$ within bounds $1\le j\le m_1-1$. Let us count
		first those closed saddle connection as above which split
		the total cone angle $2(m_1+1)\pi$ at the chosen zero of
		degree $m_1\ge 2$ into angles $(2j+1)\pi$ and
		$(2m_1-2j+1)\pi$. Our saddle connection has multiplicity one, which implies that there are no other
		homologous saddle connections. The condition $1\le j\le
		m_1-1$ automatically implies that our saddle connections
		\textit{do not} bound a cylinder.
		
		Denote by $c^{loop}_{m_1}(j;\cH(m_1,\dots))$ the Siegel---Veech
		constant corresponding to the number  of saddle connections
		of multiplicity one joining a fixed zero of degree $m_1$ to
		itself returning at the angle $(2j+1)\pi$ and not bounding
		a cylinder.
		
		Let $b'=j-1$, let $b''=m_1-j-1$. The saddle connections
		described in Corollary~\ref{cor:zero:to:itself:all:angles}
		correspond to ``creating a pair of holes assignment''
		in terminology of~\cite{PBC} applied to a
		fixed pair of zeroes of degrees $b', b''$ on a surface in a
		stratum $\cH(m')$, where $m=\{m_1,\dots,\}$ and
		$m'$ is obtained from $m$ by replacing the first
		entry (i.e. $m_1$) by two entries $b',b''$.
		
		Note that $m'$ corresponds to genus $g-1$,
		but has an extra entry with respect to $m$, so
		$\dim_{\C{}}\cH(m')=\dim_{\C{}}\cH(m)-1$.
		
		If $b'=b''$ we have ``$\gamma\to-\gamma$ symmetry'' in
		terminology of~\cite{PBC}, and this is the
		only possible symmetry. In notations
		of~\cite{PBC} we have $|\Gamma|=1$ and
		$$
		|\Gamma_-|=
		\begin{cases}
		2&\text{if }j=m_1/2-1\\
		1&\text{otherwise}\,,
		\end{cases}
		$$
		
		We are in the setting of Problem 1 from section 13.2
		in~\cite{PBC} when all the zeroes are labeled.
		Applying formula 13.1
		from~\cite{PBC} from which we remove all terms
		containing symbols $o(\cdot)$ responsible for
		unlabeling the zeroes we get
		$$
		c^{loop}_{m_1}(j;\cH(m_1,m_2,\dots))=
		\frac{1}{|\Gamma_-|}
		\cdot(b'+1)(b''+1)\cdot
		\frac{\nu_1\big(\cH(m')\big)}{\nu_1\big(\cH(m)\big)}\,.
		$$
		Applying~\eqref{eq:asymptotic:formula:for:the:volume}
		to the ratio of volumes we get
		$$
		\frac{\nu_1\big(\cH_1(m')\big)}{
			\nu_1\big(\cH_1(m)\big)}=
		\frac{m_1+1}{(b'+1)(b''+1)}
		\cdot\frac{1+\varepsilon(m')}{1+\varepsilon(m)}\,.
		$$
		Bounds~\eqref{eq:asymptotic:formula:for:the:volume:bound}
		now imply that
		\begin{equation}
		\label{eq:universal:bound:for:loop}
		\sup_{g\ge 2}
		\ g\cdot
		\max_{\substack{{m\in\mathbb{Y}_{2g - 2}}\\
				m_1\in m;\ 1\le j\le m_1-1\\
				\cH(m)\text{ is connected}}}
		\left|
		\frac{1+\varepsilon(m')}{1+\varepsilon(m)}
		-1\right|
		=:\ubound^{loop}<+\infty
		\end{equation}
		and we conclude that $c^{loop}_{m_1}(j;\cH(m_1,m_2,\dots))$
		satisfies
		\begin{equation}
		\label{eq:loop:restrictive}
		c^{loop}_{m_1}(j;\cH(m_1,m_2,\dots))=
		\begin{cases}
		\cfrac{(m_1+1)}{2}
		\cdot\big(1+\varepsilon_{m_1;j}(m)\big)
		&\text{if }j=(\frac{m_1}{2}-1)\\
		(m_1+1)
		\cdot\big(1+\varepsilon_{m_1;j}(m)\big)
		&\text{otherwise}\,,
		\end{cases}
		\end{equation}
		where
		\begin{equation*}
		\max_{m\in\mathbb{Y}_{2g - 2}} |\varepsilon_{m_1;j}(m)|
		\le \frac{\ubound^{loop}}{g}\,.
		\end{equation*}
		
		Now we pass to the count with no restrictions on the
		return angle. We have to take the sum of all Siegel--Veech
		constants as in~\eqref{eq:loop:restrictive}
		over all possible return angles, where the return angle
		$(2j+1)\pi$ is equivalent to the return angle
		$(2m_1-2j+1)\pi$ for we are counting unoriented saddle
		connections. Thus, letting $j$ run all the range
		$1,2,\dots,m_1-1$ of possible values, we count each
		configuration twice with exception for the symmetric
		situation when $m_1$ is odd and $j=(m_1+1)/2$. However, in
		this symmetric situation we have extra factor $1/2$
		in~\eqref{eq:loop:restrictive} and our counting
		formula~\eqref{eq:saddle:connections:to:itself:all} follows.
	\end{proof}
	
	\begin{Remark}
		Note that
		formula~\eqref{eq:saddle:connections:to:itself:all}
		suggests the following naive interpretation. Consider a conical
		point with angle $2\pi(m_1+1)$. There are $m_1+1$ ways
		to launch a trajectory in any chosen direction and $m_1-1$
		ways for such trajectory to come back since we do not count
		the trajectories returning at the angle $\pi$. Since we
		count \textit{unoriented} saddle connections we get
		$\frac{(m_1+1)(m_1-1)}{2}$ ways of pairing.
	\end{Remark}
	
	\subsection*{Cylinders having a pair of distinct zeroes on its boundaries.}
	Consider any connected stratum of the form $\cH(m_1,m_2,\dots)$,
	i.e. one which has at least two distinct zeroes, where $m_1,m_2$
	denote their degrees. The situation when $m_1=m_2$
	is not excluded. We assume that $m_1,m_2\ge 1$, i.e. that we have
	true zeroes and not just marked points.

	Consider a configuration consisting of a flat cylinder 
    embedded into our translation surface 
    such that each of the two boundary components of the cylinder is
    represented by a single saddle connection joining a zero to itself.
    We first consider the situation when the two zeroes are distinct. 
    Such surface can be obtained following the construction represented in 
    Figure~\ref{fig:parallelogram:construction} 
    except that instead of identifying the two geodesic boundary components 
    of the surface in the middle picture, we attach to them a flat cylinder.

    By construction the two saddle connections bounding the cylinder are 
    homologous. We assume that there are no other saddle connections homologous 
    to them. 
	
	\begin{Corollary}
		\label{cor:distinct:zeroes:cylinder}
		There exists a universal constant $\ubound^{cyl} > 0$ such that the
		Siegel---Veech constant
		$c^{cyl}_{m_1,m_2}(\cH(m_1,m_2,\dots))$ corresponding to the
		number  of  configurations of saddle connections
		of multiplicity one
		which bound a cylinder with a fixed zero
		of degree $m_1$ on one boundary component of the cylinder
		and a fixed zero of degree $m_2$ on the other boundary
		component of the cylinder satisfies
		\begin{equation}
		\label{eq:cylinder:distinct:zeroes}
		c^{cyl}_{m_1,m_2}(\cH(m_1,m_2,\dots))
		=\cfrac{(m_1+1)(m_2+1)}{\dim_\C\cH(m)-2}
		\cdot\big(1+\varepsilon^{cyl}_{m_1,m_2}(m)\big)
		\end{equation}
		where
		\begin{equation}
		\label{eq:cylinder:distinct:zeroes:bound}
		\max_{m\in\mathbb{Y}_{2g - 2}} |\varepsilon^{cyl}_{m_1,m_2}(m)|
		\le \frac{\ubound^{cyl}}{g}\,.
		\end{equation}
	\end{Corollary}
	In the context of the above corollary the condition of
	``multiplicity one'' means that there are no other saddle
	connections homologous to the two ones on the boundaries of
	the cylinder.
	\begin{proof}
		Let $b':=m_1-1$, let $b'':=m_2-1$.
		In terminology of~\cite{PBC}
		the configurations of saddle connections
		described in Corollary~\ref{cor:distinct:zeroes:cylinder}
		correspond to the ``creation of pair of holes assignment'' applied to
		a fixed pair of zeroes of degrees $b', b''$ on a
		surface in a stratum $\cH(m')$, where
		$m=\{m_1,m_2,\dots,\}$ and $m'$ is obtained from
		$m$ by replacing the first two entries (i.e., the
		entries $m_1,m_2$) by the entries $m_1-1,m_2-1$.
		
		The new partition $m'$ represents the stratum in genus
		$g-1$, so $\dim_{\C{}}\cH(m')=\dim_{\C{}}\cH(m)-2$.
		
		We are in the setting of Problem 1 from section 13.2
		in~\cite{PBC} when all the zeroes are labeled.
		Thus we do not have any symmetries, $|\Gamma|=|\Gamma_-|=1$
		even if $b'=b''$.
		
		Applying formula 13.1 from~\cite{PBC} from
		which we remove all terms containing symbols $o(\cdot)$
		responsible for unlabeling the zeroes we get
		\begin{multline*}
		c^{cyl}_{m_1,m_2}(\cH(m_1,m_2,\dots))
		=\frac{(b'+1)(b''+1)}{\dim_\C\cH(m)-2}
		\cdot\frac{\nu_1\big(\cH(m')\big)}{\nu_1\big(\cH(m)\big)}
		=\frac{m_1\cdot m_2}{\dim_\C\cH(m)-2}
		\cdot\frac{\nu_1\big(\cH(m')\big)}{\nu_1\big(\cH(m)\big)}.
		\end{multline*}
		Applying~\eqref{eq:asymptotic:formula:for:the:volume}
		to the ratio of volumes we get
		$$
		\frac{\nu_1\big(\cH_1(m')\big)}{\nu_1\big(\cH_1(m)\big)}=
		\frac{(m_1+1)(m_2+1)}{m_1\cdot m_2}
		\cdot\frac{1+\varepsilon(m')}{1+\varepsilon(m)}\,.
		$$
		Bounds~\eqref{eq:asymptotic:formula:for:the:volume:bound}
		now imply that
		$$
		\sup_{g\ge 2}
		\ g\cdot
		\max_{\substack{{m\in\mathbb{Y}_{2g - 2}}\\
				m_1,m_2\in m\\
				\cH(m)\text{ is connected}}}
		\left|
		\frac{1+\varepsilon(m')}{1+\varepsilon(m)}
		-1\right|
		=:\ubound^{cyl}<+\infty
		$$
		and~\eqref{eq:cylinder:distinct:zeroes} follows.
	\end{proof}
	
	\subsection*{Cylinders having the same fixed zero on both
		boundary components.}
	%
	%
	
    Consider a configuration consisting of a flat cylinder
    embedded into our translation surface
    with boundary components 
    represented by saddle connections joining the common saddle point to itself.
    We suppose that there are no other saddle connections homologous to the
    two boundary components of the cylinder.

\begin{figure}[hbt!]
%
%

\includegraphics[height=3.5cm,width=3.5cm]{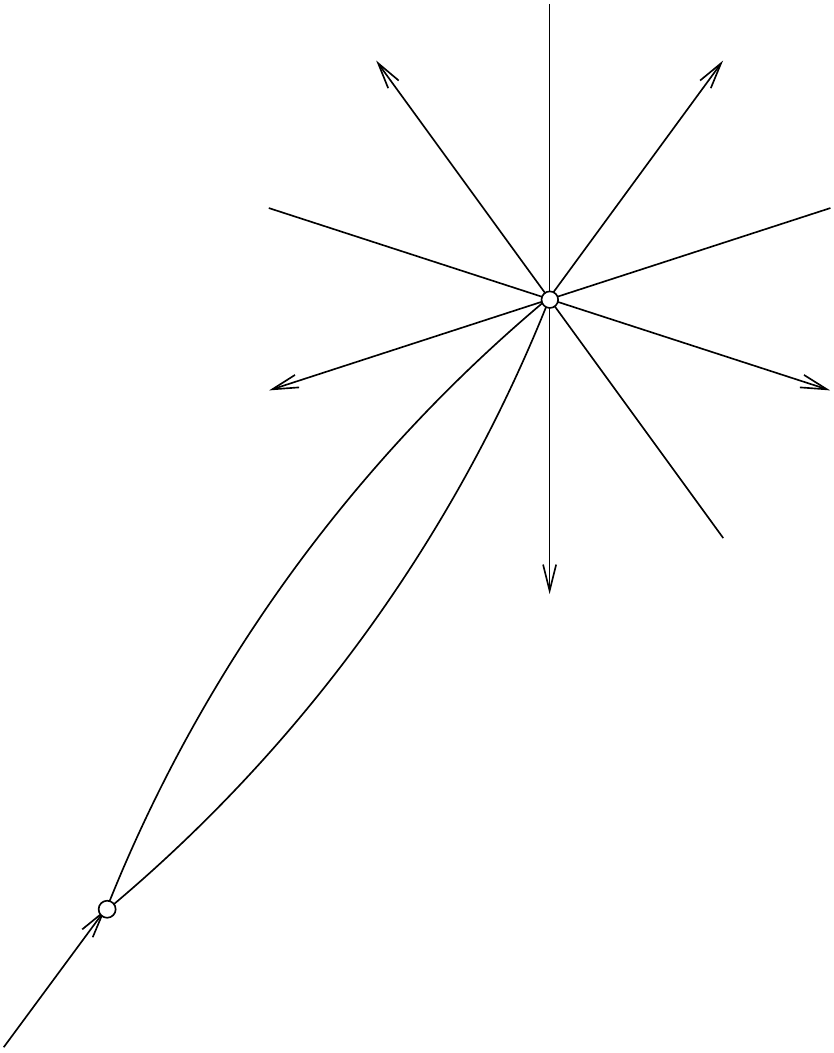} \qquad \includegraphics[height=4.5cm, width = 4.5cm]{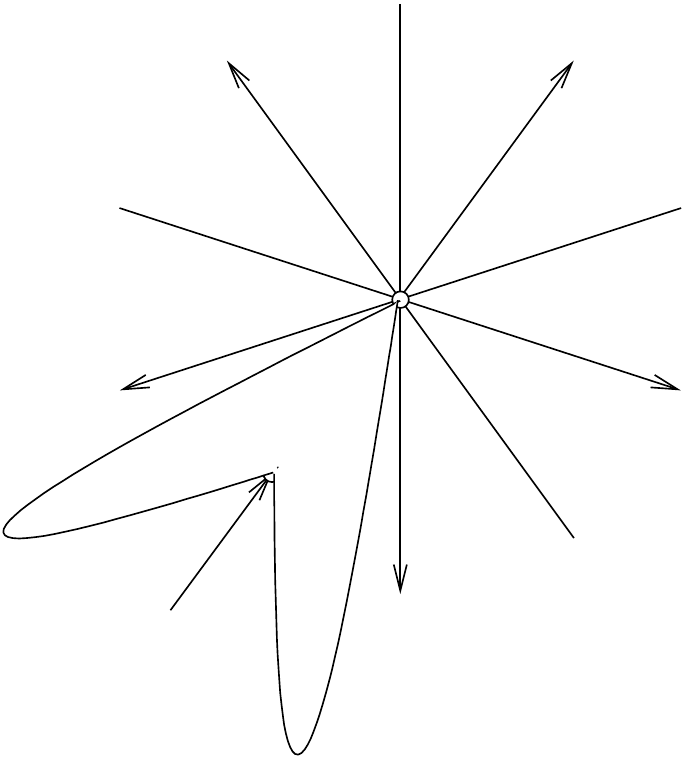} \qquad \includegraphics[height = 4.5cm, width = 4.5cm]{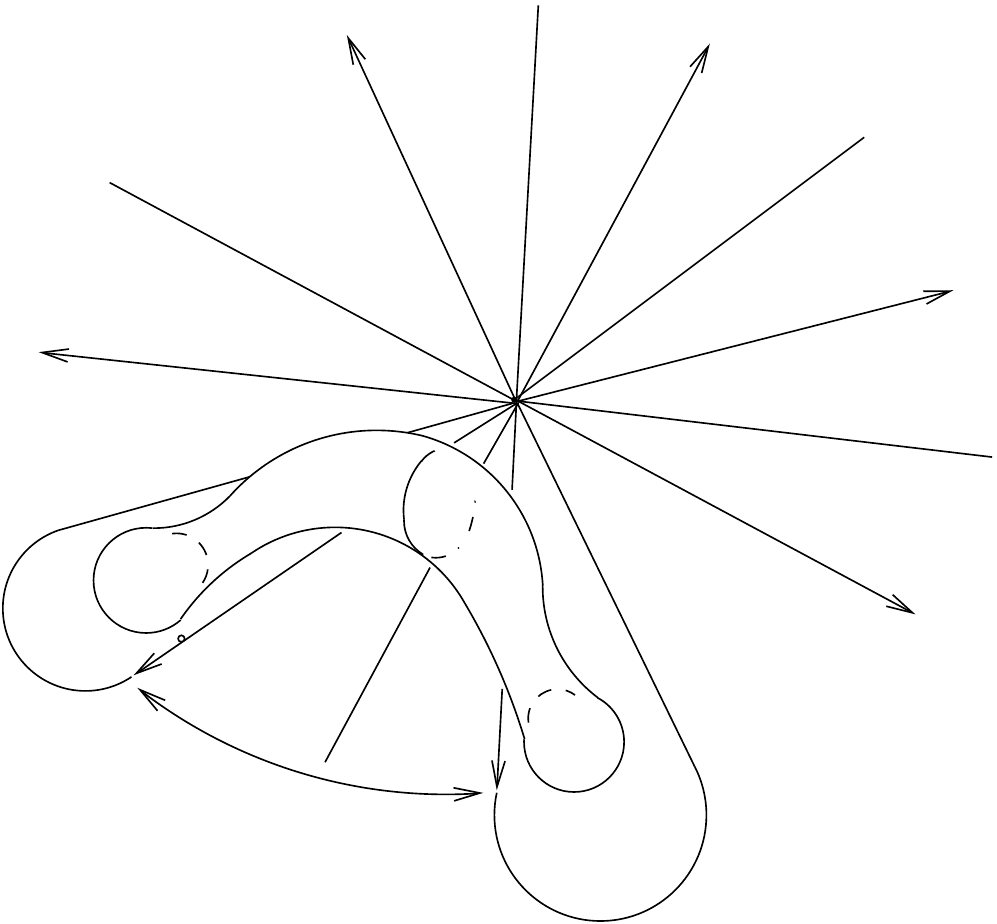}

\begin{picture}(0,0)(0,0)
\put(0,0){\begin{picture}(0,0)(0,0) \put(80,20){$2\pi(a'+1)$}
	\end{picture}}
\end{picture}

\caption{
\label{pic:figure:8:constr} 
A flat cylinder bounded by two saddle connections 
joining the common saddle point to itself.
}
\end{figure}
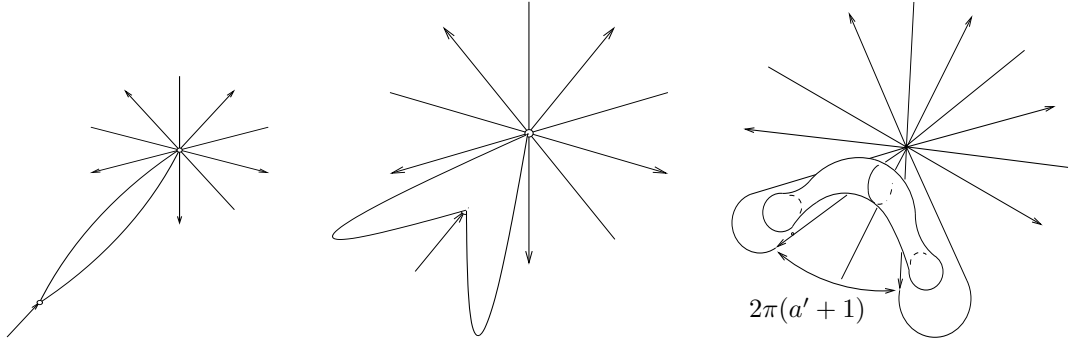

    Figure~\ref{pic:figure:8:constr} (reproduced from Figure~10 in~\cite{PBC})
    describes how to create such a configuration from a translation surface 
    of genus $g-1$. We start by slitting a translation surface of genus $g-1$ along
    a geodesic segment with no saddle points in its interior. In this way we get 
    a surface with boundary as in the left picture. We identify the two 
    endpoints of the slit (as indicated in the middle picture) 
    getting a surface with geodesic boundary in the shape of figure eight.
    Finally we paste a flat cylinder to the two saddle connections forming
    two loops of figure eight and get a translation surface of genus $g$.
    It is easy to see that the construction is invertible.          	

	\begin{Corollary}
		\label{cor:figure:eight:fixed:zero}
		There exists a universal constant $\ubound^{handle} > 0$ such
		that the Siegel---Veech constant
		$c^{handle}_{m_1}(\cH(m_1\dots))$ corresponding to the
		number  of  configurations of saddle connections of
		multiplicity one which bound a cylinder having the same
		fixed zero of degree $m_1$ on both boundary components
		satisfies
		\begin{equation}
		\label{eq:figure:eight:fixed:zero}
		c^{handle}_{m_1}(\cH(m_1,\dots))
		=\cfrac{1}{2}\cdot\cfrac{(m_1+1)(m_1-1)}{\dim_\C\cH(m)-2}
		\cdot\big(1+\varepsilon^{handle}_{m_1}(m)\big)
		\end{equation}
		where
		\begin{equation}
		\label{eq:handle:bound}
		\max_{m\in\mathbb{Y}_{2g - 2}} |\varepsilon^{handle}_{m_1}(m)|
		\le \frac{\ubound^{handle}}{g}\,.
		\end{equation}
	\end{Corollary}
	Note that we do not specify the angles between the pair of
	saddle connections bounding the cylinder.
	
	\begin{proof}
		Note that by geometric reasons
		the common zero located at the boundaries of the
		cylinder has order at least $2$. Thus
		$$
		c^{handle}_{1}(\cH(m))=0\,,
		$$
		which justifies~\eqref{eq:figure:eight:fixed:zero} for
		$m_1=1$. From now on we exclude this trivial case and
		assume that $m_1\ge 2$.
		
		Let $a:=m_1-2$; let $a'+a''=a$ be a partition of $a$ into
		an ordered sum of nonnegative integers. The configurations
		described in Corollary~\ref{cor:figure:eight:fixed:zero}
		correspond to the ``figure eight assignment'' applied to a
		fixed zero of degree $a=m_1-2$ on a surface in the stratum
		$\cH(m')$, where $m=\{m_1,\dots,\}$ and $m'$ is obtained
		from $m$ by replacing the first entry (i.e., the entry
		$m_1$) by the entry $m_1-2$.
		
		The new partition $m'$ represents the stratum in genus
		$g-1$, so $\dim_{\C{}}\cH(m')=\dim_{\C{}}\cH(m)-2$. The
		partition $a'+a''=a$ encodes the angles between saddle
		connections at the zero. In the setting of Problem 1 from
		section 13.2 in~\cite{PBC}. we have
		$|\Gamma|=1$ and
		$$
		|\Gamma_-|=
		\begin{cases}
		2&\text{if }j=a'=a''\\
		1&\text{otherwise}
		\end{cases}
		$$
		
		Applying formula on page 135 of~\cite{PBC}
		for any fixed partition $a'+a''=a=m_1-2$, where
		the combinatorial factor is computed on top of page 140
		in~\cite{PBC}, we get the value
		$$
		\frac{1}{|\Gamma_-|}\cdot\frac{a+1}{\dim_\C\cH(m)-2}
		\cdot\frac{\nu_1\big(\cH(m')\big)}{\nu_1\big(\cH(m)\big)}
		$$
		for the Siegel--Veech constants for the more restricted
		count when the pair $(a',a'')$ is fixed. Pairs of
		partitions $(a',a'')$ and $(a'',a')$ of $a=m-2$ represent
		the same configurations in our setting. Thus, the sum over
		all unordered partitions $a'+a''=m_1-2$, where
		$a'=0,1,\dots,m_1-2$, gives
		$$
		c^{cyl}_{m_1}(\cH(m_1,\dots))
		=
		\frac{m_1-1}{2}\cdot\frac{m_1-1}{\dim_\C\cH(m)-2}
		\cdot\frac{\nu_1\big(\cH(m')\big)}{\nu_1\big(\cH(m)\big)}.
		$$
		Applying~\eqref{eq:asymptotic:formula:for:the:volume}
		to the ratio of volumes we get
		$$
		\frac{\nu_1\big(\cH_1(m')\big)}{
			\nu_1\big(\cH_1(m)\big)}=
		\frac{m_1+1}{m_1-1}
		\cdot\frac{1+\varepsilon(m')}{1+\varepsilon(m)}\,.
		$$
		Bounds~\eqref{eq:asymptotic:formula:for:the:volume:bound}
		now imply that
		\begin{equation*}
		\sup_{g\ge 2}
		\ g\cdot
		\max_{\substack{{m\in\mathbb{Y}_{2g - 2}}\\
				2\le m_1\in m\\
				\cH(m)\text{ is connected}}}
		\left|
		\frac{1+\varepsilon(m')}{1+\varepsilon(m)}
		-1\right|
		=:\ubound^{handle}<+\infty
		\end{equation*}
		and~\eqref{eq:handle:bound} follows.
	\end{proof}
	
	\subsection*{Count of all cylinders of multiplicity one.}
	Combining results of Corollaries~\ref{cor:distinct:zeroes:cylinder}
	and~\ref{cor:figure:eight:fixed:zero} we get the following
	result.
	
	\begin{Theorem}
		\label{th:SV:for:all:multiplicity:1:cylinders}
		The  Siegel---Veech constant $c^{(1)}_{cyl}(\cH(m))$ for the
		number of all cylinders of multiplicity one on a surface of
		area one in a connected stratum $\cH(m)$ has the form
		\begin{equation}
		\label{eq:all:mult:1:cylinders}
		c^{(1)}_{cyl}(\cH(m))
		=\frac{1}{2}
		\left((\dim_{\C}\cH(m)-2)
		-\frac{1}{\dim_{\C}\cH(m)-2}\right)
		\cdot\left(1+\varepsilon_{cyl}(m)\right)
		\end{equation}
		where
		\begin{equation}
		\label{eq:all:mult:1:cylinders:bound}
		\max_{m\in\mathbb{Y}_{2g - 2}} |\varepsilon_{cyl}(m)|
		\le \frac{\max(\ubound^{cyl},\ubound^{handle})}{g}\,.
		\end{equation}
		and the universal constants $\ubound^{cyl},\ubound^{handle}$
		are defined in equations~\eqref{eq:cylinder:distinct:zeroes:bound}
		and~\eqref{eq:handle:bound}.
		
		Under the same assumptions as above, the  Siegel---Veech
		constant $c^{(1)}_{area}(\cH(m))$ corresponding to the
		weighted count of cylinders of multiplicity one, with the
		area of the cylinder taken as the weight, has the
		following form
		\begin{equation}
		\label{eq:mult:1:carea}
		c^{(1)}_{area}(\cH(m))
		=\frac{c^{(1)}_{cyl}(\cH(m))}{\dim_\C{}\cH(m)-1}
		=\frac{1}{2}
		\cdot\big(1+\varepsilon_{area}(m)\big)\,,
		\end{equation}
		where
		\begin{equation}
		\label{eq:mult:1:carea:bound}
		\max_{m\in\mathbb{Y}_{2g - 2}} |\varepsilon_{area}(m)|
		\le \frac{B}{g}\,,
		\end{equation}
		and $B$ is a universal constant.
	\end{Theorem}
	\begin{proof}
		We are counting maximal cylinders of multiplicity one
		filled with closed flat geodesics, i.e. we assume that
		there are no saddle connections homologous to the waist
		curve of the cylinder outside of the cylinder. To count all
		such cylinders we have to sum up the Siegel--Veech
		constants $c^{cyl}_{m_i,m_j}(\cH(m_1,m_2,\dots,m_n))$ for
		all pairs $1\le i< j\le n$ and the Siegel--Veech constants
		$c^{handle}_{m_i}(\cH(m_1,\dots))$ for all $i$ in the range
		$1\le i \le n$. Representing the sum over all unordered
		distinct pairs $(i,j)$ as half of the sum of ordered
		distinct pairs and
		combining~\eqref{eq:cylinder:distinct:zeroes}
		and~\eqref{eq:figure:eight:fixed:zero} we get
		\begin{multline*}
		c^{(1)}_{cyl}(\cH(m_1,\dots,m_n))=\frac{1}{2}\cdot\frac{1}{\dim_\C\cH(m)-2}
		\cdot\Bigg(
		\sum_{\substack{ i,j=1\\i\neq j}}^n
		(m_i+1)(m_j+1)
		\cdot\big(1+\varepsilon^{cyl}_{m_i,m_j}(m)\big)\\
		+
		\sum_{i=1}^n
		\big((m_i+1)^2 - 2(m_i+1)\big)
		\cdot\big(1+\varepsilon^{handle}_{m_i}(m)\big)
		\Bigg)\,.
		\end{multline*}
		Bounds~\eqref{eq:cylinder:distinct:zeroes:bound}
		and~\eqref{eq:handle:bound} for
		$\varepsilon^{cyl}_{m_i,m_j}(m)$ and
		$\varepsilon^{handle}_{m_i}(m)$ imply that there exists
		$\varepsilon_{cyl}(m)$ satisfying
		bounds~\eqref{eq:all:mult:1:cylinders:bound} such that the
		above expression takes the form
		$$
		\frac{1}{2}\cdot\frac{1}{\dim_\C\cH(m)-2}
		\cdot
		\left(
		\left(\sum_{i=1}^n (m_i+1)\right)^2 -
		2\sum_{i=1}^n (m_i+1)
		\right)
		\cdot\big(1+\varepsilon_{cyl}(m)\big)
		$$
		Note that
		$$
		\dim_\C\cH(m_1,\dots,m_n)-1=
		\sum_{i=1}^n (m_i+1)=(2g-2)+n=|m|+\ell(m)\,,
		$$
		where $|m|$ and $\ell(m)$ are the size and the length
		of the partition $m$ respectively. Hence, we can
		represent the latter expression for
		$c^{(1)}_{cyl}(\cH(m))$
		as
		$$
		c^{(1)}_{cyl}(\cH(m))=
		\frac{1}{2}\cdot\frac{(\dim_\C\cH(m)-2)^2-1}{\dim_\C\cH(m)-2}
		\cdot\big(1+\varepsilon_{cyl}(m)\big)\,,
		$$
		where $\varepsilon_{cyl}(m)$ satisfies
		bounds~\eqref{eq:all:mult:1:cylinders:bound}. This
		completes the proof of the first part of the statement.
		
		By the formula of Vorobets (see~(2.16)
		in~\cite{ERG} or the original
		paper~\cite{PGTS}),  the Siegel---Veech constant
		$c_{area}(\cH(m))$ is expressed in terms of the
		Siegel---Veech constants of configurations of homologous
		closed saddle connections as follows
		\begin{equation}
		\label{eq:carea:formula}
		c_{area}(\cH(m))=
		\cfrac{1}{\dim_\C{}\cH(m)-1}\cdot\sum_{q=1}^{g-1} q\cdot
		\sum_{\substack{Configurations\ \cC\\ containing\ q\ cylinders}}
		c_\cC(\cH(m))\,.
		\end{equation}
		The  Siegel---Veech constant
		$c^{(1)}_{area}(\cH(m))$ corresponding to the weighted
		count of cylinders of multiplicity one represents the
		term of the above sum corresponding to $q=1$, namely,
		$$
		c^{(1)}_{area}(\cH(m)):=
		\frac{c^{(1)}_{cyl}(\cH(m))}{\dim_\C{}\cH(m)-1}\,.
		$$
		Expression~\eqref{eq:all:mult:1:cylinders} for
		$c^{(1)}_{cyl}(\cH(m))$ and
		bound~\eqref{eq:all:mult:1:cylinders:bound} for
		$\varepsilon_{cyl}(m)$ imply existence of a universal
		constant $B$ such that the ratio
		$\cfrac{c^{(1)}_{cyl}(\cH(m))}{\dim_\C{}\cH(m)-1}$ can be
		represented in the form~\eqref{eq:mult:1:carea} with
		$\varepsilon_{area}(m)$ satisfying
		bound~\eqref{eq:mult:1:carea:bound}.
	\end{proof}
	
	\subsection*{Arithmetic nature of Siegel--Veech constants.}
	By result of Eskin and Okounkov~\cite{ANBCTV} the
	Masur--Veech volume $\cH(m)$ of any stratum in genus $g$
	has the form of a rational number multiplied by $\pi^{2g}$.
	The Siegel--Veech constants $c^{sc}_{m_1,m_2}$ and
	$c^{sc;2}_{m_1,m_2}$ responsible for the count of saddle
	connections joining \textit{distinct} zeroes are expressed
	as a rational factor times the ratio of volumes
	$\nu_1\big(\cH(m')\big)/\nu_1\big(\cH(m)\big)$ of strata in the same genus,
	so these Siegel--Veech constants are rational numbers.
	
	The Siegel--Veech constants $c^{loops}_{m_1},\,
	c^{cyl}_{m_1,m_2},\, c^{handle}_{m_1},\, c^{(1)}_{cyl},\,
	c^{(1)}_{area},\, c_{area}$ responsible for the count of
	saddle connections going from a zero to itself are also
	expressed as a rational factor times the ratio
	$\nu_1\big(\cH(m')\big)/\nu_1\big(\cH(m)\big)$, but this time the stratum
	$\nu_1\big(\cH(m')\big)$ corresponds to genus $g-1$ while the stratum
	$\nu_1\big(\cH(m)\big)$ corresponds to genus $g$. Thus these latter
	Siegel--Veech constants have the form of a rational number
	divided by $\pi^2$.
	
	\subsection*{Final remark.}
	It was conjectured in~\cite{VSDCLG} that the
	Siegel--Veech constant $c_{area}$ tends to
	$\frac{1}{2}$ uniformly for all nonhyperelliptic connected
	components of all strata as genus tends to infinity:
	\begin{equation}
	\label{eq:conj:carea}
	\lim_{g\to\infty} c_{area}(\cH^{comp}(m)) = \frac{1}{2}\,.
	\end{equation}
	Theorem~\ref{th:SV:for:all:multiplicity:1:cylinders}
	proves the uniform asymptotic lower bound for
	all connected strata $\cH(m)$:
	$$
	\liminf_{g\to\infty} c_{area}(\cH(m)) \ge \frac{1}{2}
	$$
	and shows that the conjecture~\eqref{eq:conj:carea} for
	connected strata is equivalent to conjectural vanishing of
	the contribution of configurations with $q\ge 2$ cylinders
	in formula~\eqref{eq:carea:formula} uniformly for all
	connected strata in large genera.
	The conjectural asymptotic~\eqref{eq:conj:carea} was recently proved first by A.~Aggarwal in~\cite{LGAC} and then by
    D.~Chen, M.~M\"oller, A.~Sauvaget and D.~Zagier in~\cite{VSCC} by completely different methods.
\smallskip

\noindent\textbf{Acknowledgements.} We are grateful to 
the anonymous referee for helpful
suggestions which allowed to improve the presentation.


\begin{thebibliography}{}
	
\bibitem{LGAC} \label{LGAC} A. Aggarwal, Large Genus Asymptotics for Siegel--Veech Constants, To appear in \emph{Geom. Funct. Anal.}, preprint, arxiv:1810.05227. 

\bibitem{TRV} \label{TRV} J. E. Andersen, G. Borot, S. Charbonnier, V. Delecroix, A. Giacchetto, D. Lewanski, and C. Wheeler, Topological Recursion for Masur-Veech Volumes, preprint, arxiv:1905.10352.

	\bibitem{QLGL} \label{QLGL} D. Chen, M. M\"{o}ller, and D. Zagier, Quasimodularity and Large Genus Limits of Siegel-Veech Constants, \emph{J. Amer. Math. Soc.} \textbf{31}, 1059--1163, 2018.

\bibitem{VSCC}\label{VSCC} D. Chen, M. M\"oller, A. Sauvaget, and D. Zagier, Masur-Veech Volumes and Intersection Theory on Moduli Spaces of Abelian Differentials, preprint, arXiv:1901.01785.

	
	\bibitem{SSCT} \label{SSCT} V. Delecroix, E. Goujard, P. Zograf, and A. Zorich, Contribution of One-Cylinder Square-Tiled Surfaces to Masur-Veech Volumes, With an Appendix by P. Engel, To appear in \emph{Asterisque}, preprint, arXiv:1903.10904.
	
	
	\bibitem{VFGIN} \label{VFGIN} V. Delecroix, E. Goujard, P. Zograf, and A. Zorich, Masur-Veech Volumes, Frequencies of Simple Closed Geodesics, and Intersection Numbers of Moduli Spaces of Curves, preprint, arxiv:1908.08611.
	
	
	\bibitem{ERG} \label{ERG} A. Eskin, M. Kontsevich, and A. Zorich, Sum of Lyapunov Exponents of the Hodge Bundle With Respect to the Teichm\"{u}ller Geodesic Flow, \textit{Publ. Math. IHES} \textbf{120}, 207--333, 2014.
	
	\bibitem{AFS} \label{AFS} A. Eskin and H. Masur, Asymptotic Formulas on Flat Surfaces, \emph{Ergod. Th. Dynam. Sys.} \textbf{21}, 443--478, 2001.
	
	\bibitem{PBC} \label{PBC} A. Eskin, H. Masur, and A. Zorich, Moduli Spaces of Abelian Differentials: The Principal Boundary, Counting Problems, and the Siegel-Veech Constants, \textit{Publ. Math. IHES} \textbf{97}, 61--179, 2003. 	
	
	
	\bibitem{ANBCTV} \label{ANBCTV} A. Eskin and A. Okounkov, Asymptotics of Numbers of Branched Coverings of a Torus and Volumes of Moduli Spaces of Holomorphic Differentials, \textit{Invent. Math.} \textbf{145}, 59--103, 2001.
	
	\bibitem{VSDCLG} \label{VSDCLG} A. Eskin and A. Zorich, Volumes of Strata of Abelian Differentials and Siegel-Veech Constants in Large Genera, \textit{Arnold Math. J.} \textbf{1}, 481--488, 2015.
	
	\bibitem{CCMS} \label{CCMS} M. Kontsevich and A. Zorich, Connected Components of the Moduli Spaces of Abelian Differentials With Prescribed Singularities, \textit{Invent. Math.} \textbf{153}, 631--678, 2003.
	
	\bibitem{ETMF} \label{ETMF} H. Masur, Interval Exchange Transformations and Measured Foliations, \textit{Ann. Math.} \textbf{115}, 169--200, 1982. 
	
	\bibitem{RBF} \label{RBF} H. Masur and S. Tabachnikov, Rational Billiards and Flat Structures, In: \emph{Handbook of Dynamical Systems} (B. Hasselblatt and A. Katok ed.), Elsevier Science B.V., 1015--1089, 2002.
	
	\bibitem{GVRHSLG} \label{GVRHSLG} M. Mirzakhani, Growth of Weil-Petersson Volumes and Random Hyperbolic Surfaces of Large Genus, \emph{J. Differential Geom.} \textbf{94}, 267--300, 2013.
	
	\bibitem{LGAIMSC} \label{LGAIMSC} M. Mirzakhani and P. Zograf, Towards Large Genus Asymptotics of Intersection Numbers on Moduli Spaces of Curves, \emph{Geom. Funct. Anal.} \textbf{25}, 1258--1289, 2015.
	
	
	\bibitem{CCSD} \label{CCSD} A. Sauvaget, Cohomology Classes of Strata of Differentials, \emph{Geom. Topol.} \textbf{23}, 1085--1171, 2019. 
	
	\bibitem{VSI} \label{VSI}	A. Sauvaget, Volumes and Siegel-Veech Constants of $\mathcal{H} (2g - 2)$ and Hodge Integrals, \emph{Geom. Funct. Anal.} \textbf{28}, 1756--1779, 2018.
	
	\bibitem{LGAEV} \label{LGAEV} A. Sauvaget, The Large Genus Asymptotic Expansion of Masur-Veech Volumes, preprint, arxiv:1903.04454.
	
	\bibitem{MTSIEM} \label{MTSIEM} W. A. Veech, Gauss Measures for Transformations on the Space of Interval Exchange Maps, \textit{Ann. Math.} \textbf{115}, 201--242, 1982.
	
	
	\bibitem{PGTS} \label{PGTS}	Ya. Vorobets, \textit{Periodic Geodesics of Translation Surfaces}, 2003, in S. Kolyada, Yu. I. Manin and T. Ward (eds.) Algebraic and Topological Dynamics, Contemporary Math., vol. 385, pp. 205--258, Amer. Math. Soc., Providence, 2005.
	
	
	\bibitem{TSOC} \label{TSOC} A. Wright, Translation Surfaces and Their Orbit Closures: An Introduction for a Broad Audience, \emph{EMS Surv. Math. Sci.} \textbf{2}, 63--108, 2015.
	
	\bibitem{LGAV} \label{LGAV} P. Zograf, On the Large Genus Asymptotics of Weil-Petersson Volumes, preprint, arxiv:0812.0544.
	
	\bibitem{FS} \label{FS} A. Zorich, Flat Surfaces, In: \emph{Frontiers in Number Theory, Physics, and Geometry} (P. E. Cartier, B. Julia, P. Moussa, and P. Vanhove ed.), Springer, Berlin, 437--583, 2006
	
	\bibitem{SVMS} \label{SVMS} A. Zorich, Square Tiled Surfaces and Teichm\"{u}ller Volumes of the Moduli Spaces of Abelian Differentials, In: \emph{Rigidity in Dynamics and Geometry} (M. Burger and A. Iozzi), Springer, Berlin, 459--471, 2002.
\end{thebibliography}
\end{document}